\documentclass[10pt]{article}

\usepackage[utf8]{inputenc}
\usepackage[T1]{fontenc}
\usepackage[english]{babel}

\usepackage[lmargin = 1.5 in, rmargin = 1 in, tmargin = 1 in, bmargin = 1 in]{geometry}

\usepackage{fancyhdr}
\author{Charles Devlin VI}
\usepackage[stretch=20]{microtype}
\usepackage{amsmath}
\usepackage{amsthm}
\usepackage{amssymb} 
\usepackage{mathtools}

\usepackage{color}

\usepackage{stackrel}

\usepackage[hyphens]{url}

\usepackage[colorlinks]{hyperref}
\hypersetup{
  linkcolor=blue
}%

\usepackage{enumitem}

\usepackage{bm}

\allowdisplaybreaks 

\renewcommand{\AA}{\mathbb{A}}
\newcommand{\CC}{\mathbb{C}}
\newcommand{\DD}{\mathbb{D}}

\newcommand{\NN}{\mathbb{N}}

\newcommand{\QQ}{\mathbb{Q}}
\newcommand{\RR}{\mathbb{R}}
\newcommand{\ZZ}{\mathbb{Z}}

\newcommand{\fa}{\mathfrak{a}}
\newcommand{\fh}{\mathfrak{h}}

\newcommand{\cD}{\mathcal{D}}
\newcommand{\cE}{\mathcal{E}}
\newcommand{\cH}{\mathcal{H}}
\newcommand{\cL}{\mathcal{L}}
\newcommand{\cR}{\mathcal{R}}
\newcommand{\cS}{\mathcal{S}}
\newcommand{\cU}{\mathcal{U}}

\DeclareMathOperator{\id}{id}
\DeclareMathOperator{\dist}{dist}
\DeclareMathOperator{\Len}{Len}

\newcommand{\hlint}[2]{\left({#1}, {#2}\right]}
\newcommand{\hrint}[2]{\left[{#1}, {#2}\right)}

\newcommand{\D}{\mathrm{d}}
\newcommand{\p}{\partial}

\newcommand{\w}{\wedge}

\renewcommand{\i}{{\mathrm{i}}}

\newtheorem{thm}{Theorem}[section]
\newtheorem{prop}[thm]{Proposition}

\newtheorem{lemma}[thm]{Lemma}
\newtheorem{defn}[thm]{Definition}

\usepackage{tikz}
\usepackage{graphicx}

\makeatletter
\renewcommand{\fnum@figure}{Fig. \thefigure}
\makeatother
 \usepackage[font=small,labelfont=bf,labelsep=space]{caption}

\newcommand{\ratios}[1][\epsilon]{\frac{r \fa_{{#1}}^{-1}}{r^{\xi Q} \fa_{{#1}/r}^{-1}}}
\newcommand{\invratios}[1][\epsilon]{\frac{r^{\xi Q} \fa_{{#1}/r}^{-1}}{r \fa_{{#1}}^{-1}}}

\NewDocumentCommand \locLFPP { o o } {%
  \IfNoValueTF{#2}
    {%
      \IfNoValueTF{#1}
        {%
          \fa_{\epsilon}^{-1} \hat{D}_h^{\epsilon}%
        }
        {%
          \fa_{#1}^{-1} \hat{D}_h^{#1}%
        }
    }
    {%
      \fa_{#1}^{-1} \hat{D}_{#2}^{#1}%
    }
}

\NewDocumentCommand \LFPP { o o } {%
  \IfNoValueTF{#2}
    {%
      \IfNoValueTF{#1}
        {%
          \fa_{\epsilon}^{-1} D_h^{\epsilon}%
        }
        {%
          \fa_{#1}^{-1} D_h^{#1}%
        }
    }
    {%
      \fa_{#1}^{-1} D_{#2}^{#1}%
    }
}

\newcommand{\DerivativeBound}{\tau}

\newcommand{\hphi}[1][\phi]{
  \mathchoice
    {\hyperlink{hphi}{h^{#1}}} 
    {\hyperlink{hphi}{h^{#1}}} 
    {\hyperlink{hphi}{\scriptstyle h^{#1}}} 
    {\hyperlink{hphi}{\scriptscriptstyle h^{#1}}} 
}
\NewDocumentCommand{\confmaps}{O{\DerivativeBound} O{V} O{U}}{%
  \mathchoice
    {\hyperlink{confmaps}{\Lambda_{#1}(#2, #3)}} 
    {\hyperlink{confmaps}{\Lambda_{#1}(#2, #3)}} 
    {\hyperlink{confmaps}{\scriptstyle \Lambda_{#1}(#2, #3)}} 
    {\hyperlink{confmaps}{\scriptscriptstyle \Lambda_{#1}(#2, #3)}} 
}
\NewDocumentCommand{\diagonal}{O{\rho} O{W}}{%
  \mathchoice
    {\hyperlink{Delta}{\Delta_{#1}(#2)}} 
    {\hyperlink{Delta}{\Delta_{#1}(#2)}} 
    {\hyperlink{Delta}{\scriptstyle \Delta_{#1}(#2)}} 
    {\hyperlink{Delta}{\scriptscriptstyle \Delta_{#1}(#2)}} 
}

\DeclareMathOperator{\Prob}{\bm{\mathrm{P}}}
\DeclareMathOperator{\E}{\bm{\mathrm{E}}}

\newcommand{\frepsilon}{\epsilon \log \epsilon^{-1}}

\title{The coordinate change formula for the Liouville quantum gravity metric holds for all conformal maps simultaneously}
\date{March 4, 2026}

\usepackage[style=alphabetic,maxnames=50,maxalphanames=50,sorting=nyt,safeinputenc]{biblatex}
\DeclareNameAlias{author}{last-first}

\begin{document}
  \maketitle

  \abstract{
    \textit{Liouville quantum gravity (LQG)} is, heuristically, a theory of random Riemannian geometry with Riemannian metric tensor $e^{\gamma h} (\D x^2 + \D y^2)$, where $h$ is a variant of the Gaussian free field and $\gamma > 0$ is a parameter.
    If $U \subset \CC$ is an open set, $\phi \colon U \to \phi(U)$ is a conformal map, and $h^{\phi} = h \circ \phi^{-1} + Q \log|(\phi^{-1})'|$ (where $Q = Q(\gamma)$ is a parameter), then the LQG surface on $U$ defined with field $h$ is equivalent to the LQG surface on $\phi(U)$ with field $h^{\phi}$.
    This equivalence is meant in the sense that the area measures and distance functions on these surfaces are almost surely equivalent.
    It is known for the area measure that, in fact, this equivalence holds almost surely for all conformal maps $\phi$ simultaneously (Sheffield-Wang 2016).
    We prove the corresponding result for the distance function.
    This makes precise the frequently used heuristic definition that a \textit{quantum surface} is a random equivalence class of domains equipped with the LQG area measure and LQG distance function.
  }

  \tableofcontents

  \vspace{12pt} \noindent \textbf{Acknowledgements} We thank Ewain Gwynne for helpful discussions.
    The author was partially supported by NSF grant DMS-2245832.

  \section{Introduction}
    A \textit{Liouville quantum gravity} (LQG) surface is a random $2$-dimensional Riemannian manifold with Riemannian metric tensor $e^{\gamma h} (\D x^2 + \D y^2)$, where $\gamma \in (0, 2)$ is a parameter, $h$ is a Gaussian free field (GFF), and $\D x^2 + \D y^2$ is the Euclidean metric tensor.
    This definition does not make rigorous sense because $h$ is not a function.
    The rigorous definition of LQG does not directly try to define a probability measure on the set of Riemannian metrics, instead defining probability distributions on the key observables of a Riemannian manifold: a volume form (area measure) and a distance function (metric).
    We will briefly discuss the rigorous definition below.
    For more details, see \cite{GwynneLQGSurvey}, \cite{ScottICM}, and \cite{BerestyckiPowell}.

    Let $p_t(z) \coloneqq \frac{1}{2 \pi t} e^{-|z|^2/2t}$ be the heat kernel.
    Given a GFF-type distribution $h$, define an area measure as the almost sure weak limit \cite{RhodesVargasGMC} \cite{BerestyckiGMC}
    \begin{align}
      \mu \
      &\coloneqq \ \lim_{\epsilon \to 0} \epsilon^{\gamma^2/2} e^{\gamma h_{\epsilon}^{*}(z)} \, \D z,
      \label{eq:LQGAreaMeasure}
    \end{align}
    where $\D z$ is Lebesgue measure on $\CC$ and
    \begin{align*}
      h_{\epsilon}^{*}(z) \
      &\coloneqq \ \left(h \ast p_{\epsilon^2/2} \right)(z) \
      = \ \int\limits_{\CC} h(w) p_{\epsilon^2/2}(z - w) \, \D w,
    \end{align*}
    where the integral should be interpreted in the sense of distributional pairing.
    It is worth noting that the heat kernel mollification $h_{\epsilon}^{*}(z)$ in \eqref{eq:LQGAreaMeasure} can be replaced with other continuous regularizations, such as the circle average process.
    We state the definition in terms of the heat kernel mollification as it is the smooth approximation of the GFF used in the definition of the LQG metric below.

    In analogy with Riemannian geometry, the distance between points on an LQG surface should be the infimum over lengths of paths, where lengths are weighted by $e^{\xi h}$.
    Here, $\xi = \xi(\gamma) \coloneqq \gamma/d_{\gamma}$ where $d_{\gamma}$ is the Hausdorff dimension of the LQG metric\footnote{The precise definition of $d_{\gamma}$ is not as the Hausdorff dimension of the LQG metric. It is defined in \cite{FractalDimensionOfLQG}, and is shown to agree with the Hausdorff dimension of the LQG metric in \cite{KPZFormulasForLQG}.}.
    Like with the area measure \eqref{eq:LQGAreaMeasure}, we make this definition precise by mollifying and taking a limit.
    For $\epsilon > 0$, define the $\epsilon$\textit{-Liouville first passage percolation} (LFPP) metric with parameter $\xi$ by
    \begin{align}
      D_h^{\epsilon}\left(z, w \right) \
      &\coloneqq \ \inf_{P \colon z \to w} \int\limits_{0}^{1} e^{\xi h_{\epsilon}^{*}(P(t))} |P'(t)| \, \D t,
      \label{eq:LFPPDefn}
    \end{align}
    where the infimum is over all piecewise continuously differentiable paths $P \colon [0, 1] \to \CC$ from $z$ to $w$.
    Define the normalization constants
    \begin{align*}
      \fa_{\epsilon} \
      &\coloneqq \ \text{median of } \inf\left\{\int_{0}^{1} e^{\xi h_{\epsilon}^{*}(P(t))} |P'(t)| \, \D t \right\},
    \end{align*}
    where the infimum is over all piecewise continuously differentiable paths $P \colon [0, 1] \to [0, 1]^2$ with $P(0) \in \{0\} \times [0,1]$ and $P(1) \in \{1\} \times [0,1]$.
    While the exact value of $\fa_{\epsilon}$ is unknown, it is known that \cite[Proposition 1.1]{TightnessOfSupercriticalLFPP}
    \begin{align*}
      \fa_{\epsilon} \
      &= \ \epsilon^{1 - \xi Q + o_{\epsilon}(1)}, \quad \epsilon \to 0
    \end{align*}
    for a nonexplicit exponent $Q = Q(\xi) > 0$.
    The $\gamma$-LQG metric is then defined by
    \begin{align}
      D_h(z, w) \
      &\coloneqq \ \lim_{\epsilon \to 0} \LFPP[\epsilon][h]\left(z, w \right).
      \label{eq:LQGMetricDefn}
    \end{align}
    This definition makes sense for arbitrary $\xi > 0$, but there is a critical parameter $\xi_{\text{crit}} \approx 0.41$ (see \cite{TightnessOfSupercriticalLFPP} and \cite{ExistenceAndUniqueness}) where a phase transition occurs.
    For $\xi$ in the \textit{subcritical} and \textit{critical regimes}, i.e. $0 < \xi \leq \xi_{\text{crit}}$, there is a corresponding $\gamma \in \hlint{0}{2}$ such that $\xi = \gamma/d_{\gamma}$.
    In the \textit{supercricial regime} where $\xi > \xi_{\text{crit}}$, the corresponding $\gamma$ is complex.

    In the subcritical regime, the limit \eqref{eq:LQGMetricDefn} is known to exist almost surely with respect to the topology of uniform convergence on compact subsets of $\CC \times \CC$ \cite{AlmostSureConvergence}.
    The proof of almost sure convergence builds on the results of \cite{TightnessOfSubcriticalLFPP}, \cite{ExistenceAndUniqueness}, \cite{WeakLQGMetrics}, \cite{LocalMetricsOfTheGFF}, and \cite{ConfluenceOfGeodesicsSubcriticalLQG}, which together prove convergence in probability, and \cite{UpToConstants}, which gives a more quantitative comparison of LFPP and the LQG metric.
    For the critical and supercritical phases, only convergence in probability is known, and a weaker topology on function space introduced in \cite{BeerTopology} must be used instead of local uniform convergence, see \cite{UniquenessOfSupercriticalLQGMetrics}, \cite{TightnessOfSupercriticalLFPP}, and \cite{WeakSupercriticalLQGMetrics}.

    \subsection{Axioms of LQG metrics}
      An important question is whether there could be multiple different metrics which satisfy the heuristic definition of the LQG metric as the ``Riemannian distance function associated to the Riemannian metric tensor $e^{\gamma h} (\D x^2 + \D y^2)$.''
      This question is answered in \cite{ExistenceAndUniqueness} and \cite{UniquenessOfSupercriticalLQGMetrics} by stating a list of axioms which any reasonable notion of an LQG metric should satisfy, then showing that these axioms uniquely characterize the LQG metric up to multiplicative constant.
      Since the limit \eqref{eq:LQGMetricDefn} satisfies these axioms \cite{WeakLQGMetrics} \cite{ExistenceAndUniqueness} \cite{ConformalCovariance}, it makes sense to call \eqref{eq:LQGMetricDefn} \textit{the} LQG metric.

      Let $(X, D)$ be a metric space.
      A \textit{curve} in $X$ is a continuous function $P \colon [a,b] \to X$, and its $D$-length is
      \begin{align}
        \Len(P; D) \
        &\coloneqq \ \sup_{T} \sum_{i=1}^{\# T} D\left(P(t_i), P(t_{i-1}) \right),
        \label{eq:PathLengthDefinition}
      \end{align}
      where the supremum is over all finite partitions $T = \{a = t_0 < t_1 < \cdots < t_{\# T} = b\}$ of $[a, b]$.
      If $Y \subset X$, the \textit{internal metric} of $D$ on $Y$ is defined by
      \begin{align*}
        D(x, y; Y) \
        &\coloneqq \ \inf_{P} \Len(P; D), \quad \forall x, y \in Y,
      \end{align*}
      where the infimum is over all curves in $Y$ from $x$ to $y$.
      Then $D(\cdot, \cdot; Y)$ is a metric on $Y$, except it may take the value $\infty$.

      We say $(X, D)$ is a \textit{length space} if for each $x, y \in X$ and each $\epsilon > 0$, there is a curve from $x$ to $y$ with $D$-length at most $D(x, y) + \epsilon$.

      We call a metric $D$ on an open set $U \subset \CC$ \textit{continuous} if it induces the Euclidean topology, i.e. the identity mapping $(U, |\cdot|) \to (U, D)$ is a homeomorphism.
      We will equip the space of continuous metrics on $U$ with the topology of local uniform convergence of functions $U \times U \to \hrint{0}{\infty}$ with its associated Borel $\sigma$-algebra.
      If $U$ is disconnected, we allow $D(x, y) = \infty$ when $x$ and $y$ are in different connected components of $U$.
      In this case, in order to have a sequence of continuous metrics $D^n$ converge to $D$ locally uniformly, we require $D^n(x, y) = \infty$ for all $n$ sufficiently large if and only if $D(x, y) = \infty$.

      If $U \subset \CC$ is an open set, let $\cD'(U)$ denote the space of distributions on $U$ equipped with the weak topology.
      Fix $\gamma \in (0, 2)$ and let $\xi \coloneqq \gamma/d_{\gamma}$.
      A $\gamma$-\textit{LQG metric} is a collection of measurable functions $D^U \colon \cD'(U) \to \{\text{continuous metrics on } U\}$, one for each open set $U \subset \CC$, such that when $h$ is a GFF plus a continuous function on $U$, the following axioms are true.
      \begin{enumerate}[label=\Roman*]
        \item \label{axiom:LengthSpace} \textbf{Length space.} Almost surely, $(U, D_h^U)$ is a length space.
        \item \label{axiom:Locality} \textbf{Locality.} If $V \subset U$ is a deterministic open subset, then almost surely, $D_h^U(\cdot, \cdot; V) = D_{h|_V}^V(\cdot, \cdot)$.
        \item \label{axiom:WeylScaling} \textbf{Weyl scaling.} For each continuous function $f \colon U \to \RR$, define
          \begin{align*}
            \left(e^{\xi f} \cdot D_h^U \right)(z,w) \
            &\coloneqq \ \inf_{P \colon z \to w} \int\limits_{0}^{\Len(P; D_h^U)} e^{\xi f(P(t))} \, \D t, \ \forall z,w \in \CC,
          \end{align*}
          where the infimum is over all curves $P$ from $z$ to $w$ parameterized by $D_h^U$-length.
          Then almost surely, $e^{\xi f} \cdot D_h^U = D_{h + f}^U$ for all continuous functions $f \colon U \to \RR$.
        \item \label{axiom:CoordinateChange} \textbf{Coordinate change.} For each conformal transformation $\phi \colon U \to \phi(U)$, almost surely,
          \begin{align*}
            D_h^U(z, w) \
            &= \ D_{h \circ \phi^{-1} + Q \log|(\phi^{-1})'|}^{\phi(U)} (\phi(z),\phi(w)) \ \forall z,w \in U.
          \end{align*}
      \end{enumerate}
      See \cite[Section 1.2]{ExistenceAndUniqueness} for a discussion regarding why these axioms are natural.
      Note that the axioms only specify properties of $D_h^U$ when $h$ is a GFF plus continuous function; for other distributions, $D_h^U$ can be defined arbitrarily.

    \subsection{Outline of main results}
      In axiom \ref{axiom:CoordinateChange}, the almost sure event on which the coordinate change formula holds is allowed to depend on the conformal map $\phi$.
      This is not ideal for a few reasons.
      Often, one would like to work with random coordinate changes, but instead must rely on ad hoc arguments such as those of \cite[Section 2.4.1]{MinkowskiContentOfLQGMetric}.
      Moreover, in Riemannian geometry, one usually views isometric surfaces as different parameterizations of the ``same'' surface, so the central objects of study are equivalence classes of isometric surfaces rather than specific equivalence class representatives.
      Taking the perspective that LQG is ``random Riemannian geometry'', LQG surfaces should be viewed as random equivalence classes of surfaces with equivalence given by the coordinate change formula.
      For this heuristic to be valid, all parameterizations of the same LQG surface must be equivalent simultaneously.
      This was shown to be true for the LQG area measure in \cite{SheffieldWangConformalCoordinateChange}, but their proof does not adapt nicely to the LQG metric due to the additional challenge posed by the non-linear minimization over paths.
      The main result of this paper is the proof of the analogous result for the LQG metric.

      \begin{thm}
        \label{thm:StrongCoordinateChange}
        If $\xi < \xi_{\text{crit}}$, there is a version of the LQG metric $(D^U)_{U \subset \CC}$ which satisfies:
        \begin{enumerate}[label=\Roman*]
          \setcounter{enumi}{4}
          \item \label{axiom:StrongCoordinateChange} \textbf{Strong coordinate change} Fix an open set $U \subset \CC$.
            If $h$ is a whole-plane GFF plus a continuous function on $U$, then almost surely, for each conformal transformation $\phi \colon U \to \phi(U)$, 
            \begin{align*}
              D_h^U(z, w) \
              &= \ D_{h \circ \phi^{-1} + Q \log|(\phi^{-1})'|}^{\phi(U)}\left(\phi(z), \phi(w) \right) \ \forall z, w \in U.
            \end{align*}
        \end{enumerate}
      \end{thm}

      We only prove Theorem \ref{thm:StrongCoordinateChange} for $\xi < \xi_{\text{crit}}$, but parts of our argument hold for $\xi \geq \xi_{\text{crit}}$.
      Specifically, all results in Section \ref{section:SmallScale} hold for arbitrary $\xi$.
      The results in sections \ref{section:LargeScale} and \ref{section:CoordinateChange} only hold for $\xi < \xi_{\text{crit}}$ since they rely on almost sure convergence of LFPP locally uniformly, which is only known for $\xi < \xi_{\text{crit}}$ \cite{AlmostSureConvergence}.

      Let us briefly describe the main ideas in the proof of Theorem \ref{thm:StrongCoordinateChange}.
      Let $h$ be a whole-plane GFF and fix an open set $U \subset \CC$.
      To ease notation, if $\phi \colon U \to \phi(U)$ is a conformal map, write $\hypertarget{hphi}{h^{\phi}} \coloneqq h \circ \phi^{-1} + Q \log|(\phi^{-1})'|$.
      The goal is roughly to prove that for all conformal transformations $\phi \colon U \to \phi(U)$, the metrics $\LFPP[\epsilon][\hphi](\phi(\cdot), \phi(\cdot))$ all converge simultaneously along the dyadic sequence $\epsilon = \{2^{-j}\}_{j \in \NN}$ to the same limiting metric.
      We can then define $D_h^U$ to equal $\lim_{\epsilon \to 0} \LFPP[\epsilon][h]$ on the event of this simultaneous convergence, and define it arbitrarily otherwise.
      It is not hard to check that the collection of functions $(h \mapsto D^U)_{U \subset \CC}$ satisfies the LQG metric axioms plus the strong coordinate change formula.

      To sketch the proof of simultaneous convergence, fix a conformal transformation $\phi \colon U \to \phi(U)$ and assume $0 \in U$.
      Near $0$, $\phi$ looks like the affine transformation $z \mapsto a z + b$ with $a = \phi'(0)$ and $b = \phi(0)$, so $\LFPP[\epsilon][\hphi](\phi(\cdot), \phi(\cdot))$ should be well-approximated by $\LFPP[\epsilon][h(a^{-1}(\cdot - b)) + Q \log|a|](a \cdot + b, a \cdot + b)$.
      Using the exact scaling relation \eqref{eq:LFPPScaling} for LFPP, 
      \begin{align*}
        \LFPP[\epsilon][h(a^{-1}(\cdot - b)) + Q \log|a|]\left(a \cdot + b, a \cdot + b \right) \
        &= \ \frac{|a| \fa_{\epsilon}^{-1}}{|a|^{\xi Q} \fa_{\epsilon |a|}^{-1}} \LFPP[\epsilon |a|][h]\left(\cdot, \cdot \right).
      \end{align*}
      The factor $\frac{|a| \fa_{\epsilon}^{-1}}{|a|^{\xi Q} \fa_{\epsilon |a|}^{-1}}$ should be close to $1$ by regular variation of $\epsilon \mapsto \fa_{\epsilon}$ \cite[Corollary 1.11]{ExistenceAndUniqueness}.
      It is shown in \cite{AlmostSureConvergence} that $\LFPP[\epsilon][h]$ converges locally uniformly along the continuum index $\epsilon \in (0, 1)$ almost surely to a metric on $\CC$ we will denote by $D_h$.
      Therefore, the metrics $\LFPP[\epsilon |a|][h](\cdot, \cdot)$ converge to $D_h$.
      Note that almost sure convergence along the continuum $(0, 1)$ rather than a subsequence is crucial to prove simultaneous convergence for all $\phi$ since $\epsilon |a|$ depends on $\phi$.
      Since this argument only applies ``near $0$'', i.e. on a ball of vanishing radius as $\epsilon \to 0$, we appeal to local independence properties of the GFF to prove there are many such balls where this small-scale argument holds, which can be used to deduce convergence of $\LFPP[\epsilon][\hphi](\phi(\cdot), \phi(\cdot))$ to $D_h$ globally.

      There are several issues which must be dealt with to rigorize this proof sketch.
      First, the metrics $\LFPP[\epsilon][\hphi]$ do not make sense since $h \circ \phi^{-1}$ is only defined on $\phi(U)$ but the heat kernel is nonzero on all of $\CC$.
      To avoid this, we will work with a variant of LFPP which depends locally on the GFF.
      This localized variant does not satisfy the exact spatial scaling property \eqref{eq:LFPPScaling}, so we will need to compare to ordinary LFPP defined using the heat kernel mollification, then rescale.

      Second, we will need to approximate $\LFPP[\epsilon][\hphi](\phi(\cdot), \phi(\cdot))$ by $\LFPP[\epsilon][h(a^{-1}(\cdot - b)) + Q \log|a|](a \cdot + b, a \cdot + b)$ in a manner which is uniform over all conformal maps $\phi$.
      This will be done using distortion estimates to carefully compare mollifications of $h \circ \phi^{-1}$ and $h(a^{-1} (\cdot - b))$.
      This comparison is carried out in Section \ref{section:SmallScale}.

      Third, since our comparison of $\LFPP[\epsilon][\hphi](\phi(\cdot), \phi(\cdot))$ and $\LFPP[\epsilon |a|][h]$ only holds on small scales, we will use a multiscale argument to upgrade to a global comparison between $\LFPP[\epsilon][\hphi](\phi(\cdot), \phi(\cdot))$ and $D_h$.
      This argument is inspired by the argument in Section 3.2 of \cite{ConformalCovariance} and its adaptation in \cite{AlmostSureConvergence}.
      The idea is to use local independence properties of the GFF plus the small-scale estimate described above to show that $\LFPP[\epsilon][\hphi](\phi(\cdot), \phi(\cdot))$ and $D_h$ are ``almost Lipschitz equivalent''.
      We will then show that the Lipschitz constant can be made arbitrarily close to $1$ as $\epsilon \to 0$, implying that all the coordinate-changed LFPP metrics converge to the same limit.
      This argument is carried out in Section \ref{section:LargeScale}.

  \section{Preliminaries}
    \subsection{Notation}
      For $z \in \CC$ and $r > 0$, let $B_r(z) = \{w \in \CC : |z - w| < r\}$.
      More generally, for $U \subset \CC$, let $B_r(U) \coloneqq \cup_{z \in U} B_r(z)$.
      For $z \in \CC$ and $0 < r_1 < r_2$, let $\AA_{r_1, r_2}(z) \coloneqq B_{r_2}(z) \setminus \overline{B_{r_1}(z)}$.
      If $f \colon (0, \infty) \to \RR$ and $g \colon (0, \infty) \to (0, \infty)$ are functions, we say $f(\epsilon) = O_{\epsilon}(g(\epsilon))$ if $f(\epsilon)/g(\epsilon)$ remains bounded as $\epsilon \to 0$ or as $\epsilon \to \infty$ depending on the context.
      If $U, V \subset \CC$, we write $V \Subset U$ to mean $\overline{V}$ is compact and $\overline{V} \subset U$.
      If $h$ is a distribution on a subdomain $U \subset \CC$ and $\phi \colon U \to \phi(U)$ is a conformal transformation, we define a distribution on $\phi(U)$ by $\hphi \coloneqq h \circ \phi^{-1} + Q \log|(\phi^{-1})'|$.

      If $D$ is a metric on a subdomain $U$ of $\CC$ and $A \subset U$ is a region with the topology of an annulus, we define $D(\text{around } A)$ to be the infimum over all $D$-lengths of paths in $A$ which disconnect the inner and outer boundaries.

      In \cite[Theorem 1.4]{AlmostSureConvergence}, a version of the LQG metric on $\CC$ is constructed which satisfies axiom \ref{axiom:StrongCoordinateChange} for all affine transformations.
      This function will be used throughout this paper, and will be denoted by $\cD'(\CC) \ni h \mapsto D_h$.
      The statement of Theorem 1.4 of \cite{AlmostSureConvergence} is that almost surely, for all $a \in \CC \setminus \{0\}$ and $b \in \CC$,
      \begin{align}
        D_h\left(a \cdot + b, a \cdot + b \right) \
        &= \ D_{h(a \cdot + b) + Q \log|a|}\left(\cdot, \cdot \right).
        \label{eq:ASCoordinateChangeForAffineTransformations}
      \end{align}

    \subsection{Variants of LFPP}
      Let us discuss the variants of LFPP which will be used in the remainder of this article.
      As stated in the introduction, LFPP defined with the heat kernel mollification of a GFF \eqref{eq:LFPPDefn} appropriately renormalized is known to converge uniformly on compacts as $\epsilon \to 0$ to the $\gamma$-LQG metric.
      The main reason the heat kernel mollification is used is that this is the continuous approximation of the GFF for which tightness of LFPP is known \cite{TightnessOfSubcriticalLFPP}.
      However, $h_{\epsilon}^{*}$ is also a convenient choice of regularization because $\LFPP$ satisfies an exact coordinate change formula for affine transformations: if $a \in \CC \setminus \{0\}$ and $b \in \CC$, then
      \begin{align}
        \LFPP[\epsilon][h]\left(a z + b, a w + b \right) \
        &= \ \frac{|a| \fa_{\epsilon}^{-1}}{|a|^{\xi Q} \fa_{\epsilon/|a|}^{-1}} \LFPP[\epsilon/|a|][h(a \cdot + b) + Q \log|a|]\left(z,w \right).
        \label{eq:LFPPScaling}
      \end{align}
      This is used in \cite{AlmostSureConvergence} to define a version of the LQG metric on $\CC$ which satisfies the coordinate change formula for all affine transformations simultaneously.

      One inconvenience of working with $\LFPP$ is that $p_{\epsilon^2/2}(\cdot)$ is non-vanishing on all of $\CC$, so $\LFPP$ doesn't depend locally on the field.
      This is troublesome if we want to take advantage of local independence properties of the GFF (see Lemma \ref{lemma:IndependenceAcrossConcentricAnnuli} below).
      Additionally, we would like to consider variants of LFPP defined for fields of the form $h \circ \phi^{-1}$, where $h$ is a GFF on a subdomain $U \subset \CC$ and $\phi \colon U \to \phi(U)$ is a conformal transformation.
      The mollification of such a field with the heat kernel doesn't make sense since $\phi^{-1}$ is defined only on $\phi(U)$.
      For these reasons, we will consider a truncated version of LFPP.
      Let $\psi = \psi_1 \colon \CC \to [0,1]$ be a smooth radially-symmetric bump function which is identically equal to $1$ on $B_{1/2}(0)$ and vanishes outside $B_{1}(0)$.
      Define
      \begin{align*}
        \psi_{\epsilon}(z) \
        &\coloneqq \ \psi\left(\frac{z}{\frepsilon} \right).
      \end{align*}
      The reason for the scaling factor $R = R(\epsilon) = \frepsilon$ is as follows.
      Our localized variant of $\epsilon$-LFPP on an open set $U$ will be determined by the GFF restricted to $B_{R}(U)$.
      Our arguments will rely on local independence properties of the GFF restricted to disjoint concentric annuli with radii $\approx \epsilon^{1 - \zeta}$ with $\zeta \in (0, 1)$ close to $0$, so we need $R(\epsilon) = o(\epsilon^{1 - \zeta})$ for any $\zeta \in (0, 1)$.
      This suggests $R(\epsilon) = \epsilon$ is appropriate, but choosing $R(\epsilon) = \frepsilon$ allows us to compare ordinary LFPP \eqref{eq:LFPPDefn} with our localized variant (see \ref{loc:UniformComparison} from Lemma \ref{lemma:PropertiesOfLocalizedFieldAndLFPP} below).

      Define
      \begin{align*}
        Z_{\epsilon} \
        &\coloneqq \ \int\limits_{\CC} \psi_{\epsilon}(w) p_{\epsilon^2/2}(w) \, \D A(w), \\
        \hat{h}_{\epsilon}^{*}(z) \
        &= \ Z_{\epsilon}^{-1} h \ast (\psi_{\epsilon} p_{\epsilon^2/2})(z) \
        \coloneqq \ Z_{\epsilon}^{-1} \left\langle h, \psi_{\epsilon}(\cdot - z) p_{\epsilon^2/2}(\cdot - z) \right\rangle.
      \end{align*}
      The \textit{localized} $\epsilon$-\textit{LFPP} metric with parameter $\xi > 0$ is defined by
      \begin{align*}
        \hat{D}_h^{\epsilon}(z,w) \
        &\coloneqq \ \inf_{P \colon z \to w} \int\limits_{0}^{1} e^{\xi \hat{h}_{\epsilon}^{*}(P(t))} |P'(t)| \, \D t,
      \end{align*}
      where the infimum is over piecewise $C^1$ paths $P$ from $z$ to $w$.
      It is shown in the proof of \cite[Theorem 1.1]{AlmostSureConvergence} that $\locLFPP$ converges almost surely uniformly on compacts.
      The advantage of working with $\locLFPP$ is that it depends locally on the field.
      In particular, if $\phi \colon U \to \phi(U)$ is a conformal transformation and $V \Subset U$, then the metric
      \begin{align*}
        &\hat{D}_{h \circ \phi^{-1} + Q \log|(\phi^{-1})'|}^{\epsilon} \left(\phi(z), \phi(w); \phi(V) \right) \\
        &\qquad \qquad \coloneqq \ \inf_{\substack{P \colon z \to w \\ P \subset V}} \int\limits_{0}^{1} e^{\xi Z_{\epsilon}^{-1} (h \circ \phi^{-1} + Q \log|(\phi^{-1})'|) \ast (\psi_{\epsilon} p_{\epsilon^2/2})(\phi(P(t)))} |(\phi \circ P)'(t)| \, \D t,
      \end{align*}
      where the infimum is over piecewise $C^1$ paths $P$ from $z$ to $w$ contained in $V$, is well-defined provided $B_{\frepsilon}(\phi(V)) \subset \phi(U)$.
      Note that if $\DerivativeBound > 1$, then by the Koebe-$\frac{1}{4}$ theorem, there exists $\epsilon_0 > 0$ such that for all $\epsilon \in (0, \epsilon_0)$ and all conformal maps $\phi \colon U \to \phi(U)$ with $\DerivativeBound^{-1} \leq |\phi'(z)| \leq \DerivativeBound$ for all $z \in V$, we have $B_{\frepsilon}(\phi(V)) \subset \phi(U)$, hence $\hat{D}_{h \circ \phi^{-1} + Q \log |(\phi^{-1})'|}^{\epsilon}(\phi(\cdot), \phi(\cdot); \phi(V))$ are all well-defined.


      The following lemma summarizes the basic properties of $\hat{h}_{\epsilon}^{*}$ and $\locLFPP$ which will be used in this paper.
      \begin{lemma}
        \label{lemma:PropertiesOfLocalizedFieldAndLFPP}
        \leavevmode
        \begin{enumerate}
          \item \label{loc:AddRandomVariableToField} If $c \in \RR$ is a random variable, then $\widehat{(h + c)}_{\epsilon}^{*}(z) = \hat{h}_{\epsilon}^{*}(z) + c$ for all $z \in \CC$.
          \item \label{loc:Locality} $\hat{h}_{\epsilon}^{*}(z)$ is almost surely determined by $h|_{B_{\frepsilon}(z)}$.
            Consequently, for any deterministic open set $U \subset \CC$, the internal metric $\hat{D}_h^{\epsilon}(\cdot, \cdot; U)$ is almost surely determined by $h|_{B_{\frepsilon}(U)}$.
          \item \label{loc:Continuity} $\hat{h}_{\epsilon}^{*}(z)$ has a modification which is jointly continuous in $z$ and $\epsilon$.
            We will always assume we are working with such a modification.
          \item \label{loc:UniformComparison} Let $U \subset \CC$ be a connected, bounded, open set.
            Almost surely, 
            \begin{align*}
              \lim_{\epsilon \to 0} \sup_{z \in \overline{U}} \left|h_{\epsilon}^{*}(z) - \hat{h}_{\epsilon}^{*}(z) \right| \
              &= \ 0,
            \end{align*}
            and 
            \begin{align*}
              \lim_{\epsilon \to 0} \frac{\hat{D}_h^{\epsilon}(z,w;V)}{D_h^{\epsilon}(z,w;V)} \
              &= \ 1, \ \text{ uniformly over all } z,w \in V \text{ with } z \neq w \text{ and all connected } V \subset U.
            \end{align*}
          \item \label{loc:TranslationInvariance} If $b \in \CC$ and if $H(z) \coloneqq h(z + b)$, then
            \begin{align*}
              \locLFPP(z + b, z + w) \
              &= \ \locLFPP[\epsilon][H](z,w) \ \forall (z,w) \in \CC^2.
            \end{align*}
          \item \label{loc:AlmostSureConvergence} If $h$ is a whole-plane GFF, then $\locLFPP \to D_h$ locally uniformly almost surely.
        \end{enumerate}
        \begin{proof}
          \ref{loc:AddRandomVariableToField} and \ref{loc:TranslationInvariance} are routine calculations, while \ref{loc:Locality} follows from the fact that $\psi_{\epsilon}$ is supported on $B_{\frepsilon}(0)$.
          For \ref{loc:Continuity} and \ref{loc:UniformComparison}, see \cite[Lemma 2.2]{UpToConstants}.
          Finally, \ref{loc:AlmostSureConvergence} is shown in the proof of Theorem 1.1 in \cite{AlmostSureConvergence}.
        \end{proof}
      \end{lemma}

  \section{Comparison of LFPP and coordinate-changed LFPP at small scales}
    \label{section:SmallScale}
    Given open sets $V \Subset U$ and $\DerivativeBound > 1$, define
    \begin{align*}
      \hypertarget{confmaps}{\Lambda_{\DerivativeBound}(V, U)} \
      &\coloneqq \ \left\{\phi \colon U \to \phi(U) : \phi \text{ is conformal, } |\phi'(z)| \in [\DerivativeBound^{-1}, \DerivativeBound] \ \forall z \in \overline{V} \right\}.
    \end{align*}
    Throughout this section, $h$ will denote a whole-plane GFF and the LQG metric parameter $\xi > 0$ is arbitrary.
    In this section, we will compare the mollification of $\hphi$ with the mollification of $h$ in a manner which is uniform over $\phi \in \confmaps$.
    The basic ingredient is distortion estimates for conformal maps.
    Our main goal is to prove the following.
    \begin{prop}
      \label{prop:MainRegularityConditionHighProbability}
      Fix open sets $W \Subset V \Subset U$, $\DerivativeBound > 1$, and $\zeta, p, \delta \in (0, 1)$.
      There exists $\epsilon_0 = \epsilon_0(W, V, U, \DerivativeBound, \zeta, p, \delta) \in (0,1)$ such that for each $\epsilon \in (0, \epsilon_0)$ and each $z_0 \in W$, it holds with probability at least $p$ that for all $\phi \in \confmaps$ and all piecewise $C^1$ paths $P$ in $B_{2 \epsilon^{1 - \zeta}}(z_0)$,
      \begin{align*}
        \left(1 + \delta \right)^{-1} \Len\left(P; \locLFPP[\epsilon/|\phi'(z_0)|][h] \right) \
        &\leq \ \Len\left(\phi \circ P; \locLFPP[\epsilon][\hphi] \right) \\
        &\leq \ \left(1 + \delta \right) \Len\left(P; \locLFPP[\epsilon/|\phi'(z_0)|][h] \right).
      \end{align*}
    \end{prop}

    To ease notation throughout this section, if $\phi \colon U \to \phi(U)$ is a conformal transformation and $z, w \in U$, define
    \begin{align*}
      \Psi_{\epsilon}^{\phi, z}(w) \
      &\coloneqq \ \frac{|\phi'(w)|^2}{Z_{\epsilon}} \psi_{\epsilon}(\phi(w) - \phi(z)) p_{\epsilon^2/2}(\phi(w) - \phi(z)) \\
      &= \ \frac{|\phi'(w)|^2}{Z_{\epsilon} \epsilon^2} \psi\left(\frac{\phi(w) - \phi(z)}{\frepsilon} \right) p_{1/2}\left(\frac{\phi(w) - \phi(z)}{\epsilon} \right).
    \end{align*}
    The following lemma summarizes the basic properties of the functions $\Psi_{\epsilon}^{\phi, z}$.
    Here and throughout, $\ast$ denotes convolution.
    \begin{lemma}
      \label{lemma:DistortedBumpProperties}
      Fix open sets $V \Subset U$ and $\DerivativeBound > 1$.
      \begin{enumerate}[label=(\roman*)]
        \item \label{DistortedBump:Support} If $\phi \in \confmaps$ and $z \in V$, then $\Psi_{\epsilon}^{\phi, z}$ is supported on $B_{4 \DerivativeBound \frepsilon}(z)$;
        \item \label{DistortedBump:FieldMollification} $Z_{\epsilon}^{-1} (h \circ \phi^{-1}) \ast (\psi_{\epsilon} p_{\epsilon^2/2})(\phi(z)) = \langle h, \Psi_{\epsilon}^{\phi, z} \rangle$;
        \item \label{DistortedBump:LogMollification} $Z_{\epsilon}^{-1} \log|(\phi^{-1})'| \ast (\psi_{\epsilon} p_{\epsilon^2/2})(\phi(z)) = \langle -\log|\phi'|, \Psi_{\epsilon}^{\phi, z} \rangle$.
      \end{enumerate}
      \begin{proof}
        \textit{Proof of \ref{DistortedBump:Support}:} Since $\psi$ is supported on $\DD$, $\Psi_{\epsilon}^{\phi, z}$ is supported on $\phi^{-1}(B_{\frepsilon}(\phi(z)))$.
        By the Koebe-$\frac{1}{4}$ theorem, the image of $\DD$ under $\omega \mapsto \phi(z + 4 \DerivativeBound \frepsilon \omega)$ contains $B_{\DerivativeBound \frepsilon |\phi'(z)|}(\phi(z))$. 
        Since $|\phi'(z)| \geq \DerivativeBound^{-1}$, it follows that $B_{\DerivativeBound \frepsilon |\phi'(z)|}(\phi(z)) \supset B_{\frepsilon}(\phi(z))$.
        Therefore, the support of $\Psi_{\epsilon}^{\phi, z}$ is contained in
        \begin{align*}
          \phi^{-1}\left(B_{\frepsilon}(\phi(z)) \right) \
          &\subset \phi^{-1}\left(B_{4 \DerivativeBound \frepsilon |\phi'(z)|/4}(\phi(z)) \right) \
          \subset \ B_{4 \DerivativeBound \frepsilon}(z).
        \end{align*}

        \textit{Proof of \ref{DistortedBump:FieldMollification}:} 
        \begin{align*}
          Z_{\epsilon}^{-1} (h \circ \phi^{-1}) \ast (\psi_{\epsilon} p_{\epsilon^2/2})(\phi(z)) \
          &= \ \left\langle h \circ \phi^{-1}, Z_{\epsilon}^{-1} \psi_{\epsilon}(\cdot - \phi(z)) p_{\epsilon^2/2}(\cdot - \phi(z)) \right\rangle \\
          &= \ \left\langle h, Z_{\epsilon}^{-1} |\phi'(\cdot)|^2 \psi_{\epsilon}(\phi(\cdot) - \phi(z)) p_{\epsilon^2/2}(\phi(\cdot) - \phi(z)) \right\rangle \\
          &= \ \left\langle h, \Psi_{\epsilon}^{\phi, z} \right\rangle.
        \end{align*}

        \textit{Proof of \ref{DistortedBump:LogMollification}:}
        \begin{align*}
          Z_{\epsilon}^{-1} \log|(\phi^{-1})'| \ast \psi_{\epsilon}(\phi(z)) \
          &= \ \left\langle \log|(\phi^{-1})'(\cdot)|, Z_{\epsilon}^{-1} \psi_{\epsilon}(\cdot - \phi(z)) p_{\epsilon^2/2}(\cdot - \phi(z)) \right\rangle \\
          &= \ \left\langle \log|(\phi^{-1})'(\phi(\cdot)), Z_{\epsilon}^{-1} |\phi'(\cdot)|^2 \psi_{\epsilon}(\phi(\cdot) - \phi(z)) p_{\epsilon^2/2}(\phi(\cdot) - \phi(z)) \right\rangle \\
          &= \ \left\langle -\log|\phi'|, \Psi_{\epsilon}^{\phi, z} \right\rangle.
        \end{align*}
      \end{proof}
    \end{lemma}

    The next lemma is our main estimate for comparing $\Psi_{\epsilon}^{\phi, z}$ and $\Psi_{\epsilon/|\phi'(z_0)|}^{\id, z}$ when $z$ is close to $z_0$.

    \begin{lemma}
      \label{lemma:GrowthRateOfDistortedBumpFunctions}
      Fix open sets $W \Subset V \Subset U$, $\DerivativeBound > 1$, and $\zeta \in (0, 1)$.
      There exist $C = C(W, V, \DerivativeBound, \zeta) > 0$ and $\epsilon_0 = \epsilon_0(W, V, \DerivativeBound, \zeta) \in (0,1)$ such that for all $\epsilon \in (0, \epsilon_0)$,
      \begin{align}
        \sup_{z_0 \in W} \sup_{z \in B_{2 \epsilon^{1 - \zeta}}(z_0)} \sup_{\phi \in \confmaps} \|\Psi_{\epsilon}^{\phi, z} - \Psi_{\epsilon/|\phi'(z_0)|}^{\id, z} \|_{L^{\infty}(U)} \
        &\leq \ C \epsilon^{-1 - \zeta} \log \epsilon^{-1}, \label{eq:GrowthRateZerothDerivative} \\
        \sup_{z_0 \in W} \sup_{z \in B_{2 \epsilon^{1 - \zeta}}(z_0)} \sup_{\phi \in \confmaps} \|\nabla \Psi_{\epsilon}^{\phi, z} - \nabla \Psi_{\epsilon/|\phi'(z_0)|}^{\id, z} \|_{L^{\infty}(U)} \
        &= \ C \epsilon^{-2 - \zeta} \log \epsilon^{-1}. \label{eq:GrowthRateFirstDerivative}
      \end{align}
      \begin{proof}
        The idea is that if $z$ is close to $z_0$, then on a small ball around $z$, $w \mapsto \phi(w)$ looks like the affine transformation $w \mapsto \phi'(z_0) (w - z) + \phi(z)$, where the error from this estimate is uniform over $\phi \in \confmaps$ by distortion-type inequalities.
        Therefore,
        \begin{align}
          p_{\epsilon^2/2}\left(\phi(w) - \phi(z) \right) \
          &\approx \ p_{\epsilon^2/2}\left(\left(\phi'(z_0)(w - z) + \phi(z) \right) - \phi(z) \right) \nonumber \\
          &= \ p_{1/2}\left(\frac{w - z}{\epsilon/|\phi'(z_0)|} \right) \nonumber \\
          &= \ p_{(\epsilon/|\phi'(z_0)|)/2}\left(z - w \right). \label{eq:DistortionHeuristic}
        \end{align}
        Since the bump function $\psi_{\epsilon}$ is $\equiv 1$ on $B_{\frepsilon/2}(0)$ and $p_{\epsilon^2/2}(z) \leq C e^{-(\log \epsilon^{-1})^2/4}$ when $z \not\in B_{\frepsilon/2}(0)$, we will use \eqref{eq:DistortionHeuristic} to show that $\psi_{\epsilon}(\phi(w) - \phi(z)) p_{\epsilon^2/2}(\phi(w) - \phi(z)) \approx \psi_{\epsilon/|\phi'(z_0)|}(w - z) p_{(\epsilon/|\phi'(z_0)|)/2}(w - z)$ for all $w$ near $z$.

        Since $\int_{\CC} p_{\epsilon^2/2}(w) \, \D A(w) = 1$, we see that 
        \begin{align}
          0 \
          &\leq \ 1 - Z_{\epsilon} \
          = \ \frac{2}{\epsilon^2} \int\limits_{0}^{\infty} r \left(1 - \psi_{\epsilon}(r) \right) e^{-r^2/\epsilon^2} \, \D r \
          \leq \ \frac{2}{\epsilon^2} \int\limits_{\frepsilon/2}^{\infty} r e^{-r^2/\epsilon^2} \
          = \ e^{-(\log \epsilon^{-1})^2/4},
          \label{eq:NormalizingConstantEstimate}
        \end{align}
        where we have used radial symmetry of $\psi$ to rewrite the integral $Z_{\epsilon}$ in polar coordinates.
        Therefore, $\|\Psi_{\epsilon}^{\phi, z} - Z_{\epsilon} \Psi_{\epsilon}^{\phi, z}\|_{L^{\infty}(U)} \leq C \epsilon^{-2} \epsilon^{\frac{1}{4} \log \epsilon^{-1}}$ and $\|\nabla \Psi_{\epsilon}^{\phi ,z} - Z_{\epsilon} \nabla \Psi_{\epsilon}^{\phi, z}\|_{L^{\infty}(U)} \leq C \epsilon^{-3} \epsilon^{\frac{1}{4} \log \epsilon^{-1}}$, where the constant $C$ depends only on $\DerivativeBound$ (not on $\phi$).
        Therefore, it will suffice to prove \eqref{eq:GrowthRateZerothDerivative} and \eqref{eq:GrowthRateFirstDerivative} with $Z_{\epsilon} \Psi_{\epsilon}^{\phi, z}$ and $Z_{\epsilon/|\phi'(z_0)|} \Psi_{\epsilon/|\phi'(z_0)|}^{\id, z}$ in place of $\Psi_{\epsilon}^{\phi, z}$ and $\Psi_{\epsilon/|\phi'(z_0)|}^{\id, z}$.

        Note that
        \begin{align*}
          Z_{\epsilon} \Psi_{\epsilon}^{\phi, z}(w) \
          &= \ \frac{|\phi'(w)|^2}{\epsilon^2} \psi\left(\frac{\phi(w) - \phi(z)}{\frepsilon} \right) p_{1/2}\left(\frac{\phi(w) - \phi(z)}{\epsilon} \right), \\
          \p_x Z_{\epsilon} \Psi_{\epsilon}^{\phi, z}(w) \
          &= \ \frac{2 \phi'(w) \cdot \phi''(w)}{\epsilon^2} \psi\left(\frac{\phi(w) - \phi(z)}{\frepsilon} \right) p_{1/2}\left(\frac{\phi(w) - \phi(z)}{\epsilon} \right) \\
          &\qquad + \frac{|\phi'(w)|^2}{\epsilon^2} p_{1/2}\left(\frac{\phi(w) - \phi(z)}{\epsilon} \right) (\nabla \psi)\left(\frac{\phi(w) - \phi(z)}{\frepsilon} \right) \cdot \frac{\phi'(w)}{\frepsilon} \\
          &\qquad + \frac{|\phi'(w)|^2}{\epsilon^2} \psi\left(\frac{\phi(w) - \phi(z)}{\frepsilon} \right) (\nabla p_{1/2})\left(\frac{\phi(w) - \phi(z)}{\epsilon} \right) \cdot \frac{\phi'(w)}{\epsilon},
        \end{align*}
        and similarly for $\p_y Z_{\epsilon} \Psi_{\epsilon}^{\phi, z}$, where $\cdot$ denotes the standard Euclidean dot product and $\p_x, \p_y$ denote derivatives with respect to the real and imaginary parts of $w = x + \i y$ respectively.
        Recall that $\psi \equiv 1$ on $B_{1/2}(0)$, so 
        \begin{align*}
          \left|Z_{\epsilon} \Psi_{\epsilon}^{\phi, z}(w) - \frac{|\phi'(w)|^2}{\epsilon^2} p_{1/2}\left(\frac{\phi(w) - \phi(z)}{\epsilon} \right)\right|
          &\leq
          \begin{cases}
            0 & \phi(w) \in B_{\frepsilon/2}(\phi(z)) \\
            C \epsilon^{-2} e^{-\frac{1}{4} (\log \epsilon^{-1})^2} & \text{otherwise}
          \end{cases},
        \end{align*}
        where the constant $C$ depends only on $\DerivativeBound$.
        A similar estimate holds for each of the three terms in $\p_x Z_{\epsilon} \Psi_{\epsilon}^{\phi, z}$, the terms in $\p_y Z_{\epsilon} \Psi_{\epsilon}^{\phi, z}$, and with $Z_{\epsilon/|\phi'(z_0)|} \Psi_{\epsilon/|\phi'(z_0)|}^{\id, z}$ in place of $Z_{\epsilon} \Psi_{\epsilon}^{\phi, z}$. 
        Therefore, if
        \begin{align*}
          \tilde{\Psi}_{\epsilon}^{\phi, z}(w) \
          &\coloneqq \ \frac{|\phi'(w)|^2}{\epsilon^2} p_{1/2}\left(\frac{\phi(w) - \phi(z)}{\epsilon} \right), \\
          \tilde{\Psi}_{\epsilon, x}^{\phi, z}(w) \
          &\coloneqq \ \frac{2 \phi'(w) \cdot \phi''(w)}{\epsilon^2} p_{1/2}\left(\frac{\phi(w) - \phi(z)}{\epsilon} \right) + \frac{|\phi'(w)|^2}{\epsilon^2} (\nabla p_{1/2})\left(\frac{\phi(w) - \phi(z)}{\epsilon}\right) \cdot \frac{\phi'(w)}{\epsilon}, \\
          \tilde{\Psi}_{\epsilon, y}^{\phi, z}(w) \
          &\coloneqq \ \frac{2 \phi'(w) \cdot \i \phi''(w)}{\epsilon^2} p_{1/2}\left(\frac{\phi(w) - \phi(z)}{\epsilon} \right) + \frac{|\phi'(w)|^2}{\epsilon^2} \left(\nabla p_{1/2} \right)\left(\frac{\phi(w) - \phi(z)}{\epsilon} \right) \cdot \frac{\i \phi'(w)}{\epsilon},
        \end{align*}
        it will suffice to prove that there is a constant $C > 0$ such that
        \begin{align*}
          \|\tilde{\Psi}_{\epsilon}^{\phi, z} - \tilde{\Psi}_{\epsilon/|\phi'(z_0)|}^{\id, z}\|_{L^{\infty}(U)} \
          &\leq \ C \epsilon^{-1 - \zeta} \log \epsilon^{-1}, \\
          \|\tilde{\Psi}_{\epsilon, x}^{\phi, z} - \tilde{\Psi}_{\epsilon/|\phi'(z_0)|, x}^{\id, z}\|_{L^{\infty}(U)} \
          &\leq \ C \epsilon^{-2 - \zeta} \log \epsilon^{-1}, \\
          \|\tilde{\Psi}_{\epsilon, y}^{\phi, z} - \tilde{\Psi}_{\epsilon/|\phi'(z_0)|, y}^{\id, z}\|_{L^{\infty}(U)} \
          &\leq \ C \epsilon^{-2 - \zeta} \log \epsilon^{-1}.
        \end{align*}
        To obtain these estimates, note that $p_{1/2}$ has bounded derivatives of orders $0$, $1$, and $2$, and $|\phi'|$ is bounded by $\DerivativeBound$ on $V$, so by the mean value theorem, it will suffice to prove that there is a constant $C$ such that for all $\phi \in \confmaps$, all $z_0 \in W$, all $z \in B_{2 \epsilon^{1 - \zeta}}(z_0)$, and all $w \in B_{4 \DerivativeBound \frepsilon}(z)$,
        \begin{align}
          \left|\phi'(w) - \phi'(z_0) \right| \
          &\leq \ C \epsilon^{1 - \zeta}, \label{eq:DeBrangesDerivativeBound} \\
          \left|\frac{\phi(w) - \phi(z)}{\epsilon} - \frac{w - z}{\epsilon/\phi'(z_0)} \right| \
          &\leq \ C \epsilon^{1 - \zeta} \log \epsilon^{-1}, \label{eq:DeBrangesFunctionBound} \\
          |\phi''(w)| \
          &\leq \ C. \label{eq:DeBrangesSecondDerivativeBound}
        \end{align}
        In \eqref{eq:DeBrangesFunctionBound}, note that we have written $\phi'(z_0)$ rather than $|\phi'(z_0)|$ because $p_{(\epsilon/|\phi'(z_0)|)/2}(w - z) = p_{1/2}(\frac{w - z}{\epsilon/\phi'(z_0)})$ by rotational symmetry of the heat kernel.
        Let us now prove \eqref{eq:DeBrangesDerivativeBound}, \eqref{eq:DeBrangesFunctionBound}, and \eqref{eq:DeBrangesSecondDerivativeBound}.

        Fix $r > 0$ such that $B_r(W) \subset V$ and fix $\epsilon_0 \in (0, 1)$ such that $\frac{4 \DerivativeBound \frepsilon + 2 \epsilon^{1 - \zeta}}{r} < \frac{1}{2}$ for all $\epsilon \in (0, \epsilon_0)$.
        For any $z \in W$, Taylor expand the function $\omega \mapsto \frac{1}{r} \frac{\phi(z + r \omega) - \phi(z)}{\phi'(z)} = \omega + \sum_{n=2}^{\infty} a_n \omega^n$ on $\DD$, and apply de Branges's theorem\footnote{We do not need the full strength of de Branges's theorem, but use it for convenience. For example, the result of Littlewood \cite{Littlewood} $|a_n| \leq e n$ for all $n$ is sufficient.} \cite{deBranges} to see that $|a_n| \leq n$ for all $n$.
        Therefore, for all $\omega \in \DD$,
        \begin{align}
          &\left|\frac{\phi'(z_0 + (4 \DerivativeBound \frepsilon + 2 \epsilon^{1 - \zeta}) \omega)}{\phi'(z_0)} - 1 \right| \nonumber \\
          &\qquad \qquad = \ \left|\frac{\phi'(z_0 + r \frac{(4 \DerivativeBound \frepsilon + 2 \epsilon^{1 - \zeta}) \omega}{r})}{\phi'(z_0)} - 1 \right| \nonumber \\
          &\qquad \qquad \leq \ \sum_{n=2}^{\infty} n^2 \left|\frac{4 \DerivativeBound \frepsilon + 2 \epsilon^{1 - \zeta}}{r} \right|^{n-1} \nonumber \\
          &\qquad \qquad = \ \frac{4 \DerivativeBound \frepsilon + 2 \epsilon^{1 - \zeta}}{r} \frac{4 - 3\frac{4 r \frepsilon + 2 \epsilon^{1 - \zeta}}{r} + (\frac{4 r \frepsilon + 2 \epsilon^{1 - \zeta}}{r})^2}{(1 - \frac{4 \DerivativeBound \frepsilon + 2 \epsilon^{1 - \zeta}}{r})^3} \nonumber \\
          &\qquad \qquad \leq \ 34 \frac{4 \DerivativeBound \frepsilon + 2 \epsilon^{1 - \zeta}}{r}, \qquad \qquad \qquad \qquad \left(\frac{4 \DerivativeBound \frepsilon + 2 \epsilon^{1 - \zeta}}{r} < \frac{1}{2} \right). \label{eq:DerivativePreliminaryDistortionEstimate}
        \end{align}
        For any $z \in B_{2 \epsilon^{1 - \zeta}}(z_0)$ and $w \in B_{4 \DerivativeBound \frepsilon}(z)$, we can find $\omega \in \DD$ with $w = z_0 + (4 \DerivativeBound \frepsilon + 2 \epsilon^{1 - \zeta}) \omega$.
        Applying \eqref{eq:DerivativePreliminaryDistortionEstimate} with this choice of $\omega$ and multiplying through by $|\phi'(z_0)| \leq \DerivativeBound$ proves \eqref{eq:DeBrangesDerivativeBound}.

        Similarly, for each $\omega \in \DD$,
        \begin{align}
          &\left|\frac{\phi(z + 4 \DerivativeBound \frepsilon \omega) - \phi(z)}{\phi'(z) \epsilon} - \frac{(z + 4 \DerivativeBound (\frepsilon) \omega) - z}{\epsilon} \right| \nonumber \\
          &= \ \frac{r}{\epsilon} \left|\frac{\phi(z + r \frac{4 \DerivativeBound (\frepsilon)\omega}{r}) - \phi(z)}{\phi'(z) r} - \frac{4 \DerivativeBound (\frepsilon) \omega}{r} \right| \nonumber \\
          &= \ \frac{r}{\epsilon} \left|\sum_{n=2}^{\infty} a_n \left(\frac{4 \DerivativeBound (\frepsilon) \omega}{r} \right)^n \right| \nonumber \\
          &\leq \ \frac{r}{\epsilon} \sum_{n=2}^{\infty} n \left|\frac{4 \DerivativeBound \frepsilon}{r} \right|^n \nonumber \\
          &= \ \frac{r}{\epsilon} \left(\frac{\frac{4 \DerivativeBound \frepsilon}{r}}{(1 - \frac{4 \DerivativeBound \frepsilon}{r})^2} - \frac{4 \DerivativeBound \frepsilon}{r} \right) \nonumber \\
          &= \ \frac{16 \DerivativeBound^2 \epsilon (\log \epsilon^{-1})^2}{r} \frac{2 - \frac{4 \DerivativeBound \frepsilon}{r}}{(1 - \frac{4 \DerivativeBound \frepsilon}{r})^2} \nonumber \\
          &\leq \ \frac{128 \DerivativeBound^2 \epsilon (\log \epsilon^{-1})^2}{r} & \left(\frac{4 \DerivativeBound \frepsilon}{r} < \frac{1}{2} \right) \nonumber \\
          &\leq \ \frac{128 \DerivativeBound^2}{r} \epsilon^{1 - \zeta} \label{eq:LipschitzDeBrangesEstimate}
        \end{align}
        For $w \in B_{4 \DerivativeBound \frepsilon}(z)$, we can find $\omega \in \DD$ such that $w = z + 4 \DerivativeBound \frepsilon \omega = w$, then apply \eqref{eq:LipschitzDeBrangesEstimate} with this choice of $\omega$ to get
        \begin{align}
          \left|\frac{\phi(w) - \phi(z)}{\phi'(z) \epsilon} - \frac{w - z}{\epsilon} \right| \
          &\leq \ \frac{128 \DerivativeBound^2}{r} \epsilon^{1 - \zeta}.
          \label{eq:LipschitzDeBrangesEstimateSubstituted}
        \end{align}
        To deduce \eqref{eq:DeBrangesFunctionBound} from \eqref{eq:LipschitzDeBrangesEstimateSubstituted}, note that 
        \begin{align*}
          \left|\frac{\phi(w) - \phi(z)}{\epsilon} - \frac{w - z}{\epsilon/\phi'(z_0)} \right| \
          &\leq \ \left|\frac{\phi(w) - \phi(z)}{\epsilon} \right| \left|1 - \frac{\phi'(z_0)}{\phi'(z)} \right| + |\phi'(z_0)| \left|\frac{\phi(w) - \phi(z)}{\phi'(z) \epsilon} - \frac{w - z}{\epsilon} \right| \\
          &\leq \ \left(4 \DerivativeBound^2 \log \epsilon^{-1} \right) C \DerivativeBound \epsilon^{1 - \zeta} + \DerivativeBound  \frac{128 \DerivativeBound^2}{r} \epsilon^{1 - \zeta}
        \end{align*}
        where we have applied \eqref{eq:DeBrangesDerivativeBound} with $z$ in place of $w$ to bound $|1 - \frac{\phi'(z_0)}{\phi'(z)}|$ and used the mean value theorem to bound
        \begin{align*}
          \left|\frac{\phi(w) - \phi(z)}{\epsilon} \right| \
          &\leq \ \frac{|\phi'(c)| 4 \DerivativeBound \frepsilon}{\epsilon} \
          \leq \ 4 \DerivativeBound^2 \log \epsilon^{-1},
        \end{align*}
        where $c \in V$ is some point on the line segment from $z$ to $w$.

        Finally, the bound
        \begin{align*}
          &\left|\frac{r \phi''(z + 4 \DerivativeBound \frepsilon \omega)}{\phi'(z)} \right| \\
          &\qquad \qquad \leq \ \sum_{n=2}^{\infty} n^2(n-1) \left|\frac{4 \DerivativeBound \frepsilon}{r} \right|^{n-2} \\
          &\qquad \qquad = \ \left(\frac{4 \DerivativeBound \frepsilon}{r} \right)^2 \frac{2(24 - 46 \frac{4 \DerivativeBound \frepsilon}{r} + 34 (\frac{4 \DerivativeBound \frepsilon}{r})^2 - 9 (\frac{4 \DerivativeBound \frepsilon}{r})^3)}{(1 - \frac{4 \DerivativeBound \frepsilon}{r})^4} \\
          &\qquad \qquad \leq \ 1040 \left(\frac{4 \DerivativeBound \frepsilon}{r} \right)^2, \qquad \qquad \qquad \qquad \left(\frac{4 \DerivativeBound \frepsilon}{r} < \frac{1}{2} \right)
        \end{align*}
        applied with $\omega \in \DD$ chosen so that $w = z + 4 \DerivativeBound \frepsilon \omega$, together with the fact that $|\phi'(z)| \geq \DerivativeBound^{-1}$, is enough to obtain \eqref{eq:DeBrangesSecondDerivativeBound}.
      \end{proof}
    \end{lemma}

    Our next lemma will be used to compare the mollified fields $\hat{h}_{\epsilon/|\phi'(z_0)|}^{*}$ and $\hat{h}_{\epsilon/|\phi'(z_1)|}^{*}$ when $z_0$ and $z_1$ are close.

    \begin{lemma}
      \label{lemma:MollifiedFieldComparison}
      Fix a compact set $K \subset \CC$ and constants $C > 0$ and $\zeta \in (0, 1)$.
      Almost surely,
      \begin{align*}
        \sup_{\substack{t, s \in [\DerivativeBound^{-1}, \DerivativeBound] \\ 1 - C \epsilon^{1 - \zeta} \leq \frac{t}{s} \leq 1 + C \epsilon^{1 - \zeta}}} \sup_{z \in K} \left|\hat{h}_{\epsilon t}^{*}(z) - \hat{h}_{\epsilon s}^{*}(z) \right| \
        &\stackrel{\epsilon \to 0}{\to} \ 0.
      \end{align*}
      \begin{proof}
        We will prove that there is a $K$-dependent random variable $X$ such that for $\epsilon, \delta \in (0, e^{-1})$ with $1 \leq \frac{\delta}{\epsilon} \leq 2$,
        \begin{align}
          \sup_{z \in K} \left|\hat{h}_{\epsilon}^{*}(z) - \hat{h}_{\delta}^{*}(z) \right|
          &\leq X \left(\log \epsilon^{-1} \right)^3 \log \frac{1}{\epsilon \log \epsilon^{-1}} \left[\left|\frac{\delta \log \delta^{-1}}{\epsilon \log \epsilon^{-1}} - 1 \right| \vee \left(\frac{\delta}{\epsilon} - 1 \right) + 2 e^{-(\log \delta^{-1})^2/4} \right].
          \label{eq:MollifiedFieldComparison}
        \end{align}
        Applying \eqref{eq:MollifiedFieldComparison} with $\epsilon t$ in place of $\epsilon$ and $\epsilon s$ in place of $\delta$ will prove the lemma.
        Since
        \begin{align*}
          \left|\hat{h}_{\epsilon}^{*}(z) - \hat{h}_{\delta}^{*}(z) \right| \
          &\leq \ \left(1 - Z_{\epsilon} \right) \left|\hat{h}_{\epsilon}^{*}(z) \right| + \left|Z_{\epsilon} \hat{h}_{\epsilon}^{*}(z) - Z_{\delta} \hat{h}_{\delta}^{*}(z) \right| + \left(1 - Z_{\delta} \right) \left|\hat{h}_{\delta}^{*}(z)\right|,
        \end{align*}
        it will suffice to prove that there is a random variable $X \in (0, \infty)$ such that for all $z \in K$ and all $\epsilon, \delta \in (0, e^{-1})$ with $1 \leq \frac{\delta}{\epsilon} \leq 2$,
        \begin{align}
          \left(1 - Z_{\epsilon} \right) \left|\hat{h}_{\epsilon}^{*}(z) \right| \
          &\leq \ X e^{-(\log \epsilon^{-1})^2/4} \left(\log \epsilon^{-1} \right)^2 \log \frac{1}{\epsilon \log \epsilon^{-1}}, \label{eq:ContinuityOfMollificationZTerms} \\
          \left|Z_{\epsilon} \hat{h}_{\epsilon}^{*}(z) - Z_{\delta} \hat{h}_{\delta}^{*}(z) \right| \
          &\leq \ X \left(\log \delta^{-1} \right)^3 \log \frac{1}{\delta \log \delta^{-1}} \left[\left|\frac{\delta \log \delta^{-1}}{\epsilon \log \epsilon^{-1}} - 1 \right| \vee \left(\frac{\delta}{\epsilon} - 1 \right)\right]. \label{eq:ContinuityOfMollificationMainTerms}
        \end{align}
        Using polar coordinates centered at $z$, we see that
        \begin{align}
          \hat{h}_{\epsilon}^{*}(z) \
          &= \ \int\limits_{0}^{\frepsilon} r h_r(z) \Psi_{\epsilon}^{\id, 0}(r) \, \D r.
          \label{eq:MollificationPolarCoordinates}
        \end{align}
        By \cite[Lemma 2.2]{WeakLQGMetrics}, $X \coloneqq \sup_{z \in K} \sup_{r \in (0, e^{-1})} \frac{|h_r(z)|}{\log r^{-1}} < \infty$ almost surely.
        It follows that
        \begin{align*}
          \left|\hat{h}_{\epsilon}^{*}(z) \right| \
          &\leq \ X \int\limits_{0}^{\frepsilon} r \log r^{-1} \Psi_{\epsilon}^{\id, 0}(r) \, \D r \\
          &= \ \frac{C_1 C_2 X}{Z_{\epsilon} \epsilon^2} \int\limits_{0}^{\frepsilon} r \log r^{-1} \, \D r \\
          &= \ \frac{C_1 C_2 X}{Z_{\epsilon}} \frac{1}{4} \left(\log \epsilon^{-1} \right)^2 \left(2 \log \frac{1}{\frepsilon} + 1 \right),
        \end{align*}
        where $C_1 \coloneqq \max_{w \in \CC} \psi(w)$ and $C_2 \coloneqq \max_{w \in \CC} p_{1/2}(w)$.
        From \eqref{eq:NormalizingConstantEstimate}, we see that when $\epsilon \in (0, e^{-1})$,
        \begin{align*}
          Z_{\epsilon} \
          &\geq \ 1 - e^{-(\log \epsilon^{-1})^2/4} \
          > \ 1 - e^{-1/4}, \\
          1 - Z_{\epsilon} \
          &\leq \ e^{-(\log \epsilon^{-1})^2/4}.
        \end{align*}
        It follows that
        \begin{align*}
          \left(1 - Z_{\epsilon} \right) \left|\hat{h}_{\epsilon}^{*}(z) \right| \
          &\leq \ \frac{C_1 C_2 X}{1 - e^{-1/4}} e^{-(\log \epsilon^{-1})^2/4} \left(\log \epsilon^{-1} \right)^2 \left(2 \log \frac{1}{\epsilon \log \epsilon^{-1}} + 1 \right),
        \end{align*}
        which proves \eqref{eq:ContinuityOfMollificationZTerms} with $\frac{4 C_1 C_2 X}{1 - e^{-1/4}}$ in place of $X$.

        To prove \eqref{eq:ContinuityOfMollificationMainTerms}, assume $1 \leq \frac{\delta}{\epsilon} \leq 2$ and use \eqref{eq:MollificationPolarCoordinates} to see that
        \begin{align}
          \left|Z_{\epsilon} \hat{h}_{\epsilon}^{*}(z) - Z_{\delta} \hat{h}_{\delta}^{*}(z) \right| \
          &\leq \ X \int\limits_{0}^{\delta \log \delta^{-1}} r \log r^{-1} \left|Z_{\epsilon} \Psi_{\epsilon}^{\id, 0}(r) - Z_{\delta} \Psi_{\delta}^{\id, 0}(r) \right| \, \D r.
          \label{eq:ContinuityOfMollificationMainTermsFirstStep}
        \end{align}
        We will use the bound
        \begin{align*}
          \left|Z_{\epsilon} \Psi_{\epsilon}^{\id, 0}(r) - Z_{\delta} \Psi_{\delta}^{\id, 0}(r) \right| \
          &= \ \left|\frac{1}{\epsilon^2} \psi\left(\frac{r}{\epsilon \log \epsilon^{-1}} \right) p_{1/2}\left(\frac{r}{\epsilon} \right) - \frac{1}{\delta^2} \psi\left(\frac{r}{\delta \log \delta^{-1}} \right) p_{1/2}\left(\frac{r}{\delta} \right) \right| \\
          &\leq \ \left|\frac{1}{\epsilon^2} - \frac{1}{\delta^2} \right| \psi\left(\frac{r}{\frepsilon} \right) p_{1/2}\left(\frac{r}{\epsilon} \right) \\
          &\qquad \qquad + \frac{1}{\delta^2} \left|\psi\left(\frac{r}{\frepsilon} \right) - \psi\left(\frac{r}{\delta \log \delta^{-1}} \right)\right| p_{1/2}\left(\frac{r}{\epsilon} \right) \\
          &\qquad \qquad + \frac{1}{\delta^2} \psi\left(\frac{r}{\delta \log \delta^{-1}} \right) \left|p_{1/2}\left(\frac{r}{\epsilon} \right) - p_{1/2}\left(\frac{r}{\delta} \right)\right| \\
          &\leq \ \frac{C_1 C_2}{\delta^2} \left|\frac{\delta^2}{\epsilon^2} - 1 \right| + \frac{C_2 L_1}{\delta^2} \left|\frac{r}{\epsilon \log \epsilon^{-1}} - \frac{r}{\delta \log \delta^{-1}} \right| + \frac{C_1 L_2}{\delta^2} \left|\frac{r}{\epsilon} - \frac{r}{\delta} \right|,
        \end{align*}
        where $L_1 \coloneqq \max_{w \in \CC} |\nabla \psi(w)|$ and $L_2 \coloneqq \max_{w \in \CC} |\nabla p_{1/2}(w)|$.
        Combining this with \eqref{eq:ContinuityOfMollificationMainTermsFirstStep}, we find that when $\frac{\delta}{\epsilon} \leq 2$,
        \begin{align*}
          \left|Z_{\epsilon} \hat{h}_{\epsilon}^{*}(z) - Z_{\delta} \hat{h}_{\delta}^{*}(z) \right| \
          &\leq \ \frac{3 X C_1 C_2}{4} \left(\frac{\delta}{\epsilon} - 1 \right) \left(\log \delta^{-1} \right)^2 \left(2 \log \frac{1}{\delta \log \delta^{-1}} + 1 \right) \\
          &\qquad \qquad + \frac{X C_2 L_1}{9}  (\log \delta^{-1})^{2} \left|\frac{\delta \log \delta^{-1}}{\epsilon \log \epsilon^{-1}} - 1 \right| \left(3 \log \frac{1}{\delta \log \delta^{-1}} + 1 \right) \\
          &\qquad \qquad + \frac{X C_1 L_2}{9} \left(\frac{\delta}{\epsilon} - 1 \right) \left(\log \delta^{-1} \right)^3 \left(3 \log \frac{1}{\delta \log \delta^{-1}} + 1 \right).
        \end{align*}
        Equation \eqref{eq:ContinuityOfMollificationMainTerms} now follows after absorbing constants into $X$.
      \end{proof}
    \end{lemma}

    The main ingredient in the proof of Proposition \ref{prop:MainRegularityConditionHighProbability} is the following.

    \begin{lemma}
      \label{lemma:CoordinateChangedFieldComparison}
      Fix open sets $W \Subset V \Subset U$, $\DerivativeBound > 1$, and $\delta, \zeta \in (0,1)$.
      Then
      \begin{align}
        \lim_{\epsilon \to 0} \sup_{z_0 \in W} \Prob\left\{\sup_{z \in B_{2 \epsilon^{1 - \zeta}}(z_0)} \sup_{\phi \in \confmaps} \left|\langle h, \Psi_{\epsilon}^{\phi, z} - \Psi_{\epsilon/|\phi'(z_0)|}^{\id, z} \rangle \right| > \delta \right\} \
        &= \ 0.
        \label{eq:CoordinateChangedFieldComparisonProbability}
      \end{align}
    \end{lemma}

    Before proving Lemma \ref{lemma:CoordinateChangedFieldComparison}, we will prove the following more quantitative variant of Lemma \ref{lemma:CoordinateChangedFieldComparison} which will be used later to estimate LFPP distances for points at Euclidean distance $\epsilon^{1 - \zeta}$ apart.
    It is also used in the proof of Lemma \ref{lemma:CoordinateChangedFieldComparison}.
    We emphasize that Lemma \ref{lemma:QuantitativeCoordinateChangedFieldComparison} only applies to $\zeta \in (0, 1/3)$ while Lemma \ref{lemma:CoordinateChangedFieldComparison} holds for all $\zeta \in (0, 1)$.
    
    \begin{lemma}
      \label{lemma:QuantitativeCoordinateChangedFieldComparison}
      Fix open sets $W \Subset V \Subset U$, $\DerivativeBound > 1$, $\delta \in (0,1)$, and $\zeta \in (0, 1/3)$.
      There are constants $c = c(W, V, \DerivativeBound, \zeta, \delta) > 0$, $\epsilon_0 = \epsilon_0(W, V, \DerivativeBound, \zeta, \delta) \in (0, 1)$ such that for each $z_0 \in W$ and each $\epsilon \in (0, \epsilon_0)$,
      \begin{align*}
        \Prob\left\{\sup_{z \in B_{2 \epsilon^{1 - \zeta}}(z_0)} \sup_{\phi \in \confmaps} \left|\langle h, \Psi_{\epsilon}^{\phi, z} - \Psi_{\epsilon/|\phi'(z_0)|}^{\id, z} \rangle \right| > \delta \right\} \
        &\leq \ \exp\left(-c \frac{\epsilon^{-2 + 6 \zeta}}{(\log \epsilon^{-1})^5} \right).
      \end{align*}
      \begin{proof}
        The general idea is to view $h$ as a continuous linear functional on the Sobolev space $\cH^1(2 \epsilon^{1 - \zeta} \DD)$, then bound $|\langle h, \Psi_{\epsilon}^{\phi, z} - \Psi_{\epsilon/|\phi'(z_0)|}^{\id, z} \rangle|$ by the norm of the GFF times the norm of $\Psi_{\epsilon}^{\phi, z} - \Psi_{\epsilon/|\phi'(z_0)|}^{\id, z}$.
        The norm of the GFF will be controlled using the Borell-TIS inequality.
        The norm of $\Psi_{\epsilon}^{\phi, z} - \Psi_{\epsilon/|\phi'(z_0)|}^{\id, z}$ will be controlled using Lemma \ref{lemma:GrowthRateOfDistortedBumpFunctions}.
        Here, we will use the Sobolev norm $\|f\|_{\cH^1} = (\|f\|_{L^2} + \|f'\|_{L^2})^{1/2}$ and the dual norm on $\cH^{-1}$.

        Fix an open set $\tilde{W}$ such that $W \Subset \tilde{W} \Subset V$.
        Let $\epsilon_0$ and $C$ be as in Lemma \ref{lemma:GrowthRateOfDistortedBumpFunctions} with $\tilde{W}$ in place of $W$.
        We may assume $\epsilon_0$ is small enough that $\epsilon^{\zeta} \log \epsilon^{-1} < 1$ and $B_{4 \DerivativeBound \frepsilon + 2 \epsilon^{1 - \zeta}}(W) \subset \tilde{W}$ for all $\epsilon \in (0, \epsilon_0)$.
        Fix $z_0 \in W$, $\epsilon \in (0, \epsilon_0)$, and $z \in B_{2 \epsilon^{1 - \zeta}}(z_0)$.
        Note that since $\Psi_{\epsilon}^{\phi, z}$ and $\Psi_{\epsilon/|\phi'(z_0)|}^{\id, z}$ are supported on $B_{4 \DerivativeBound \frepsilon}(z) \subset \tilde{W}$, the support of $\Psi_{\epsilon}^{\phi, z}(\epsilon^{1 - \zeta} \cdot + z_0) - \Psi_{\epsilon/|\phi'(z_0)|}^{\id, z}(\epsilon^{1 - \zeta} \cdot + z_0)$ is contained in $B_{4 \DerivativeBound \epsilon^{\zeta} \log \epsilon^{-1}}(\epsilon^{-(1 - \zeta)}(z - z_0)) \subset 6 \DerivativeBound \DD$.
        Therefore, 
        \begin{align}
          &\left|\left\langle h, \Psi_{\epsilon}^{\phi, z} - \Psi_{\epsilon/|\phi'(z_0)|}^{\id, z} \right\rangle \right| \nonumber \\
          &\qquad = \ \left|\left\langle h(\epsilon^{1 - \zeta} \cdot + z_0), \epsilon^{2(1 - \zeta)} \left(\Psi_{\epsilon}^{\phi, z}(\epsilon^{1 - \zeta} \cdot + z_0) - \Psi_{\epsilon/|\phi'(z_0)|}^{\id, z}(\epsilon^{1 - \zeta} \cdot + z_0) \right) \right\rangle \right| \nonumber \\
          &\qquad \leq \ \|h(\epsilon^{1 - \zeta} \cdot + z_0)\|_{\cH_0^{-1}(6 \DerivativeBound \DD)} \epsilon^{2(1 - \zeta)} \|\Psi_{\epsilon}^{\phi, z}(\epsilon^{1 - \zeta} \cdot + z_0) - \Psi_{\epsilon/|\phi'(z_0)|}^{\id, z}(\epsilon^{1 - \zeta} \cdot + z_0)\|_{\cH_0^1(6 \DerivativeBound \DD)}. \label{eq:FieldNormInequality}
        \end{align}
        We will first bound the $\cH_0^1$ norm on the previous line.
        The idea is that the functions $\Psi_{\epsilon}^{\phi, z}$ and $\Psi_{\epsilon/|\phi'(z_0)|}^{\id, z}$ are supported on $B_{4 \DerivativeBound \frepsilon}(z)$ by Lemma \ref{lemma:DistortedBumpProperties}, so $\Psi_{\epsilon}^{\phi, z}(\epsilon^{1 - \zeta} \cdot + z_0) - \Psi_{\epsilon/|\phi'(z_0)|}^{\id, z}(\epsilon^{1 - \zeta} \cdot + z_0)$ is supported on a disk of radius $4 \DerivativeBound \epsilon^{\zeta} \log \epsilon^{-1}$.
        So we will bound the $L^2$ norm of this function and its derivative by their maximum (using Lemma \ref{lemma:GrowthRateOfDistortedBumpFunctions}) times the area of its support $\pi 16 \DerivativeBound^2 \epsilon^{2 \zeta} (\log \epsilon^{-1})^2$.
        The gradient is handled similarly.
        More precisely, using Lemma \ref{lemma:GrowthRateOfDistortedBumpFunctions},
        \begin{align*}
          &\|\Psi_{\epsilon}^{\phi, z}(\epsilon^{1 - \zeta} \cdot + z_0) - \Psi_{\epsilon/|\phi'(z_0)|}^{\id, z}(\epsilon^{1 - \zeta} \cdot + z_0)\|_{L^2(6 \DerivativeBound \DD)} \\
          &\qquad \qquad \leq \ \left(C^2 \epsilon^{2(-1 - \zeta)} (\log \epsilon^{-1})^2 \pi 16 \DerivativeBound^2 \epsilon^{2 \zeta} (\log \epsilon^{-1})^2 \right)^{1/2} \\
          &\qquad \qquad = \ C \sqrt{\pi} 4 \DerivativeBound \epsilon^{-1} \left(\log \epsilon^{-1} \right)^2, \\
          &\|\nabla (\Psi_{\epsilon}^{\phi, z}(\epsilon^{1 - \zeta} \cdot + z_0) - \Psi_{\epsilon/|\phi'(z_0)|}^{\id, z}(\epsilon^{1 - \zeta} \cdot + z_0)) \|_{L^2(6 \DerivativeBound \DD)} \\
          &\qquad \qquad = \ \epsilon^{1 - \zeta} \|(\nabla \Psi_{\epsilon}^{\phi, z})(\epsilon^{1 - \zeta} \cdot + z_0) - (\nabla \Psi_{\epsilon}^{\id, z})(\epsilon^{1 - \zeta} \cdot + z_0) \|_{L^2(6 \DerivativeBound \DD)} \\
          &\qquad \qquad \leq \ \epsilon^{1 - \zeta} \left( C^2 \epsilon^{2(-2 - \zeta)} \left(\log \epsilon^{-1} \right)^2 \pi 16 \DerivativeBound^2 \epsilon^{2 \zeta} (\log \epsilon^{-1})^2 \right)^{1/2} \\
          &\qquad \qquad = \ C \sqrt{\pi} 4 \DerivativeBound \epsilon^{-1 - \zeta} \left(\log \epsilon^{-1} \right)^2.
        \end{align*}
        It follows that
        \begin{align}
          &\epsilon^{2(1 - \zeta)} \|\Psi_{\epsilon}^{\phi, z}(\epsilon^{1 - \zeta} \cdot + z_0) - \Psi_{\epsilon/|\phi'(z_0)|}^{\id, z}(\epsilon^{1 - \zeta} \cdot + z_0)\|_{\cH_0^{1}(6 \DerivativeBound \DD)} \nonumber \\
          &\qquad \qquad \leq \ \epsilon^{2(1 - \zeta)} \left(C^2 \pi 16 \DerivativeBound \epsilon^{-2} (\log \epsilon^{-1})^4 + C^2 \pi 16 \DerivativeBound \epsilon^{-2 - 2 \zeta} (\log \epsilon^{-1})^4 \right)^{1/2} \nonumber \\
          &\qquad \qquad = \ C \sqrt{\pi} 4 \DerivativeBound \left(\epsilon^{2 - 4 \zeta} + \epsilon^{2 - 6 \zeta} \right)^{1/2} \left(\log \epsilon^{-1} \right)^2 \nonumber \\
          &\qquad \qquad \leq \ C \epsilon^{1 - 3 \zeta} \left(\log \epsilon^{-1} \right)^2, \label{eq:DistortedBumpFunctionH1NormBound}
        \end{align}
        where we have absorbed constants into $C$.
        Therefore,
        \begin{align*}
          &\Prob\left\{\sup_{z \in B_{2 \epsilon^{1 - \zeta}}(z_0)} \sup_{\phi \in \confmaps} \left|\langle h, \Psi_{\epsilon}^{\phi, z} - \Psi_{\epsilon/|\phi'(z_0)|}^{\id, z} \rangle \right| > \delta \right\} \\
          &\qquad \qquad \leq \ \Prob\left\{\|h(\epsilon^{1 - \zeta} \cdot + z_0)\|_{\cH_0^{-1}(6 \DerivativeBound \DD)} > \frac{\delta \epsilon^{3 \zeta - 1}}{C (\log \epsilon^{-1})^2}\right\} \tag*{(by \eqref{eq:FieldNormInequality}, \eqref{eq:DistortedBumpFunctionH1NormBound})} \\
          &\qquad \qquad \leq \ \Prob\left\{\|H\|_{\cH_0^{-1}(6 \DerivativeBound \DD)} > \frac{\delta \epsilon^{3 \zeta - 1}}{2 C (\log \epsilon^{-1})^2} \right\} + \Prob\left\{|h_{\epsilon^{1 - \zeta}}(z_0)| > \frac{\delta \epsilon^{3 \zeta - 1}}{2 C (\log \epsilon^{-1})^2} \right\},
        \end{align*}
        where $H$ is a whole-plane GFF normalized so that $H_1(0) = 0$.
        Since $\|H\|_{\cH_0^{-1}(6 \DerivativeBound \DD)} < \infty$ almost surely \cite[Theorem 1.45]{BerestyckiPowell}, the Borell-TIS inequality \cite{BorellTISBorell} \cite{BorellTISTIS} (see, e.g. \cite[Theorem 2.1.1]{RandomFieldsAndGeometry}) implies that if
        \begin{align*}
          \sigma^2 \
          &\coloneqq \ \sup_{\substack{\varphi \in \cH^1(6 \DerivativeBound \DD) \\ \|\varphi\|_{\cH^1(6 \DerivativeBound \DD)} \leq 1}} \E\left|\langle H, \varphi \rangle \right|^2,
        \end{align*}
        then $\sigma^2 < \infty$ and $\E[\|H\|_{\cH_0^{-1}(6 \DerivativeBound \DD)}] < \infty$, and moreover for each $u > 0$,
        \begin{align*}
          \Prob\left\{\|H\|_{\cH_0^{-1}(6 \DerivativeBound \DD)} > \E\left[\|H\|_{\cH_0^{-1}(6 \DerivativeBound \DD)} \right] + u \right\} \
          &\leq \ \exp\left(-\frac{u^2}{2 \sigma^2} \right).
        \end{align*}
        Apply this with $u$ equal to 
        \begin{align*}
          u(\epsilon) \
          &\coloneqq \ -\E\left[\|H\|_{\cH_0^{-1}(6 \DerivativeBound \DD)} \right] + \frac{\delta \epsilon^{3 \zeta - 1}}{2 C (\log \epsilon^{-1})^2},
        \end{align*}
        which is positive for $\epsilon$ small enough because $3 \zeta - 1 < 0$.
        We obtain
        \begin{align*}
          \Prob\left\{\|H\|_{\cH_0^{-1}(6 \DerivativeBound \DD)} > \frac{\delta \epsilon^{3 \zeta - 1}}{2 C (\log \epsilon^{-1})^2} \right\} \
          &\leq \ \exp\left(-\frac{u(\epsilon)^2}{2 \sigma^2} \right).
        \end{align*}
        Since $h_{\epsilon^{1 - \zeta}}(z_0)$ is a mean-zero Gaussian with variance bounded by a constant $c_0$ depending only on $W$ plus $\log \epsilon^{-(1 - \zeta)}$, standard estimates for the tail probabilities of Gaussian random variables show that
        \begin{align*}
          &\Prob\left\{|h_{\epsilon^{1 - \zeta}}(z_0)| > \frac{\delta \epsilon^{3 \zeta - 1}}{2 C (\log \epsilon^{-1})^2} \right\} \\
          &\qquad \qquad \leq \ \frac{2 C (\log \epsilon^{-1})^2 \sqrt{c_0 + \log \epsilon^{-(1 - \zeta)}}}{\sqrt{2 \pi} \delta} \epsilon^{1 - 3 \zeta} \exp\left(-\frac{\delta^2 \epsilon^{6 \zeta - 2}}{8 C^2 (\log \epsilon^{-1})^4 (c_0 + \log \epsilon^{-(1 - \zeta)})} \right).
        \end{align*}
        Since $1 - 3 \zeta > 0$, we see that after shrinking $\epsilon_0$, we will have $\frac{2 C (\log \epsilon^{-1})^2 \sqrt{c_0 + \log \epsilon^{-(1 - \zeta)}}}{\sqrt{2 \pi} \delta} \epsilon^{1 - 3 \zeta} < 1$.
        The lemma now follows with constant
        \begin{align*}
          c \
          &\coloneqq \ \min\left\{\frac{1}{2 \sigma^2} \inf_{\epsilon \in (0, \epsilon_0)} \frac{u(\epsilon)^2}{\frac{\epsilon^{-2 + 6 \zeta}}{(\log \epsilon^{-1})^5}}, \frac{\delta^2}{8 C^2} \inf_{\epsilon \in (0, \epsilon_0)} \frac{\log \epsilon^{-1}}{c_0 + \log \epsilon^{-(1 - \zeta)}} \right\}.
        \end{align*}
      \end{proof}
    \end{lemma}

    \begin{proof}[Proof of Lemma \ref{lemma:CoordinateChangedFieldComparison}]
      The idea is roughly as follows.
      Lemma \ref{lemma:QuantitativeCoordinateChangedFieldComparison} implies Lemma \ref{lemma:CoordinateChangedFieldComparison} when $\zeta \in (0, 1/3)$, but we would like the result to hold even when $\zeta$ is close to $1$.
      To deal with this, we will fix some $\zeta_{-} \in (0, 1/3)$, say $\zeta_{-} = 1/4$, and use Lemma \ref{lemma:QuantitativeCoordinateChangedFieldComparison} to estimate $\langle h, \Psi_{\epsilon}^{\phi, z} - \Psi_{\epsilon/|\phi'(z_1)|}^{\id, z} \rangle$ for some $z_1$ with $|z - z_1| \approx \epsilon^{1 - \zeta_{-}}$, then use distortion estimates and Lemma \ref{lemma:MollifiedFieldComparison} to estimate $\langle h, \Psi_{\epsilon/|\phi'(z_1)|}^{\id, z} - \Psi_{\epsilon/|\phi'(z_0)|}^{\id, z} \rangle$ when $|z_1 - z_0| \approx \epsilon^{1 - \zeta}$.

      Let us make this precise.
      Fix $0 < \zeta_{-} < \frac{1}{3} \w \zeta$. 
      For each $z_0 \in W$, the probability in \eqref{eq:CoordinateChangedFieldComparisonProbability} is bounded by the sum of the probabilities
      \begin{align}
        &\Prob\left\{\sup_{z \in B_{2 \epsilon^{1 - \zeta}}(z_0)} \sup_{z_1 \in B_{2 \epsilon^{1 - \zeta_{-}}}(z) \cap (\frac{\epsilon^{1 - \zeta_{-}}}{4} \ZZ^2)} \sup_{\phi \in \confmaps} \left|\langle h, \Psi_{\epsilon}^{\phi, z} - \Psi_{\epsilon/|\phi'(z_1)|}^{\id, z} \rangle \right| > \frac{\delta}{2} \right\}, \label{eq:CoordinateChangedFieldComparisonExponentialTerm} \\
        &\Prob\left\{\sup_{z \in B_{2 \epsilon^{1 - \zeta}}(z_0)} \sup_{z_1 \in B_{2 \epsilon^{1 - \zeta_{-}}}(z) \cap (\frac{\epsilon^{1 - \zeta_{-}}}{4} \ZZ^2)} \sup_{\phi \in \confmaps} \left|\langle h, \Psi_{\epsilon/|\phi'(z_1)|}^{\id, z} - \Psi_{\epsilon/|\phi'(z_0)|}^{\id, z} \rangle \right| > \frac{\delta}{2} \right\}. \label{eq:CoordinateChangedFieldComparisonScalingTerm}
      \end{align}
      By a union bound, we see that \eqref{eq:CoordinateChangedFieldComparisonExponentialTerm} is at most
      \begin{align*}
        \sum_{z_1 \in B_{2 \epsilon^{1 - \zeta} + 2 \epsilon^{1 - \zeta_{-}}}(z_0) \cap (\frac{\epsilon^{1 - \zeta_{-}}}{4} \ZZ^2)} \Prob\left\{\sup_{z \in B_{2 \epsilon^{1 - \zeta_{-}}}(z_1)} \sup_{\phi \in \confmaps} \left|\langle h, \Psi_{\epsilon}^{\phi, z} - \Psi_{\epsilon/|\phi'(z_1)|}^{\id, z} \rangle \right| > \frac{\delta}{2} \right\},
      \end{align*}
      which converges to $0$ uniformly over $z_0 \in W$ by Lemma \ref{lemma:QuantitativeCoordinateChangedFieldComparison}.
      To bound \eqref{eq:CoordinateChangedFieldComparisonScalingTerm}, let $r \coloneqq \dist(\partial W, \partial V)$, and note that by distortion estimates, if $\phi \in \confmaps$ and $z_0 \in W$ and $z_1 \in V$ with $|z_0 - z_1| < r$, then
      \begin{align}
        \frac{|\phi'(z_1)|}{|\phi'(z_0)|} \
        &\in \ \left[\frac{1 - r^{-1} |z_0 - z_1|}{(1 + r^{-1} |z_0 - z_1|)^3}, \frac{1 + r^{-1} |z_0 - z_1|}{(1 - r^{-1} |z_0 - z_1|)^3} \right].
        \label{eq:EpsilonRatioDistortionEstimate}
      \end{align}
      In particular, if $z_0 \in W$ and $z_1 \in V$ with $|z_0 - z_1| < 2 \epsilon^{1 - \zeta} + 2 \epsilon^{1 - \zeta_{-}}$, then 
      \begin{align*}
        1 - C \epsilon^{1 - \zeta} \
        &\leq \ \frac{|\phi'(z_1)|}{|\phi'(z_0)|} \
        \leq \ 1 + C \epsilon^{1 - \zeta}
      \end{align*}
      with constant $C$ depending only on $r$.
      By Lemma \ref{lemma:MollifiedFieldComparison} (with $s = 1/|\phi'(z_0)|$ and $t = 1/|\phi'(z_1)|$), \eqref{eq:CoordinateChangedFieldComparisonScalingTerm} converges to $0$ as $\epsilon \to 0$ uniformly over $z_0 \in W$.
    \end{proof}

    For the next lemma, note that since $Z_{\epsilon}^{-1} \psi_{\epsilon} p_{\epsilon^2/2}$ is an approximate identity, by \ref{DistortedBump:LogMollification} of Lemma \ref{lemma:DistortedBumpProperties}, we have $\langle -\log|\phi'|, \Psi_{\epsilon}^{\phi, z} \rangle \to -\log|\phi'(z)|$ as $\epsilon \to 0$.
    The next lemma shows that this convergence is uniform over $\phi \in \confmaps$.
    This is used to deal with the $\log|(\phi^{-1})'|$ terms which arise in the coordinate change formula.

    \begin{lemma}
      \label{lemma:LogMollificationConvergence}
      If $V \Subset U$ are open sets and $\DerivativeBound > 1$, then
      \begin{align*}
        \lim_{\epsilon \to 0} \sup_{z \in V} \sup_{\phi \in \confmaps} \left|\left\langle -\log|\phi'|, \Psi_{\epsilon}^{\phi, z} \right\rangle - (-\log|\phi'(z)|) \right| \
        &= \ 0.
      \end{align*}
      \begin{proof}
        We have
        \begin{align}
          &\langle -\log|\phi'|, \Psi_{\epsilon}^{\phi, z} \rangle \nonumber \\
          &\qquad \qquad = \ -\int\limits_{\phi^{-1}(B_{\frepsilon}(\phi(z)))} \log|\phi'(\omega)| \Psi_{\epsilon}^{\phi, z}(\omega) \, \D A(\omega) \nonumber \\
          &\qquad \qquad = \ -\int\limits_{B_{\frepsilon}(\phi(z))} \log|\phi'(\phi^{-1}(w))| \Psi_{\epsilon}^{\id, \phi(z)}(w) \, \D A(w) \label{eq:LogTermsCoordinateChange} \\
          &\qquad \qquad = \ \int\limits_{B_{\frepsilon}(\phi(z))} \left(-\log|\phi'(\phi^{-1}(w))| + \log|\phi'(z)| \right) \Psi_{\epsilon}^{\id, \phi(z)}(w) \, \D A(w) - \log|\phi'(z)|, 
          \label{eq:LogMollificationInitialEstimate}
        \end{align}
        where the last step used the fact that $\int_{B_{\frepsilon}} \Psi_{\epsilon}^{\id, \phi(z)}(w) \, \D A(w) = 1$.
        Note that the term $|(\phi^{-1})'(w)|^2 = 1/|\phi'(\omega)|^2$ from the change of variables formula used in \eqref{eq:LogTermsCoordinateChange} cancels with the $|\phi'(\omega)|^2$ in the definition of $\Psi_{\epsilon}^{\phi, z}(\omega)$.
        Assuming $\phi \in \confmaps$ and $z \in V$, the Koebe-$\frac{1}{4}$ theorem implies $\phi^{-1}(B_{\frepsilon}(\phi(z))) \subset B_{4 \DerivativeBound \frepsilon}(z)$.
        Therefore,
        \begin{align}
          \sup_{w \in B_{\frepsilon}(\phi(z))} \left|-\log|\phi'(\phi^{-1}(w))| + \log|\phi'(z)| \right| \
          &\leq \ \sup_{\omega \in B_{4 \DerivativeBound \frepsilon}(z)} \left|-\log|\phi'(\omega)| + \log|\phi'(z)| \right| \nonumber \\
          &= \ \sup_{\omega \in B_{4 \DerivativeBound \frepsilon}(z)} \left|\log\left|\frac{\phi'(\omega)}{\phi'(z)} \right|\right|. \label{eq:LogMollificationErrorEstimate}
        \end{align}
        By distortion estimates, if $r \coloneqq \dist(\partial V, \partial U)$, $z \in V$, and $\omega \in B_{4 \DerivativeBound \frepsilon}(z)$ with $\epsilon$ small enough that $4 \DerivativeBound \frepsilon < r$, then
        \begin{align}
          \left|\frac{\phi'(\omega)}{\phi'(z)} \right| \
          &\in \ \left[\frac{1 - \frac{4 \DerivativeBound \frepsilon}{r}}{(1 + \frac{4 \DerivativeBound \frepsilon}{r})^3}, \frac{1 + \frac{4 \DerivativeBound \frepsilon}{r}}{(1 - \frac{4 \DerivativeBound \frepsilon}{r})^3} \right]. \label{eq:LogMollificationDistortion}
        \end{align}
        Combine \eqref{eq:LogMollificationInitialEstimate}, \eqref{eq:LogMollificationErrorEstimate}, and \eqref{eq:LogMollificationDistortion} to prove the lemma.
      \end{proof}
    \end{lemma}

    The following will be used to control the scaling ratios which arise in \eqref{eq:LFPPScaling}. 
    \begin{lemma}
      \label{lemma:RegularlyVaryingScalingConstants}
      For any compact subinterval of $I$ of $(0, \infty)$,
      \begin{align*}
        \lim_{\epsilon \to 0} \frac{\fa_{C \epsilon}}{\fa_{\epsilon}} \
        &= \ C^{1 - \xi Q}
      \end{align*}
      uniformly over $C \in I$.
      \begin{proof}
        It is shown in \cite[Corollary 1.11]{ExistenceAndUniqueness} that when $\xi < \xi_{\text{crit}}$, $\epsilon \mapsto \fa_{\epsilon}$ is regularly varying with exponent $1 - \xi Q$.
        In fact, the proof of \cite[Corollary 1.11]{ExistenceAndUniqueness} applies even when $\xi \geq \xi_{\text{crit}}$ since it follows from convergence in probability of LFPP (which is shown when $\xi \geq \xi_{\text{crit}}$ in \cite[Theorem 1.3]{UniquenessOfSupercriticalLQGMetrics}).

        Regular variation of $\epsilon \mapsto \fa_{\epsilon}$ means $\lim_{\epsilon \to 0} \frac{\fa_{C \epsilon}}{\fa_{\epsilon}} = C^{1 - \xi Q}$ for any fixed $C \in (0, \infty)$.
        It is a standard result that for any regularly varying function, this limit is uniform over compact subsets of $(0, \infty)$ \cite{RegularVariationMbleCase} (see, e.g. \cite[Theorem 1.2.1]{RegularVariation}).
      \end{proof}
    \end{lemma}

    \begin{proof}[Proof of Proposition \ref{prop:MainRegularityConditionHighProbability}]
      Apply Lemmas \ref{lemma:CoordinateChangedFieldComparison}, \ref{lemma:LogMollificationConvergence}, and \ref{lemma:RegularlyVaryingScalingConstants} to choose $\epsilon_0 \in (0, 1)$ such that for each $\epsilon \in (0, \epsilon_0)$ and each $z_0 \in W$, the following are true:
      \begin{align}
        &P\left\{\sup_{z \in B_{2 \epsilon^{1 - \zeta}}(z_0)} \sup_{\phi \in \confmaps} \left|\left\langle h, \Psi_{\epsilon}^{\phi, z} - \Psi_{\epsilon/|\phi'(z_0)|}^{\id, z} \right\rangle \right| > \delta \right\} \
        < \ 1 - p, \label{eq:MainRegularityConditionFieldMollification} \\
        &\left|\frac{\fa_{t \epsilon}}{\fa_{\epsilon}} - t^{1 - \xi Q} \right| \
        \leq \ \frac{\delta}{1 + \delta} \DerivativeBound^{1 - \xi Q} \ \forall t \in [\DerivativeBound^{-1}, \DerivativeBound], \label{eq:MainRegularityConditionScalingConstants} \\
        &\sup_{z \in V} \sup_{\phi \in \confmaps} \left|\langle -\log|\phi'|, \Psi_{\epsilon}^{\phi, z} \rangle - (-\log|\phi'(z)|) \right| \
        < \ \delta. \label{eq:MainRegularityConditionLogMollification}
      \end{align}
      Note that \eqref{eq:MainRegularityConditionScalingConstants} implies
      \begin{align}
        (1 + \delta)^{-1} t^{1 - \xi Q} \
        &\leq \ \frac{\fa_{t \epsilon}}{\fa_{\epsilon}} \
        \leq \ (1 + \delta) t^{1 - \xi Q} \ \forall t \in [\DerivativeBound^{-1}, \DerivativeBound].
        \label{eq:MainRegularityConditionScalingConstantsRearranged}
      \end{align}
      Let $r = \dist(\partial W, \partial V)$.
      We may assume $\epsilon_0$ is small enough that $B_{2 \epsilon^{1 - \zeta}}(W) \subset V$ and
      \begin{align}
        \left(\frac{1 + r^{-1} 2 \epsilon_0^{1 - \zeta}}{1 - r^{-1} 2 \epsilon_0^{1 - \zeta}} \right)^{3(\xi Q - 1)} \
        &< \ 1 + \delta.
        \label{eq:MainRegularityConditionDistortionEstimateComparison}
      \end{align}
      By the distortion theorem, for $\epsilon \in (0, \epsilon_0)$, $z_0 \in W$, $\phi \in \confmaps$, and $\omega \in B_{2 \epsilon^{1 - \zeta}}(z_0)$,
      \begin{align}
        \frac{1 - r^{-1} 2 \epsilon^{1 - \zeta}}{(1 + r^{-1} 2 \epsilon^{1 - \zeta})^3} \
        &\leq \ \left|\frac{\phi'(\omega)}{\phi'(z_0)} \right| \
        \leq \ \frac{1 + r^{-1} 2 \epsilon^{1 - \zeta}}{(1 - r^{-1} 2 \epsilon^{1 - \zeta})^3}.
        \label{eq:MainRegularityConditionDistortionEstimate}
      \end{align}
      It follows that for each $\epsilon \in (0, \epsilon_0)$ and each $z_0 \in W$, it holds with probability at least $p$ that for each $\phi \in \confmaps$ and each piecewise $C^1$ path $P$ in $B_{2 \epsilon^{1 - \zeta}}(z_0)$,
      \begin{align*}
        &\fa_{\epsilon}^{-1} \int\limits_{0}^{1} e^{\xi (h \circ \phi^{-1} + Q \log|(\phi^{-1})'|) \ast (Z_{\epsilon}^{-1} \psi_{\epsilon} p_{\epsilon^2/2})(\phi(P(t)))} |(\phi \circ P)'(t)| \, \D t \\
        &= \ \fa_{\epsilon/|\phi'(z_0)|}^{-1} \int\limits_{0}^{1} e^{\xi \langle h, \Psi_{\epsilon}^{\phi, P(t)} - \Psi_{\epsilon/|\phi'(z_0)|}^{\id, P(t)} \rangle} e^{\xi Q (\langle -\log|\phi'|, \Psi_{\epsilon}^{\phi, P(t)} \rangle + \log|\phi'(P(t))|)} \\
        &\qquad \qquad \cdot \frac{\fa_{\epsilon/|\phi'(z_0)|}}{\fa_{\epsilon}} |\phi'(P(t))|^{1 - \xi Q} e^{\xi \hat{h}_{\epsilon/|\phi'(z_0)|}^{*}(P(t))} |P'(t)| \, \D t \\
        &\leq \ e^{\xi \delta} e^{\xi Q \delta} \left(1 + \delta \right) \fa_{\epsilon/|\phi'(z_0)|}^{-1} \int\limits_{0}^{1} \left|\frac{\phi'(P(t))}{\phi'(z_0)} \right|^{1 - \xi Q} e^{\xi \hat{h}_{\epsilon/|\phi'(z_0)|}^{*}(P(t))} |P'(t)| \, \D t \tag*{(by \eqref{eq:MainRegularityConditionFieldMollification}, \eqref{eq:MainRegularityConditionScalingConstantsRearranged}, \eqref{eq:MainRegularityConditionLogMollification})} \\
        &\leq \ e^{\xi \delta} e^{\xi Q \delta} \left(1 + \delta \right)^2 \fa_{\epsilon/|\phi'(z_0)|}^{-1} \int\limits_{0}^{1} e^{\xi \hat{h}_{\epsilon/|\phi'(z_0)|}^{*}(P(t))} |P'(t)| \, \D t, \tag*{(by \eqref{eq:MainRegularityConditionDistortionEstimateComparison}, \eqref{eq:MainRegularityConditionDistortionEstimate})}.
      \end{align*}
      This proves the upper bound in the proposition statement with $e^{\xi \delta} e^{\xi Q \delta} (1 + \delta)^2$ in place of $1 + \delta$.
      The lower bound is done analogously.
    \end{proof}

  \section{Comparison of coordinate-changed LFPP and LQG metric at large scales}
    \label{section:LargeScale}
    In this section, we present a multiscale argument to improve the small-scale pathwise estimate of Proposition \ref{prop:MainRegularityConditionHighProbability} to a global comparison between $\locLFPP[\epsilon][\hphi](\phi(\cdot), \phi(\cdot))$ and $D_h$.
    The argument is similar to those of \cite[Theorem 1.6]{LocalMetricsOfTheGFF}, \cite[Theorem 1.8]{UpToConstants}, \cite[Section 3]{ConformalCovariance}, and \cite[Section 3]{AlmostSureConvergence}, but with a key difference that those proofs work with only two metrics at a time, while we must consider infinite family of metrics $\locLFPP[\epsilon][\hphi]$ as $\phi$ ranges over all conformal maps $U \to \phi(U)$, and our estimates must be uniform over all $\phi$.
    The general approach is to show that when $\epsilon$ is small, $\locLFPP[\epsilon][\hphi](\phi(\cdot), \phi(\cdot))$ and $D_h$ are, roughly speaking, Lipschitz equivalent with a Lipschitz constant which a priori may be large, but crucially does not depend on $\phi$.
    This is achieved using Proposition \ref{prop:MainRegularityConditionHighProbability} to upper-bound $\locLFPP[\epsilon][\hphi](\phi(\cdot), \phi(\cdot))$ by $\locLFPP[\epsilon/|\phi'(z_0)|][h]$, then using almost sure convergence of LFPP to compare $\LFPP[\epsilon/|\phi'(z_0)|][h]$ with $D_h$.
    Since the comparison of $\locLFPP[\epsilon][\hphi](\phi(\cdot), \phi(\cdot))$ and $\locLFPP[\epsilon/|\phi'(z_0)|][h]$ is only valid on small scales, we will use local independence properties of the GFF (see Lemma \ref{lemma:IndependenceAcrossConcentricAnnuli} below) to prove that there are many annuli on which this small scale estimate is valid, then string together paths around such annuli to obtain a global Lipschitz condition.
    Our result will not be a true Lipschitz equivalence since $D_h$ is fractal while the metrics $\locLFPP[\epsilon][\hphi](\phi(\cdot), \phi(\cdot))$ are smooth; instead, we will prove an approximate Lipschitz condition where distances between small Euclidean balls using one metric are comparable to distances between points of the other.
    See Proposition \ref{prop:InitialLipschitzCondition} for a more precise statement.
    This part of the argument is most similar to \cite[Theorem 1.8]{UpToConstants}, and is carried out in Section \ref{section:InitialLipschitzCondition} below.

    The Lipschitz constant we obtain may be quite large since the stringing together of paths relies on upper-bounding the distance around a narrow annulus by a constant times the distance across (similar to RSW arguments in percolation theory).
    For this Lipschitz condition to be useful, we need the Lipschitz constant to be close to $1$.
    We will do this using a similar style of argument, where we show that on small scales, the Lipschitz constant is close to $1$.
    This roughly follows from the fact that $\delta$ in Proposition \ref{prop:MainRegularityConditionHighProbability} can be made close to $0$ and LFPP converges to $D_h$.
    Like the previous argument, we can prove many annuli exist where this better Lipschitz condition holds.
    Crucially, we will prove that there are enough annuli that a uniformly positive proportion of any geodesic between two points is comprised of segments contained in one of these annuli, where the Lipschitz constant is close to $1$.
    The remaining segments of the geodesic have $\locLFPP[\epsilon][\hphi](\phi(\cdot), \phi(\cdot))$- and $D_h$-lengths comparable using the initial Lipschitz constant from Proposition \ref{prop:InitialLipschitzCondition}.
    Since the proportion contained in ``good annuli'' is universal, independent of any parameters, the result is a global Lipschitz condition with constant closer to $1$ than we started with.
    By iterating this argument, we can make the constant as close to $1$ as we wish.
    A caveat is that, again, we are comparing distances between balls of radius $\epsilon^{1 - \zeta}$ with distances between points, and each step requires making $\zeta$ closer to $0$.
    This argument is similar to \cite[Section 3]{AlmostSureConvergence}, which in turn is based on \cite[Section 3]{ConformalCovariance}.
    This is the content of Section \ref{section:ImprovingTheLipschitzConstant} below, and the main result of that section is Proposition \ref{prop:ImprovingTheLipschitzConstant}.

    One issue with Proposition \ref{prop:ImprovingTheLipschitzConstant} is that the new Lipschitz constant depends on terms of the form $\ratios$ arising from \eqref{eq:LFPPScaling}.
    For a fixed $r$, $\ratios$ should be close to $1$ for small $\epsilon$ by regular variation of $\epsilon \mapsto \fa_{\epsilon}$, but we will take $r \approx \epsilon^{1 - \zeta}$ for $\zeta \in (0, 1)$, and need $\ratios[\epsilon t]$ to be close to $1$ for all $t$ in a compact interval (roughly, $t = 1/|\phi'(x)|$ for $x \in U$ and $\phi$ a conformal map).
    The proof of \cite[Lemma 3.11]{AlmostSureConvergence} with superficial changes shows that there are many radii $r$ for which $\ratios[\epsilon t]$ is close to $1$ for all $t$ in a compact interval.
    We give the precise statement in Lemma \ref{lemma:GoodScalingRatios}.

    Finally, using the result of Section \ref{section:ImprovingTheLipschitzConstant} and Lemma \ref{lemma:GoodScalingRatios}, we will prove that all of the coordinate-changed LFPP metrics converge simultaneously to $D_h$ on a neighbourhood of the diagonal in $U$.
    This is done in Section \ref{section:ConvergenceOnDiagonal}.

    Throughout this section, we fix $\xi < \xi_{\text{crit}}$.

    \subsection{Initial Lipschitz condition}
      \label{section:InitialLipschitzCondition}
      Fix $\xi < \xi_{\text{crit}}$.
      The goal of this section is to prove the following.
      \begin{prop}
        \label{prop:InitialLipschitzCondition}
        Fix $\xi > 0$, $\DerivativeBound > 1$, connected open sets $W \Subset V \Subset U$ and $\zeta_{+}, \zeta \in (0,1)$ with $1 - \zeta_{+} < \frac{1}{2} (1 - \zeta)$.
        There exists $C_0 = C_0(\DerivativeBound, \zeta) > 0$ such that the following is true. 
        With probability $1 - O_{\epsilon}(\epsilon^{2(1 - \zeta)})$ as $\epsilon \to 0$, for all $z,w \in W$,
        \begin{align}
          &\sup_{\phi \in \confmaps} \locLFPP[\epsilon][\hphi]\left(\phi(B_{2 \epsilon^{1 - \zeta_{+}}}(z)), \phi(B_{2 \epsilon^{1 - \zeta_{+}}}(w)); \phi(B_{2 \epsilon^{1 - \zeta_{+}}}(W)) \right) \nonumber \\
          &\qquad \qquad \leq \ C_0 D_h\left(z,w; W \right), \label{eq:InitialLipschitzConditionSupLeqLQGMetric} \\
          &D_h\left(B_{2 \epsilon^{1 - \zeta_{+}}}(z), B_{2 \epsilon^{1 - \zeta_{+}}}(w); B_{2 \epsilon^{1 - \zeta_{+}}}(W) \right) \nonumber \\
          &\qquad \qquad \leq \ C_0 \inf_{\phi \in \confmaps} \locLFPP[\epsilon][\hphi]\left(\phi(z), \phi(w); \phi(W) \right). \label{eq:InitialLipschitzConditionLQGMetricLeqInf}
        \end{align}
      \end{prop}
  
      Note that the constant $C_0$ may be large.
      It will be used as the initial condition in an iterative argument to prove a stronger Lipschitz condition with Lipschitz constant close to $1$ in Section \ref{section:ImprovingTheLipschitzConstant}.

      The proof of Proposition \ref{prop:InitialLipschitzCondition} will be done in two steps.
      We will first prove the following weaker version of Proposition \ref{prop:InitialLipschitzCondition}.
  
      \begin{lemma}
        \label{lemma:WeakInitialLipschitzCondition}
        Fix $\DerivativeBound > 1$, connected open sets $W \Subset V \Subset U$, and $\zeta_{+}, \zeta \in (0,1)$ with $1 - \zeta_{+} < \frac{1}{2} (1 - \zeta)$.
        There exists $C_0 > 0$ such that the following is true. 
        For each $\epsilon \in (0, 1)$, choose $\cR_{\epsilon} \subset (\epsilon^{1 - \zeta}, \epsilon^{1 - \zeta_{+}}) \cap \{8^{-j}\}_{j \in \NN}$ with $\# \cR_{\epsilon} \geq \frac{1}{3} \# ((\epsilon^{1 - \zeta}, \epsilon^{1 - \zeta_{+}}) \cap \{8^{-j}\}_{j \in \NN})$.
        With probability $1 - O_{\epsilon}(\epsilon^{2(1 - \zeta)})$ as $\epsilon \to 0$, for all $z,w \in W$,
        \begin{align}
          &\sup_{\phi \in \confmaps} \locLFPP[\epsilon][\hphi]\left(\phi(B_{2 \epsilon^{1 - \zeta_{+}}}(z)), \phi(B_{2 \epsilon^{1 - \zeta_{+}}}(w)); \phi(B_{2 \epsilon^{1 - \zeta_{+}}}(W)) \right) \nonumber \\
          &\qquad \qquad \leq \ C_0 \sup_{r \in \cR_{\epsilon}} \sup_{t \in [\DerivativeBound^{-1}, \DerivativeBound]} \ratios[\epsilon t] D_h\left(z,w; W \right), \label{eq:WeakInitialLipschitzConditionSupLeqLQGMetric} \\
          &D_h\left(B_{2 \epsilon^{1 - \zeta_{+}}}(z), B_{2 \epsilon^{1 - \zeta_{+}}}(w); B_{2 \epsilon^{1 - \zeta_{+}}}(W) \right) \nonumber \\
          &\qquad \qquad \leq \ C_0 \sup_{r \in \cR_{\epsilon}} \sup_{t \in [\DerivativeBound^{-1}, \DerivativeBound]} \invratios[\epsilon t] \inf_{\phi \in \confmaps} \locLFPP[\epsilon][\hphi]\left(\phi(z), \phi(w); \phi(W) \right). \label{eq:WeakInitialLipschitzConditionLQGMetricLeqInf}
        \end{align}
      \end{lemma}
  
      After proving Lemma \ref{lemma:WeakInitialLipschitzCondition}, we will show that the sets $\cR_{\epsilon}$ can be chosen so that the scaling ratio terms $\ratios[\epsilon t]$ and $\invratios[\epsilon t]$ are bounded as functions of $\epsilon$.
      This implies Proposition \ref{prop:InitialLipschitzCondition}.
  
      The key to proving Lemma \ref{lemma:WeakInitialLipschitzCondition} is the following, which is a special case of \cite[Lemma 3.1]{LocalMetricsOfTheGFF}.
  
      \begin{lemma}
        \label{lemma:IndependenceAcrossConcentricAnnuli}
        Fix $0 < s_1 < s_2 < 1$ and $x \in \CC$.
        Let $(r_k)_{k=1}^{\infty}$ be a decreasing sequence of positive numbers such that $r_{k+1}/r_k \leq s_1$ for each $k$, and let $(E_{r_k})_{k=1}^{\infty}$ be events such that $E_{r_k} \in \sigma\{(h - h_{r_k}(x))|_{\AA_{s_1 r_k, s_2 r_k}(x)}\}$ for each $k$.
        Then for each $a > 0$, there exists $p = p(a, s_1, s_2) \in (0,1)$ and $c = c(a, s_1, s_2) > 0$ such that if $P[E_{r_k}] \geq p$ for all $k \in \NN$, then
        \begin{align*}
          P\left\{E_{r_k} \text{ occurs for at least one } k \leq K \right\} \
          &\geq \ 1 - c e^{-a K} \ \forall K \in \NN.
        \end{align*}
      \end{lemma}
  
      We will apply Lemma \ref{lemma:IndependenceAcrossConcentricAnnuli} to the following events.
  
      \begin{defn}
        \label{defn:InitialLipschitzConditionEvents}
        For open sets $V \Subset U$, $C > 1$, $\epsilon \in (0, 1)$, $\DerivativeBound > 1$, $x \in V$, $r > 0$ with $B_{2r}(x) \subset U$, let $E_{r, \epsilon}(x) = E_{r, \epsilon}(x; C, V, U, \DerivativeBound)$ be the event that
        \begin{align*}
          &\sup_{\phi \in \confmaps} \locLFPP[\epsilon][\hphi]\left(\text{around } \phi\left(\AA_{\frac{3r}{4}, \frac{7r}{4}}(x) \right)\right) \\
          &\qquad \qquad \leq \ C \sup_{t \in [\DerivativeBound^{-1}, \DerivativeBound]} \ratios[\epsilon t] D_h\left(\partial B_{\frac{3r}{4}}(x), \partial B_r(x) \right), \\
          &D_h\left(\text{around } \AA_{\frac{3r}{4}, \frac{7r}{4}}(x) \right) \\
          &\qquad \qquad \leq \ C \sup_{t \in [\DerivativeBound^{-1}, \DerivativeBound]} \invratios[\epsilon t] \inf_{\phi \in \confmaps} \locLFPP[\epsilon][\hphi]\left(\phi\left(\partial B_{\frac{3r}{4}}(x) \right), \phi\left(\partial B_r(x) \right)\right).
        \end{align*}
      \end{defn}
  
      The ratios $\ratios[\epsilon t]$ in Definition \ref{defn:InitialLipschitzConditionEvents} arise from the LFPP scaling relation \eqref{eq:LFPPScaling}.
      Before establishing measurability of $E_{r, \epsilon}(x)$, let us first show that the supremum and infimum over $\phi \in \confmaps$ in Definition \ref{defn:InitialLipschitzConditionEvents} can be taken over a countable set of conformal maps.
  
      \begin{lemma}
        \label{lemma:CountableDenseSetOfConformalMaps}
        There is a (deterministic) countable set $\lambda_{\DerivativeBound}(V,U) \subset \confmaps$ such that almost surely, for each $\phi \in \confmaps$, each $\epsilon \in (0, 1)$, there is a sequence $(\phi_n)_{n=1}^{\infty} \subset \lambda_{\DerivativeBound}(V,U)$ such that
        \begin{align*}
          \lim_{n \to \infty} \sup_{P \subset V} \frac{\Len(\phi \circ P; \locLFPP[\epsilon][\hphi])}{\Len(\phi_n \circ P; \locLFPP[\epsilon][h \circ \phi_n^{-1} + Q \log|(\phi_n^{-1})'|])} \
          &= \ 1,
        \end{align*}
        where the supremum is over all nontrivial piecewise $C^1$ paths $P$ contained in $V$.
        \begin{proof}
          $C^0(\overline{V}, \CC)$ and $C^1(\overline{V} \times \overline{B_{4 \DerivativeBound \frepsilon}(V)}, \RR)$ are separable normed vector spaces with norms $\|\cdot\|_{C^0}$ and $\|\cdot\|_{C^0} + \|\nabla \cdot\|_{C^0}$ respectively.
          Therefore, $C^0(\overline{V}, \CC) \times C^1(\overline{V} \times \overline{B_{4 \DerivativeBound \frepsilon}(V)}, \RR)$ is separable with respect to the norm $(f,g) \mapsto \|f\|_{C^0} + \|g\|_{C^0} + \|\nabla g\|_{C^0}$.
          The set
          \begin{align}
            \left\{\overline{V} \times \overline{V} \times \overline{B_{4 \DerivativeBound \frepsilon}(V)} \ni (x,z,w) \to (\phi'(x), \Psi_{\epsilon}^{\phi, z}(w)) : \phi \in \confmaps \right\}
            \label{eq:ProductFunctionSpace}
          \end{align}
          is contained in $C^0(\overline{V}, \CC) \times C^1(\overline{V} \times \overline{B_{4 \DerivativeBound \frepsilon}(V)}, \RR)$, so there is a countable subset $\lambda_{\DerivativeBound}(V,U) \subset \confmaps$ such that
          \begin{align*}
            \left\{\overline{V} \times \overline{V} \times \overline{B_{4 \DerivativeBound \frepsilon}(V)} \ni (x,z,w) \to (\phi'(x), \Psi_{\epsilon}^{\phi, z}(w)) : \phi \in \lambda_{\DerivativeBound}(V, U) \right\}
          \end{align*}
          is dense in \eqref{eq:ProductFunctionSpace} with respect to the norm $(f,g) \mapsto \|f\|_{C^0} + \|g\|_{C^0} + \|\nabla g\|_{C^0}$.
          The lemma now follows from the definition of $\locLFPP$ because for each $\epsilon \in (0, 1)$, $\phi, \Phi \in \confmaps$, and $z \in V$,
          \begin{align*}
            &\left|\langle h, \Psi_{\epsilon}^{\phi, z} - \Psi_{\epsilon}^{\Phi, z} \rangle \right| \\
            &\qquad \qquad \leq \ \|h\|_{\cH_0^{-1}(V)} \|\Psi_{\epsilon}^{\phi, z} - \Psi_{\epsilon}^{\Phi, z}\|_{\cH_0^{1}(V)} \\
            &\qquad \qquad \leq \ \sqrt{\mathrm{area}(V)} \|h\|_{\cH_0^{-1}(V)} \left(\|\Psi_{\epsilon}^{\phi, z} - \Psi_{\epsilon}^{\Phi, z}\|_{C^0(\overline{V})} + \|\nabla(\Psi_{\epsilon}^{\phi, z} - \Psi_{\epsilon}^{\Phi, z}) \|_{C^0(\overline{V})}\right), \\
            &\left|\langle -\log|\phi'|, \Psi_{\epsilon}^{\phi, z} \rangle - \langle -\log|\Phi'|, \Psi_{\epsilon}^{\Phi, z} \rangle \right| \\
            &\qquad \qquad = \ \left|\langle -\log \left|\frac{\phi'}{\Phi'} \right| \Psi_{\epsilon}^{\phi, z} \rangle + \langle -\log|\Phi'|, \Psi_{\epsilon}^{\phi, z} - \Psi_{\epsilon}^{\Phi, z} \rangle \right| \\
            &\qquad \qquad \leq \ \frac{\DerivativeBound^2 \mathrm{area}(V)}{Z_{\epsilon} \epsilon^2} \left[\log (1 + \DerivativeBound \|\phi' - \Phi'\|_{C^0(\overline{V})}) \vee \log\frac{1}{1 - \DerivativeBound \|\phi' - \Phi'\|_{C^0(\overline{V})}} \right] \\
            &\qquad \qquad \qquad \qquad + \log(\DerivativeBound) \|\Psi_{\epsilon}^{\phi, z} - \Psi_{\epsilon}^{\Phi, z} \|_{C^0(\overline{V})}, \\
            &\left|\frac{|\phi'(z)|}{|\Phi'(z)|} - 1 \right| \ 
            \leq \ \DerivativeBound \|\phi' - \Phi'\|_{C^0(\overline{V})},
          \end{align*}
          where we view $h$ as a random element of the Sobolev space $\cH_0^{-1}$ \cite[Theorem 1.45]{BerestyckiPowell} equipped with the dual norm to the Sobolev norm $\|f\|_{\cH_0^1} \coloneqq (\|f\|_{L^2}^2 + \|\nabla f\|_{L^2}^2)^{1/2}$.
        \end{proof}
      \end{lemma}
  
      \begin{lemma}
        \label{lemma:InitialLipschitzConditionMeasurability}
        For each $\epsilon \in (0,1)$, $r \in (4 \frepsilon, 1)$, and $x \in V$, we have $E_{r, \epsilon}(x) \in \sigma\{(h - h_{4r}(x))|_{\AA_{r/2, 2r}(x)}\}$.
        \begin{proof}
          The supremum and infimum over $\phi \in \confmaps$ can be taken over a countable set by Lemma \ref{lemma:CountableDenseSetOfConformalMaps}, so the lemma will follow if we know that for each $\phi \in \confmaps$, the event that
          \begin{align}
            \begin{split}
              &\locLFPP[\epsilon][\hphi]\left(\text{around } \phi\left(\AA_{\frac{3r}{4}, \frac{7r}{4}}(x) \right)\right) \\
              &\qquad \qquad \leq \ C \sup_{t \in [\DerivativeBound^{-1}, \DerivativeBound]} \ratios[\epsilon t] D_h\left(\partial B_{\frac{3r}{4}}(x), \partial B_r(x) \right), 
            \end{split} \label{eq:InitialLipschitzConditionSupRemoved} \\
            \begin{split}
              &D_h\left(\text{around } \AA_{\frac{3r}{4}, \frac{7r}{4}}(x) \right) \\
              &\qquad \qquad \leq \ C \sup_{t \in [\DerivativeBound^{-1}, \DerivativeBound]} \invratios[\epsilon t] \locLFPP[\epsilon][\hphi]\left(\phi\left(\partial B_{\frac{3r}{4}}(x) \right), \phi\left(\partial B_r(x) \right)\right)
            \end{split} \label{eq:InitialLipschitzConditionInfRemoved}
          \end{align}
          is determined by $(h - h_{4r}(x))|_{\AA_{r/2, 2r}(x)}$.
          Adding $-h_{4r}(x)$ to $h$ scales each metric in \eqref{eq:InitialLipschitzConditionSupRemoved} and \eqref{eq:InitialLipschitzConditionInfRemoved} by $e^{-\xi h_r(x)}$, so we may assume $h$ is normalized to have $h_{4r}(x) = 0$.
          Since $\frac{7 r}{4} + \frepsilon < 2 r$ and $\frac{3r}{4} - \frepsilon > \frac{r}{2}$, measurability of \eqref{eq:InitialLipschitzConditionSupRemoved} and \eqref{eq:InitialLipschitzConditionInfRemoved} follows from \ref{loc:Locality} of Lemma \ref{lemma:PropertiesOfLocalizedFieldAndLFPP}.
        \end{proof}
      \end{lemma}
  
      To prove that the events of Definition \ref{defn:InitialLipschitzConditionEvents} occur with high enough probability to apply Lemma \ref{lemma:IndependenceAcrossConcentricAnnuli}, we will use Proposition \ref{prop:MainRegularityConditionHighProbability} and Lemma \ref{lemma:PropertiesOfLocalizedFieldAndLFPP} to reduce to comparing $\LFPP[\epsilon/|\phi'(x)|][h]$-distances around a fixed annulus to the $D_h$-distance across a fixed annulus.
      Since the index $\epsilon/|\phi'(x)|$ depends on $\phi$, we will need the following result to control $\LFPP[\epsilon][h]$-distances around annuli for infinitely many values of $\epsilon$ simultaneously.
  
      \begin{lemma}
        \label{lemma:TightnessAroundAnnuli}
        For any $p \in (0, 1)$ and $0 < \alpha < \beta < \infty$, there exists $c, \epsilon_0 > 0$ such that with probability at least $p$,
        \begin{align*}
          D_h\left(\text{around } \AA_{\alpha, \beta}(0) \right) \vee \sup_{\epsilon \in (0, \epsilon_0)} \LFPP[\epsilon][h]\left(\text{around } \AA_{\alpha, \beta}(0) \right) \
          &\leq \ c.
        \end{align*}
        \begin{proof}
          Let $\eta \coloneqq \frac{\alpha + \beta}{2}$.
          The idea is to choose a dense enough finite collection of points on $\partial B_{\eta}(0)$ so that the $D_h$-distance between each consecutive pair of points counterclockwise around the circle is less than $D_h(\partial B_{\eta}(0), \partial \AA_{\alpha, \beta}(0))$.
          We can then use the length space property to concatenate these paths to give a path around $\AA_{\alpha, \beta}(0)$ with $D_h$-length at most $D_h(\partial B_{\eta}(0), \partial \AA_{\alpha, \beta}(0))$ times a random integer.
          Using the convergence of $\LFPP$ to $D_h$, the same set of points can be used to construct a similar path for $\LFPP$ for each $\epsilon$ sufficiently small.
  
          Let
          \begin{align*}
            R \
            &\coloneqq \ D_h(\partial B_{\eta}(0), \partial \AA_{\alpha, \beta}(0)) \\
            &\qquad \qquad \w D_h([-\beta, \beta], [-\beta - \alpha \i/2, \beta - \alpha \i/2]) \w D_h([-\beta, \beta], [-\beta + \alpha \i/2, \beta + \alpha \i/2]).
          \end{align*}
          Define $q_1 \coloneqq 0$ and
          \begin{align*}
            q_{j + 1} \
            &\coloneqq \ \sup\left\{q \in [q_j, \pi] : D_h(\eta e^{\i q_j}, \eta e^{\i q}) < R \right\}.
          \end{align*}
          Then $(q_j)_{j=1}^{\infty}$ is an increasing sequence bounded above by $\pi$, so it converges to some limit $L = \sup\{q_j\}_{j=1}^{\infty}$.
          Then $D_h(\eta e^{\i q_j}, \eta e^{\i L}) \to 0$ by continuity, so $D_h(\eta e^{\i q_j}, \eta e^{\i L}) < R$ for all $j$ sufficiently large.
          Fixing one such index $j$, if $L < \pi$, then continuity again implies $D_h(\eta e^{\i q_j}, \eta e^{\i q}) < R$ for $q \in (L, \pi)$ close enough to $L$, which is impossible because $L = \sup\{q_j\}_{j=1}^{\infty}$.
          So $L = \pi$ and moreover $D_h(\eta e^{\i q_j}, -\eta) < R$ for all $j$ sufficiently large, say for $j \geq J^{+}$.
          Letting $u_j^{+} \coloneqq \eta e^{\i q_j}$ for $j \leq J^{+}$ and $u_{J^{+}+1}^{+} \coloneqq -\eta$, the length space property and the definition of $u_j^{+}$ implies that for each $j \leq J^{+}$, there is a path from $u_j^{+}$ to $u_{j+1}^{+}$ with $D_h$-length less than $R$.
          In particular, each such path cannot leave $\AA_{\alpha, \beta}(0) \cap \{z \in \CC : \Im[z] \geq -\alpha/2\}$.
          By concatenating these paths, we obtain a path $P^{+}$ in $\AA_{\alpha, \beta}(0) \cap \{z \in \CC : \Im[z] \geq -\alpha/2\}$ from $\eta$ to $-\eta$ with $D_h$-length no larger than $J^{+} D_h(\partial B_{\eta}(0), \partial \AA_{\alpha, \beta}(0))$.
          Proceeding analogously with the lower half of the annulus, we can choose points $u_1^{-} = -\eta, u_2^{-}, \ldots, u_{J^{-} + 1} = \eta$ counterclockwise around $\partial B_{\eta}(0)$ such that $D_h(u_j^{-}, u_{j+1}^{-}) < R$ for all $j \leq J^{-}$, and concatenating paths between $u_j^{-}$ and $u_{j+1}^{-}$ like before gives a path $P^{-}$ from $-\eta$ to $\eta$ in $\AA_{\alpha, \beta}(0) \cap \{z \in \CC: \Im[z] \leq \alpha/2\}$ with $D_h$-length no larger than $J^{-} D_h(\partial B_{\eta}(0), \partial \AA_{\alpha, \beta}(0))$.
          Concatenating $P^{+}$ and $P^{-}$ shows that
          \begin{align}
            D_h\left(\text{around } \AA_{\alpha, \beta}(0) \right) \
            &\leq \ \left(J^{+} + J^{-} \right) D_h\left(\partial B_{\eta}(0), \partial \AA_{\alpha, \beta}(0) \right).
            \label{eq:DistanceAroundLQG}
          \end{align}
  
          To obtain the analogous result for $\LFPP$, define 
          \begin{align*}
            R_{\epsilon} \
            &\coloneqq \ \LFPP[\epsilon][h]\left(\partial B_{\eta}(0), \partial \AA_{\alpha, \beta}(0) \right) \w \LFPP[\epsilon][h]\left([-\beta, \beta], [-\beta - \alpha \i/2, \beta - \alpha \i/2] \right) \\
            &\qquad \qquad \w \LFPP[\epsilon][h]\left([-\beta, \beta], [-\beta + \alpha \i/2, \beta + \alpha \i/2] \right) \w D_h\left(\partial B_{\eta}(0), \partial \AA_{\alpha, \beta}(0) \right),
          \end{align*}
          and note that $R_{\epsilon} \to R$ almost surely.
          Using the almost sure convergence of $\LFPP$ to $D_h$, we can find $\epsilon_0$ such that 
          \begin{align}
            \Prob\left\{\LFPP[\epsilon][h]\left(u_j^{\pm}, u_{j+1}^{\pm} \right) < R_{\epsilon} \ \forall j \leq J^{\pm} \ \forall \epsilon \in (0, \epsilon_0) \right\} \
            &\geq \ 1 - \frac{1 - p}{2}.
            \label{eq:DistanceAroundLFPPEvent}
          \end{align}
          Then using the length space property as before, we get that on the event in \eqref{eq:DistanceAroundLFPPEvent}, for each $\epsilon < \epsilon_0$,
          \begin{align}
            \LFPP[\epsilon][h]\left(\text{around } \AA_{\alpha, \beta}(0) \right) \
            &\leq \ \left(J^{+} + J^{-} \right) D_h\left(\partial B_{\eta}(0), \partial \AA_{\alpha, \beta}(0) \right).
            \label{eq:DistanceAroundLFPP}
          \end{align}
          The claim of the lemma follows from \eqref{eq:DistanceAroundLQG}, \eqref{eq:DistanceAroundLFPPEvent}, and \eqref{eq:DistanceAroundLFPP} by choosing $c$ large enough that
          \begin{align*}
            \Prob\left\{\left(J^{+} + J^{-} \right) D_h\left(\partial B_{\eta}(0), \partial \AA_{\alpha, \beta}(0) \right) \leq c \right\} \
            &\geq \ 1 - \frac{1 - p}{2}.
          \end{align*}
        \end{proof}
      \end{lemma}

      \begin{lemma}
        \label{lemma:InitialLipschitzConditionHighProbability}
        Fix open sets $W \Subset V \Subset U$, $\DerivativeBound > 1$, and $p, \zeta, \zeta_{+} \in (0,1)$ with $\zeta < \zeta_{+}$.
        There exists $\epsilon_0 = \epsilon_0(p, V, W, \DerivativeBound, \zeta, \zeta_{+}) \in (0,1)$ and $C_0 = C_0(p) > 1$ such that for every $\epsilon \in (0, \epsilon_0)$, every $x \in W$, and every $r \in (\epsilon^{1 - \zeta}, \epsilon^{1 - \zeta_{+}})$, $\Prob[E_{r, \epsilon}(x; C_0)] \geq p$.
        \begin{proof}
          Choose $\epsilon_0$ according to Proposition \ref{prop:MainRegularityConditionHighProbability} with $\delta = 1$, $\zeta_{+}$ in place of $\zeta$, and $1 - \frac{1 - p}{4}$ in place of $p$.
          Use \ref{loc:UniformComparison} from Lemma \ref{lemma:PropertiesOfLocalizedFieldAndLFPP} to shrink $\epsilon_0$ so that with probability $1 - \frac{1 - p}{4}$, 
          \begin{align}
            \sup_{z \in V} \left|h_{\epsilon}^{*}(z) - \hat{h}_{\epsilon}^{*}(z) \right| \
            &< \ 1 \ \forall \epsilon \in (0, \epsilon_0^{\zeta} \DerivativeBound).
            \label{eq:InitialLipschitzLocToLFPP}
          \end{align}
          Using the almost sure convergence of $\LFPP$ to $D_h$, we can decrease $\epsilon_0$ and find $c_0 > 0$ so that with probability at least $1 - \frac{1 - p}{4}$, 
          \begin{align}
            D_h\left(\partial B_{3/4}(0), \partial B_1(0) \right) \w \inf_{\epsilon \in (0, \epsilon_0^{\zeta} \DerivativeBound)} \LFPP[\epsilon][h]\left(\partial B_{3/4}(0), \partial B_1(0) \right) \
            &\geq \ c_0^{-1}.
            \label{eq:InitialLipschitzLowerBoundAcross}
          \end{align}
          Using Lemma \ref{lemma:TightnessAroundAnnuli}, we can shrink $\epsilon_0$ and find $c_1 > 0$ such that with probability at least $1 - \frac{1 - p}{4}$,
          \begin{align}
            D_h\left(\text{around } \AA_{3/4, 7/4}(0) \right) \vee \sup_{\epsilon \in (0, \epsilon_0^{\zeta} \DerivativeBound)} \LFPP[\epsilon][h]\left(\text{around } \AA_{3/4, 7/4}(0) \right) \
            &\leq \ c_1.
            \label{eq:InitialLipschitzUpperBoundAround}
          \end{align}
          It follows that if $\epsilon \in (0, \epsilon_0)$, $r \in (\epsilon^{1 - \zeta}, \epsilon^{1 - \zeta_{+}})$, and $x \in W$, then with probability at least $p$,
          \begin{align*}
            &\locLFPP[\epsilon][\hphi]\left(\text{around } \phi\left(\AA_{\frac{3r}{4}, \frac{7r}{4}}(x) \right)\right) \\
            &\leq \ 2 \locLFPP[\epsilon/|\phi'(x)|][h]\left(\text{around } \AA_{\frac{3r}{4}, \frac{7r}{4}}(x) \right) \tag*{(Proposition \ref{prop:MainRegularityConditionHighProbability})} \\
            &\leq \ 2 e^{\xi} \LFPP[\epsilon/|\phi'(x)|][h]\left(\text{around } \AA_{\frac{3r}{4}, \frac{7r}{4}}(x) \right) \tag*{(by \eqref{eq:InitialLipschitzLocToLFPP})} \\
            &= \ 2 e^{\xi} \frac{r \fa_{\epsilon/|\phi'(x)|}^{-1}}{r^{\xi Q} \fa_{\epsilon/r|\phi'(x)|}^{-1}} r^{\xi Q} e^{\xi h_r(x)} \LFPP[\epsilon/r|\phi'(x)|][h(r \cdot + x) - h_r(x)]\left(\text{around } \AA_{\frac{3}{4}, \frac{7}{4}}(0) \right) \tag*{(by \eqref{eq:LFPPScaling})} \\
            &\leq \ 2 e^{\xi} \frac{r \fa_{\epsilon/|\phi'(x)|}^{-1}}{r^{\xi Q} \fa_{\epsilon/r|\phi'(x)|}^{-1}} r^{\xi Q} e^{\xi h_r(x)} c_1 \tag*{(by \eqref{eq:InitialLipschitzUpperBoundAround})} \\
            &\leq \ 2 e^{\xi} \frac{r \fa_{\epsilon/|\phi'(x)|}^{-1}}{r^{\xi Q} \fa_{\epsilon/r|\phi'(x)|}^{-1}} r^{\xi Q} e^{\xi h_r(x)} c_1 c_0 D_{h(r \cdot + x) - h_r(x)}\left(\partial B_{3/4}(0), \partial B_1(0) \right) \tag*{(by \eqref{eq:InitialLipschitzLowerBoundAcross})} \\ 
            &= \ 2 e^{\xi} \frac{r \fa_{\epsilon/|\phi'(x)|}^{-1}}{r^{\xi Q} \fa_{\epsilon/r|\phi'(x)|}^{-1}} c_1 c_0 D_{h}\left(\partial B_{\frac{3r}{4}}(x), \partial B_r(x) \right), \tag*{(axioms \ref{axiom:WeylScaling}, \ref{axiom:CoordinateChange})} \\ 
            &D_h\left(\text{around } \AA_{\frac{3r}{4}, \frac{7r}{4}}(x) \right) \\
            &= \ r^{\xi Q} e^{\xi h_r(x)} D_{h(r \cdot + x) - h_r(x)}\left(\text{around } \AA_{\frac{3}{4}, \frac{7}{4}}(0) \right) \tag*{(axioms \ref{axiom:WeylScaling}, \ref{axiom:CoordinateChange})} \\
            &\leq \ r^{\xi Q} e^{\xi h_r(x)} c_1 \tag*{(by \eqref{eq:InitialLipschitzUpperBoundAround})} \\
            &\leq \ r^{\xi Q} e^{\xi h_r(x)} c_1 c_0 \LFPP[\epsilon/r |\phi'(x)|][h(r \cdot + x) - h_r(x)]\left(\partial B_{\frac{3}{4}}(0), \partial B_1(0) \right) \tag*{(by \eqref{eq:InitialLipschitzLowerBoundAcross})} \\
            &= \ \frac{r^{\xi Q} \fa_{\epsilon/r|\phi'(x)|}^{-1}}{r \fa_{\epsilon/|\phi'(x)|}} c_1 c_0 \LFPP[\epsilon/|\phi'(x)|][h]\left(\partial B_{\frac{3r}{4}}(x), \partial B_r(x) \right) \tag*{(by \eqref{eq:LFPPScaling})} \\
            &\leq \ e^{\xi} \frac{r^{\xi Q} \fa_{\epsilon/r|\phi'(x)|}^{-1}}{r \fa_{\epsilon/|\phi'(x)|}} c_1 c_0 \locLFPP[\epsilon/|\phi'(x)|][h]\left(\partial B_{\frac{3r}{4}}(x), \partial B_r(x) \right) \tag*{(by \eqref{eq:InitialLipschitzLocToLFPP})} \\
            &\leq \ 2 e^{\xi} \frac{r^{\xi Q} \fa_{\epsilon/r|\phi'(x)|}^{-1}}{r \fa_{\epsilon/|\phi'(x)|}} c_1 c_0 \locLFPP[\epsilon][\hphi]\left(\phi\left(\partial B_{\frac{3r}{4}}(x) \right), \phi\left(\partial B_r(x) \right) \right) \tag*{(Proposition \ref{prop:MainRegularityConditionHighProbability})}.
          \end{align*}
          The lemma now follows with $C_0 \coloneqq 2 e^{\xi} c_1 c_0$.
        \end{proof}
      \end{lemma}
  
      \begin{proof}[Proof of Lemma \ref{lemma:WeakInitialLipschitzCondition}]
        The strategy is similar to the proofs of \cite[Theorem 1.6]{LocalMetricsOfTheGFF} and \cite[Theorem 1.8]{UpToConstants}; Lemma \ref{lemma:IndependenceAcrossConcentricAnnuli} implies there are many annuli for which the events from Definition \ref{defn:InitialLipschitzConditionEvents} occur, and by stringing together paths around such annuli, we obtain an approximate Lipschitz equivalence between the coordinate-changed LFPP metrics and $D_h$.

        Choose $p$ and $c$ according to Lemma \ref{lemma:IndependenceAcrossConcentricAnnuli} with $s_2 = 1/8$, $s_1 = 1/2$, and $a = 24 \log 8$.
        Choose $\epsilon_0$ and $C_0$ using Lemma \ref{lemma:InitialLipschitzConditionHighProbability} with this choice of $p$ and with an open $\tilde{W}$ such that $W \Subset \tilde{W} \Subset V$ in place of $W$.
        Since $1 - \zeta_{+} < \frac{1}{2} (1 - \zeta)$, it follows that
        \begin{align}
          \# \cR_{\epsilon} \
          &> \ \frac{1}{3} \#\left(\{8^{-k}\}_{k =1}^{\infty} \cap \left(\epsilon^{1 - \zeta}, \epsilon^{1 - \zeta_{+}} \right) \right) \
          \geq \ \frac{1 - \zeta}{6 \log 8} \log \epsilon^{-1} - \frac{1}{3}.
          \label{eq:LowerBoundOnCardinalityOfRadiiSets}
        \end{align}
        In the notation of Lemma \ref{lemma:IndependenceAcrossConcentricAnnuli}, let $(r_k)_{k=1}^{\infty}$ denote the elements of 
        \begin{align*}
          \left\{4 \cdot 8^{-k} : 8^{-k} \in (0, \epsilon^{1 - \zeta}) \cup \cR_{\epsilon} \right\},
        \end{align*}
        and $E_{r_k} = E_{r_k/4, \epsilon}(x; C_0)$ for a fixed $x \in \tilde{W}$ when $r_k/4 \in \cR_{\epsilon}$ and $E_{r_k}$ the trivial event otherwise, where $E_{r, \epsilon}(x; C_0)$ is defined as in Definition \ref{defn:InitialLipschitzConditionEvents}.
        Then Lemma \ref{lemma:IndependenceAcrossConcentricAnnuli} together with \eqref{eq:LowerBoundOnCardinalityOfRadiiSets} imply
        \begin{align*}
          \Prob\left\{\exists r \in \cR_{\epsilon} \text{ such that } E_{r, \epsilon}(x; C_0) \text{ occurs} \right\} \
          &\geq \ 1 - e^{8 \log 8} c \epsilon^{4 (1 - \zeta)}.
        \end{align*}
        Let $F_{\epsilon}$ be the event that for each $x \in \tilde{W} \cap (\frac{\epsilon^{1 - \zeta}}{8} \ZZ^2)$, there exists $r \in \cR_{\epsilon}$ such that $E_{r, \epsilon}(x; C_0)$ occurs.
        Then by a union bound, we see that $\Prob[F_{\epsilon}] = 1 - O(\epsilon^{2(1 - \zeta)})$, with big-$O$ constant depending on $V$.
  
        Now assume $F_{\epsilon}$ occurs, fix $z, w \in W$ with $|z - w| > 4 \epsilon^{1 - \zeta_{+}}$, fix $\phi \in \confmaps$, and fix a $D_h$-length parameterized path $P \subset W$ from $z$ to $w$ with length at most $2 D_h(z,w; W)$.
        Define $t_0 = 0$, and inductively define $t_{j}$ as follows.
        Let $t_{j}$ be the first time after $t_{j-1}$ that $P$ exits $B_{r_{j-1}}(x_{j-1})$, then choose $x_j \in \frac{\epsilon^{1 - \zeta}}{8} \ZZ^2$ and $r_j \in \cR_{\epsilon}$ such that $P(t_j) \in B_{r_j/4}(x_j)$, $E_{r_j, \epsilon}(x_j, C_0)$ occurs.
        If no such time $t_j$ exists, let $t_j \coloneqq \Len(P; D_h)$.
        Let $J$ denote the smallest index such that $t_{J+1} = \Len(P; D_h)$.
  
        Assume $j < J$.
        Then $P(t_j) \in W$ and $|P(t_j) - x_j| < r_j/4$, so $\overline{B_{\frac{7 r_j}{4}}(x_j)} \subset B_{2 \epsilon^{1 - \zeta_{+}}}(W)$.
        Also, since $P(t_j) \in B_{r_j/4}(x_j)$ and $P(t_{j+1}) \in \partial B_{r_j}(x_j)$, we know that $P$ crosses $\AA_{\frac{3 r_j}{4}, r_j}(x_j)$ between time $t_j$ and $t_{j+1}$.
        It follows that
        \begin{align}
          D_h\left(\partial B_{\frac{3 r_j}{4}}(x_j), \partial B_{r_j}(x_j) \right) \
          &\leq \ \Len\left(P|_{[t_j, t_{j+1}]}; D_h \right) \
          = \ t_{j+1} - t_j.
          \label{eq:InitialLipschitzCrossingBound}
        \end{align}
        By definition of $E_{r_j, \epsilon}(x; C_0)$, we can find a curve $\pi_j$ which disconnects the inner and outer boundaries of $\AA_{\frac{3 r_j}{4}, \frac{7 r_j}{4}}(x_j)$ with 
        \begin{align}
          \Len\left(\phi \circ \pi_j, \locLFPP[\epsilon][h^{\phi}]\right) \
          &\leq \ C_0 \sup_{t \in [\DerivativeBound^{-1}, \DerivativeBound]} \frac{r_j \fa_{\epsilon t}^{-1}}{r_j^{\xi Q} \fa_{\epsilon t/r_j}^{-1}} D_h\left(\partial B_{\frac{3 r_j}{4}}(x_j), \partial B_{r_j}(x_j) \right).
          \label{eq:InitialLipschitzPathsAroundAnnuli}
        \end{align}
        By \cite[Lemma 3.4]{UpToConstants}, $\cup_{j < J} \pi_j$ contains a path from $B_{2 \epsilon^{1 - \zeta_{+}}}(z)$ to $B_{2 \epsilon^{1 - \zeta_{+}}}(w)$.
        Therefore, 
        \begin{align*}
          &\locLFPP[\epsilon][\hphi]\left(\phi(B_{2 \epsilon^{1 - \zeta_{+}}}(z)), \phi(B_{2 \epsilon^{1 - \zeta_{+}}}(w)); \phi(B_{2 \epsilon^{1 - \zeta_{+}}}(W)) \right) \\
          &\qquad \leq \ \sum_{j=0}^{J - 1} \locLFPP[\epsilon][\hphi]\left(\text{around } \phi\left(\AA_{\frac{3r_j}{4}, \frac{7r_j}{4}}(x_j) \right)\right) \\
          &\qquad \leq \ \sum_{j=0}^{J - 1} C_0 \sup_{t \in [\DerivativeBound^{-1}, \DerivativeBound]} \frac{r_j \fa_{\epsilon t}^{-1}}{r_j^{\xi Q} \fa_{\epsilon t/r_j}^{-1}} D_h\left(\partial B_{\frac{3r_j}{4}}(x_j), \partial B_{r_j}(x_j) \right) \tag*{(by \eqref{eq:InitialLipschitzPathsAroundAnnuli})} \\
          &\qquad \leq \ C_0 \sup_{r \in \cR_{\epsilon}} \sup_{t \in [\DerivativeBound^{-1}, \DerivativeBound]} \ratios[\epsilon t] \sum_{j=0}^{J - 1} \left(t_{j+1} - t_j \right) \tag*{(by \eqref{eq:InitialLipschitzCrossingBound})} \\
          &\qquad \leq \ 2 C_0 \sup_{r \in \cR_{\epsilon}} \sup_{t \in [\DerivativeBound^{-1}, \DerivativeBound]} \ratios[\epsilon t] D_h\left(z,w; W \right).
        \end{align*}
        Taking the supremum over $\phi \in \confmaps$ proves \eqref{eq:WeakInitialLipschitzConditionSupLeqLQGMetric} with $2 C_0$ in place of $C_0$.
        The proof of \eqref{eq:WeakInitialLipschitzConditionLQGMetricLeqInf} is done analogously.
      \end{proof}
  
      \begin{lemma}
        \label{lemma:APrioriScalingRatiosBound}
        There exist $\cR_{\epsilon} \subset (\epsilon^{1 - \zeta}, \epsilon^{1 - \zeta_{+}}) \cap \{8^{-j}\}_{j \in \NN}$ with $\# \cR_{\epsilon} > \frac{1}{3} \# ((\epsilon^{1 - \zeta}, \epsilon^{1 - \zeta_{+}}) \cap \{8^{-j}\}_{j \in \NN})$ such that 
        \begin{align}
          \limsup_{\epsilon \to 0} \sup_{r \in \cR_{\epsilon}} \sup_{t \in [\DerivativeBound^{-1}, \DerivativeBound]} \ratios[\epsilon t] \
          &< \ \infty,
          \label{eq:LimsupOfScalingRatiosFinite} \\
          \limsup_{\epsilon \to 0} \sup_{r \in \cR_{\epsilon}} \sup_{t \in [\DerivativeBound^{-1}, \DerivativeBound]} \invratios[\epsilon t] \
          &< \ \infty.
          \label{eq:LimsupOfInverseScalingRatiosFinite}
        \end{align}
        \begin{proof}
          The proof is similar to \cite[Lemma 3.5]{UpToConstants}, except that lemma finds a large set of radii $r$ \textit{for each} $\epsilon$ where the ratios $\ratios$ are not too large, while we need a large set of radii $r$ where \textit{all} of the ratios $\ratios[\epsilon t]$ as $t$ ranges over a compact interval are not too large.
          The idea is that if there are many values of $r$ for which $\ratios[\epsilon t]$ is very small or very large for at least one value of $t$, then Lemma \ref{lemma:WeakInitialLipschitzCondition} implies $\locLFPP$ and $D_h$ satisfy either \eqref{eq:InitialLipschitzConditionSupLeqLQGMetric} or \eqref{eq:InitialLipschitzConditionLQGMetricLeqInf} with constant $C_0$ strictly less than $1$.
          This contradicts the convergence of $\LFPP$ to $D_h$.
  
          To make this reasoning precise, let
          \begin{align*}
            \cS_{\epsilon} \
            &\coloneqq \ \left\{\cR \subset (\epsilon^{1 - \zeta}, \epsilon^{1 - \zeta_{+}}) \cap \{8^{-j}\}_{j \in \NN} : \# \cR > \frac{2}{3} \# ((\epsilon^{1 - \zeta}, \epsilon^{1 - \zeta_{+}}) \cap \{8^{-j}\}_{j \in \NN}) \right\}
          \end{align*}
          and
          \begin{align*}
            X_{\epsilon}^{+} \
            &\coloneqq \ \min_{\cR \in \cS_{\epsilon}} \sup_{r \in \cR} \sup_{t \in [\DerivativeBound^{-1}, \DerivativeBound]} \ratios[\epsilon t], \\
            Y_{\epsilon}^{+} \
            &\coloneqq \ \min_{\cR \in \cS_{\epsilon}} \sup_{r \in \cR} \sup_{t \in [\DerivativeBound^{-1}, \DerivativeBound]} \invratios[\epsilon t], \\
            X_{\epsilon}^{-} \
            &\coloneqq \ \min_{\cR \in \cS_{\epsilon}} \sup_{r \in \cR} \inf_{t \in [\DerivativeBound^{-1}, \DerivativeBound]} \ratios[\epsilon t], \\
            Y_{\epsilon}^{-} \
            &\coloneqq \ \min_{\cR \in \cS_{\epsilon}} \sup_{r \in \cR} \inf_{t \in [\DerivativeBound^{-1}, \DerivativeBound]} \invratios[\epsilon t].
          \end{align*}
          Assume the lemma is false, and let $C_0$ be as in Lemma \ref{lemma:WeakInitialLipschitzCondition}.
          Then there is a subsequence $\cE$ such that $X_{\epsilon}^{+} \vee Y_{\epsilon}^{+} > 4 (C_0 + 1)$ for all $\epsilon \in \cE$.
          Indeed, if $X_{\epsilon}^{+} \vee Y_{\epsilon}^{+} \leq 4(C_0 + 1)$ for all sufficiently small $\epsilon \in (0, 1)$, then letting $\cR_{\epsilon}$ be the intersection of a minimizer of $X_{\epsilon}^{+}$ with a minimizer of $Y_{\epsilon}^{+}$, the sets $\cR_{\epsilon}$ have cardinality $> \frac{1}{3} \#((\epsilon^{1 - \zeta}, \epsilon^{1 - \zeta_{+}}) \cap \{8^{-j}\}_{j \in \NN})$ and satisfy \eqref{eq:LimsupOfScalingRatiosFinite} and \eqref{eq:LimsupOfInverseScalingRatiosFinite}, which is a contradiction.
  
          By Lemma \ref{lemma:RegularlyVaryingScalingConstants}, there exists $\epsilon_0 \in (0, 1)$ such that for all $\epsilon \in (0, \DerivativeBound \epsilon_0^{\zeta})$ and all $C \in [\DerivativeBound^{-2}, \DerivativeBound^2]$,
          \begin{align*}
            \frac{\fa_{\epsilon C}}{\fa_{\epsilon}} \
            &\geq \ \frac{1}{2} C^{1 - \xi Q}.
          \end{align*}
          This implies that for all $\epsilon \in (0, \epsilon_0)$, all $r \in (\epsilon^{1 - \zeta}, \epsilon^{1 - \zeta_{+}})$, and all $s,t \in [\DerivativeBound^{-1}, \DerivativeBound]$,
          \begin{align}
            \frac{\frac{r \fa_{\epsilon t}^{-1}}{r^{\xi Q} \fa_{\epsilon t/r}^{-1}}}{\frac{r \fa_{\epsilon s}^{-1}}{r^{\xi Q} \fa_{\epsilon s/r}^{-1}}} \
            &= \ \frac{\fa_{\epsilon s}}{\fa_{\epsilon t}} \frac{\fa_{\epsilon t/r}}{\fa_{\epsilon s/r}} \
            = \ \frac{\fa_{\epsilon t (s/t)}}{\fa_{\epsilon t}} \frac{\fa_{(\epsilon s/r)(t/s)}}{\fa_{\epsilon s/r}} \
            \geq \ \frac{1}{2} \left(\frac{s}{t} \right)^{1 - \xi Q} \frac{1}{2} \left(\frac{t}{s} \right)^{1 - \xi Q} \
            = \ \frac{1}{4}, \label{eq:SupScalingRatiosToInf} \\
            \frac{\frac{r^{\xi Q} \fa_{\epsilon t/r}^{-1}}{r \fa_{\epsilon t}^{-1}}}{\frac{r^{\xi Q} \fa_{\epsilon s/r}^{-1}}{r \fa_{\epsilon s}^{-1}}} \
            &= \ \frac{\fa_{\epsilon t}}{\fa_{\epsilon s}} \frac{\fa_{\epsilon s/r}}{\fa_{\epsilon t/r}} \
            = \ \frac{\fa_{\epsilon s (t/s)}}{\fa_{\epsilon s}} \frac{\fa_{(\epsilon t/r)(s/t)}}{\fa_{\epsilon t/r}} \
            \geq \ \frac{1}{2} \left(\frac{t}{s} \right)^{1 - \xi Q} \frac{1}{2} \left(\frac{s}{t} \right)^{1 - \xi Q} \
            = \ \frac{1}{4}.
            \label{eq:SupInverseScalingRatiosToInf}
          \end{align}
          Combining \eqref{eq:SupScalingRatiosToInf} and \eqref{eq:SupInverseScalingRatiosToInf} with the fact that $X_{\epsilon}^{+} \vee Y_{\epsilon}^{+} > 4 (C_0 + 1)$ for all $\epsilon \in \cE$, it follows that for all $\epsilon \in \cE \cap (0, \epsilon_0)$, we have $X_{\epsilon}^{-} \vee Y_{\epsilon}^{-} > C_0 + 1$.
          Therefore, for each $\epsilon \in \cE \cap (0, \epsilon_0)$, at least one of the sets
          \begin{align}
            \cR_{\epsilon}^{+} \
            &\coloneqq \ \left\{r \in (\epsilon^{1 - \zeta}, \epsilon^{1 - \zeta_{+}}) \cap \{8^{-j}\}_{j \in \NN} : \inf_{t \in [\DerivativeBound^{-1}, \DerivativeBound]} \ratios[\epsilon t] > C_0 + 1 \right\}, \label{eq:BadRadiiSetRatios} \\
            \cR_{\epsilon}^{-} \
            &\coloneqq \ \left\{r \in (\epsilon^{1 - \zeta}, \epsilon^{1 - \zeta_{+}}) \cap \{8^{-j}\}_{j \in \NN} : \inf_{t \in [\DerivativeBound^{-1}, \DerivativeBound]} \invratios[\epsilon t] > C_0 + 1 \right\}, \label{eq:BadRadiiSetInverseRatios}
          \end{align}
          has cardinality at $> \frac{1}{3} \# ((\epsilon^{1 - \zeta}, \epsilon^{1 - \zeta_{+}}) \cap \{8^{-j}\}_{j \in \NN})$.
          Apply Lemma \ref{lemma:WeakInitialLipschitzCondition} with $\cR_{\epsilon}$ equal to $\cR_{\epsilon}^{+}$ when $\epsilon \in \cE \cap (0, \epsilon_0)$ and $\# \cR_{\epsilon}^{+} > \frac{1}{3} \# ((\epsilon^{1 - \zeta}, \epsilon^{1 - \zeta_{+}}) \cap \{8^{-j}\}_{j \in \NN})$, with $\cR_{\epsilon}$ equal to $\cR_{\epsilon}^{-}$ when $\epsilon \in \cE \cap (0, \epsilon_0)$ and $\# \cR_{\epsilon}^{+} \leq \frac{1}{3} \#((\epsilon^{1 - \zeta}, \epsilon^{1 - \zeta_{+}}) \cap \{8^{-j}\}_{j \in \NN}$, and with $\cR_{\epsilon}$ defined arbitrarily when $\epsilon \not\in \cE \cap (0, \epsilon_0)$, and with open sets $W = \DD$, $V = 2 \DD$, $U = 3 \DD$.
          We get that with probability at least $1 - O_{\epsilon}(\epsilon^{2 (1 - \zeta)})$ as $\epsilon \to 0$, for all $z,w \in \DD$,
          \begin{align}
            \locLFPP[\epsilon][h]\left(B_{2 \epsilon^{1 - \zeta_{+}}}(z), B_{2 \epsilon^{1 - \zeta_{+}}}(w); B_{2 \epsilon^{1 - \zeta_{+}}}(\DD) \right) \
            &\leq \ C_0 \sup_{r \in \cR_{\epsilon}} \sup_{t \in [\DerivativeBound^{-1}, \DerivativeBound]} \ratios[\epsilon t] D_h\left(z,w; \DD \right), \label{eq:LipschitzConditionWithConstantLessThan1LFPPLeqLQG} \\
            D_h\left(B_{2 \epsilon^{1 - \zeta_{+}}}(z), B_{2 \epsilon^{1 - \zeta_{+}}}(w); B_{2 \epsilon^{1 - \zeta_{+}}}\DD) \right) \
            &\leq \ C_0 \sup_{r \in \cR_{\epsilon}} \sup_{t \in [\DerivativeBound^{-1}, \DerivativeBound]} \invratios[\epsilon t] \locLFPP[\epsilon][h]\left(z,w; \DD \right).
            \label{eq:LipschitzConditionWithConstantLessThan1LQGLeqLFPP}
          \end{align}

          At least one of the sets $\{\epsilon \in \cE \cap (0, \epsilon_0) : \cR_{\epsilon} = \cR_{\epsilon}^{+}\}$ or $\{\epsilon \in \cE \cap (0, \epsilon_0) : \cR_{\epsilon} = \cR_{\epsilon}^{-}\}$ is infinite.
          Let us assume that the former is infinite, and denote it by $\cE'$.
          Then by \eqref{eq:LipschitzConditionWithConstantLessThan1LQGLeqLFPP} and the fact that $\cR_{\epsilon} = \cR_{\epsilon}^{+}$ for all $\epsilon \in \cE'$, we see that with probability $1 - O_{\epsilon}(\epsilon^{2 (1 - \zeta)})$ as $\epsilon \to 0$ along $\cE'$,
          \begin{align}
            D_h\left(B_{2 \epsilon^{1 - \zeta_{+}}}(z), B_{2 \epsilon^{1 - \zeta_{+}}}(w); B_{2 \epsilon^{1 - \zeta_{+}}}(\DD) \right) \
            &\leq \ \frac{C_0}{C_0 + 1} \locLFPP\left(z,w; \DD \right) \ \forall z,w \in \DD.
            \label{eq:LipschitzConditionContradiction}
          \end{align}
          By Borel-Cantelli applied to a sparse subsequence $\cE'' \subset \cE'$, \eqref{eq:LipschitzConditionContradiction} holds for all $z,w \in \DD$ and all $\epsilon \in \cE''$ sufficiently small almost surely.
          Fixing $R > 0$ such that $D_h(z, w) < D_h(z, \partial \DD)$ for all $z, w \in B_R(0)$ with probability at least $\frac{1}{2}$, it follows that with probability at least $\frac{1}{2}$, for all $\epsilon$ sufficiently small,
          \begin{align*}
            D_h\left(B_{2 \epsilon^{1 - \zeta_{+}}}(0), B_{2 \epsilon^{1 - \zeta_{+}}}(R/2) \right) \
            &\leq \ \frac{C_0}{C_0 + 1} \locLFPP[\epsilon][h]\left(0, R/2 \right).
          \end{align*}
          This is impossible because the left-hand side converges to $D_h(0, R/2)$ by continuity while the right-hand side converges to $\frac{C_0}{C_0 + 1} D_h(0, R/2) < D_h(0, R/2)$.
  
          The case that $\cE' \coloneqq \{\epsilon \in \cE \cap (0, \epsilon_0) : \cR_{\epsilon} = \cR_{\epsilon}^{-}\}$ is infinite is similar.
          We pass to a subsequence $\cE''$ such that almost surely, for all $\epsilon \in \cE''$ sufficiently small,
          \begin{align*}
            \locLFPP[\epsilon][h]\left(B_{2 \epsilon^{1 - \zeta_{+}}}(z), B_{2 \epsilon^{1 - \zeta_{+}}}(w); B_{2 \epsilon^{1 - \zeta_{+}}}(\DD) \right) \
            &\leq \ \frac{C_0}{C_0 + 1} D_h\left(z,w; \DD \right) \ \forall z,w \in \DD.
          \end{align*}
          Then for $R > 0$ small enough, with probability at least $\frac{1}{2}$, for all $\epsilon \in \cE''$ sufficiently small,
          \begin{align*}
            \locLFPP[\epsilon][h]\left(B_{2 \epsilon^{1 - \zeta_{+}}}(0), B_{2 \epsilon^{1 - \zeta_{+}}}(R/2) \right) \
            &\leq \ \frac{C_0}{C_0 + 1} D_h\left(0, R/2 \right).
          \end{align*}
          But the left-hand side converges to $D_h(0, R/2) > \frac{C_0}{C_0 + 1} D_h(0, R/2)$, a contradiction.
        \end{proof}
      \end{lemma}

      \begin{proof}[Proof of Proposition \ref{prop:InitialLipschitzCondition}]
        Let $\cR_{\epsilon}$ be as in Lemma \ref{lemma:APrioriScalingRatiosBound}, so there exists $0 < L < \infty$ such that
        \begin{align*}
          \limsup_{\epsilon \to 0} \sup_{r \in \cR_{\epsilon}} \sup_{t \in [\DerivativeBound^{-1}, \DerivativeBound]} \ratios[\epsilon t] \vee \invratios[\epsilon t] \
          &< \ L.
        \end{align*}
        Apply Lemma \ref{lemma:WeakInitialLipschitzCondition} with this choice of $\cR_{\epsilon}$.
        Then Proposition \ref{prop:InitialLipschitzCondition} holds with constant $C_0$ equal to the constant $C_0$ from Lemma \ref{lemma:WeakInitialLipschitzCondition} times $L$.
      \end{proof}
  
    \subsection{Iteratively improving the Lipschitz constant}
      \label{section:ImprovingTheLipschitzConstant}
      We will now show that the constant $C_0$ from Proposition \ref{prop:InitialLipschitzCondition} can be made close to $1$.
      We assume throughout that $\xi < \xi_{\text{crit}}$.
      The argument is similar to the proof of \cite[Proposition 3.6]{ConformalCovariance}, in which the authors prove that $D_{\hphi}(\phi(\cdot), \phi(\cdot))$ and $D_h$ are almost surely Lipschitz equivalent with Lipschitz constant $1$ for each fixed $\phi$.
      The key difference for us is that we only have an ``approximate'' Lipschitz condition in which distances between small Euclidean balls of radius $\epsilon^{1 - \zeta}$ for one metric are comparable to distances between points for the other.
      So we will use a similar argument to \cite[Proposition 3.1]{AlmostSureConvergence} to iteratively construct a sequence of Lipschitz constants which converge to $1$ at the expense of increasing the scale of the radii $\epsilon^{1 - \zeta}$ (i.e. increasing $\zeta$) at each step.

      In a bit more detail, the idea is to use convergence of $\locLFPP$ together with Lemma \ref{lemma:IndependenceAcrossConcentricAnnuli} to show that with polynomially high probability, a uniformly positive proportion of a $D_h$- or $\locLFPP[\epsilon][\hphi]$-geodesic is comprised of segments for which $D_h$ and $\locLFPP[\epsilon][\hphi]$ are Lipschitz equivalent with Lipschitz constant close to $1$.
      Because the proportion is uniform over all geodesics, we (almost) obtain a global Lipschitz condition with Lipschitz constant closer to $1$.

      One caveat is that $D_h$- and $\locLFPP[\epsilon][\hphi]$-internal metrics on an open set need not admit geodesics for points far apart, so we work only with points close enough to have a geodesic which remains inside $V$.
      Actually, since we need to cover each geodesic by annuli like in the proof of Proposition \ref{prop:InitialLipschitzCondition}, where each annulus needs to be entirely inside $V$ in order to use the bound $|\phi'| \in [\DerivativeBound^{-1}, \DerivativeBound]$, we will need each geodesic to stay a fixed positive distance away from $\partial V$.
      This is the reason for the additional open set $\tilde{W}$ below in Proposition \ref{prop:ImprovingTheLipschitzConstant}.
  
      If $W \Subset \tilde{W}$ are open sets, let $\hypertarget{G}{G(W, \tilde{W})}$ denote the set of $(z,w) \in W^2$ for which there is a $D_h$-geodesic from $z$ to $w$ contained in $\tilde{W}$.
      Similarly, for a conformal map $\phi$ defined on a neighbourhood $U$ of $\tilde{W}$, let $\hypertarget{Gphi}{G_{\epsilon}^{\phi}(W, \tilde{W})}$ denote the set of $(z,w) \in W^2$ for which there is a $\locLFPP[\epsilon][\hphi](\phi(\cdot), \phi(\cdot); \phi(U))$-geodesic from $z$ to $w$ contained in $\tilde{W}$.

      The main result of this section is the following.

      \begin{prop}
        \label{prop:ImprovingTheLipschitzConstant}
        Fix $0 < \xi < \xi_{\text{crit}}$, open sets $W \Subset \tilde{W} \Subset V \Subset U$, and $\DerivativeBound > 1$.
        There is a constant $A > 0$ such that the following is true.
        Assume there exists $\zeta_{-} \in (0,1)$, $\beta > 0$, and $C_0(\epsilon), C_0'(\epsilon) > 0$ such that with probability $1 - O_{\epsilon}(\epsilon^{\beta})$ as $\epsilon \to 0$, for all $\phi \in \confmaps$,
        \begin{align}
          &\locLFPP[\epsilon][\hphi] \left(\phi\left(B_{4 \epsilon^{1 - \zeta_{-}}}(z) \right), \phi\left(B_{4 \epsilon^{1 - \zeta_{-}}}(w) \right); \phi(B_{4 \epsilon^{1 - \zeta_{-}}}(\tilde{W}))\right) \nonumber \\
          &\qquad \qquad \leq \ C_0(\epsilon) D_h\left(z,w; \tilde{W} \right) \ \forall (z,w) \in \hyperlink{G}{G(W, \tilde{W})}, \label{eq:ImprovingTheLipschitzConstantInitialSupLeqLQG} \\
          &D_h\left(B_{4 \epsilon^{1 - \zeta_{-}}}(z), B_{4 \epsilon^{1 - \zeta_{-}}}(w); B_{4 \epsilon^{1 - \zeta_{-}}}(\tilde{W}) \right) \nonumber \\
          &\qquad \qquad \leq \ C_0'(\epsilon) \LFPP[\epsilon][\hphi]\left(\phi(z), \phi(w); \phi(\tilde{W}) \right) \ \forall (z,w) \in \hyperlink{Gphi}{G_{\epsilon}^{\phi}(W, \tilde{W})}. \label{eq:ImprovingTheLipschitzConstantInitialLQGLeqInf}
        \end{align}
        Choose $\delta, \zeta, \zeta_{+} \in (0,1)$ such that $\zeta_{-} < \zeta < \zeta_{+}$ and $1 - \zeta_{+} < \frac{1}{2} (1 - \zeta)$.
        For each $\epsilon \in (0, 1)$, choose $\cR_{\epsilon} \subset (\epsilon^{1 - \zeta}, \epsilon^{1 - \zeta_{+}}) \cap \{8^{-k}\}_{k=1}^{\infty}$ with $\# \cR_{\epsilon} > \frac{1}{3} \#((\epsilon^{1 - \zeta}, \epsilon^{1 - \zeta_{+}}) \cap \{8^{-k}\}_{k=1}^{\infty})$.
        Then with probability $1 - O_{\epsilon}(\epsilon^{\beta \w 2(1 - \zeta)})$, for all $\phi \in \confmaps$,
        \begin{align}
          &\locLFPP[\epsilon][\hphi] \left(\phi\left(B_{4 \epsilon^{1 - \zeta_{+}}}(z) \right), \phi\left(B_{4 \epsilon^{1 - \zeta_{+}}}(w) \right); \phi(B_{4 \epsilon^{1 - \zeta_{+}}}(\tilde{W}))\right) \nonumber \\
          &\qquad \qquad \leq \ C_1(\epsilon) D_h\left(z,w; \tilde{W} \right) \ \forall (z,w) \in \hyperlink{G}{G(W, \tilde{W})}, \label{eq:ImprovingTheLipschitzConstantImprovedSupLeqLQG} \\
          &D_h\left(B_{4 \epsilon^{1 - \zeta_{+}}}(z), B_{4 \epsilon^{1 - \zeta_{+}}}(w); B_{4 \epsilon^{1 - \zeta_{+}}}(\tilde{W}) \right) \nonumber \\
          &\qquad \qquad \leq \ C_1'(\epsilon) \locLFPP[\epsilon][\hphi]\left(\phi(z), \phi(w); \phi(\tilde{W}) \right) \ \forall (z,w) \in \hyperlink{Gphi}{G_{\epsilon}^{\phi}(W, \tilde{W})}, \label{eq:ImprovingTheLipschitzConstantImprovedLQGLeqInf}
        \end{align}
        where the big-$O$ constant in $O_{\epsilon}(\epsilon^{\beta \w 2(1 - \zeta)})$ depends on the big-$O$ constant in $O_{\epsilon}(\epsilon^{\beta})$, $V$, and $\zeta$, and
        \begin{align}
          C_1(\epsilon) \
          &\coloneqq \ \frac{A}{A+1} \left[\sup_{r \in \cR_{\epsilon}} \sup_{t \in [\DerivativeBound^{-1}, \DerivativeBound]} \ratios[\epsilon t] \vee C_0(\epsilon) \right] + \left[\frac{1}{A + 1} + 3 \delta \right] \sup_{r \in \cR_{\epsilon}} \sup_{t \in [\DerivativeBound^{-1}, \DerivativeBound]} \ratios[\epsilon t], \label{eq:ImprovingTheLipschitzConstantImprovedC} \\
          C_1'(\epsilon) \
          &\coloneqq \ \frac{A}{A+1} \left[\sup_{r \in \cR_{\epsilon}} \sup_{t \in [\DerivativeBound^{-1}, \DerivativeBound]} \invratios[\epsilon t] \vee C_0'(\epsilon) \right] + \left[\frac{1}{A + 1} + 3 \delta \right] \sup_{r \in \cR_{\epsilon}} \sup_{t \in [\DerivativeBound^{-1}, \DerivativeBound]} \invratios[\epsilon t]. \label{eq:ImprovingTheLipschitzConstantImprovedCPrime}
        \end{align}
      \end{prop}

      Intuitively, Proposition \ref{prop:ImprovingTheLipschitzConstant} says that if $\locLFPP[\epsilon][\hphi](\phi(\cdot), \phi(\cdot))$ and $D_h$ are approximately bi-Lipschitz equivalent (``approximate'' in the same sense as Proposition \ref{prop:InitialLipschitzCondition}) with high probability with Lipschitz constants $C_0$ and $C_0'$, then they are approximately bi-Lipschitz equivalent with high probabililty with Lipschitz constants $C_1$ and $C_1'$ provided the radius $\epsilon^{1 - \zeta_{-}}$ is increased to $\epsilon^{1 - \zeta}$.
      The constant $C_1$ is, roughly, $\frac{A}{A + 1} C_0$ plus $\frac{1}{A + 1}$ times a term involving the scaling ratios $\ratios$.
      If the ratios $\ratios$ are close to $1$, then $C_1$ will be closer to $1$ than $C_0$.
      Likewise for $C_1'$.
      Note that we allow $C_0$, $C_0'$ to be $\epsilon$-dependent.
      This is so that we can apply Proposition \ref{prop:ImprovingTheLipschitzConstant} iteratively, with $C_1, C_1'$ in place of $C_0, C_0'$ to obtain Lipschitz constants $C_2, C_2'$.
      Repeating in this manner, we obtain Lipschitz constants $C_n, C_n'$.
      Since the constant $A$ is universal, we can show that $C_n, C_n' \to 1$ as $n \to \infty$.

      Like the proof of Proposition \ref{prop:InitialLipschitzCondition}, we will use Lemma \ref{lemma:IndependenceAcrossConcentricAnnuli} to find many annuli on which $\locLFPP[\epsilon][\hphi](\phi(\cdot), \phi(\cdot))$ and $D_h$ are bi-Lipschitz equivalent with constant close to $1$.
      The events we will apply Lemma \ref{lemma:IndependenceAcrossConcentricAnnuli} to are as follows.

      \begin{defn}
        \label{defn:ImprovingTheLipschitzConstantRegularityEvents}
        Fix $x \in W$, $\epsilon, \zeta_{-}, \delta \in (0, 1)$, $\alpha \in (7/8, 1)$, $A > 1$, and $r \in (\frac{4 \epsilon^{1 - \zeta_{-}}}{\alpha - 3/4}, 1)$.
        Let $E_{r, \epsilon}$ be the intersection of the following events.
        \begin{enumerate}
          \item $E_{r, \epsilon}^{(1)}(x) = E_{r, \epsilon}^{(1)}(x; \alpha, \delta)$: For each $\phi \in \confmaps$, $u \in \partial B_{\alpha r}(x)$, and $v \in \partial B_r(x)$ for which there is a $D_h(\cdot, \cdot; \AA_{3r/4, 5r/4}(x))$-geodesic (resp. $\locLFPP[\epsilon][\hphi](\phi(\cdot), \phi(\cdot); \phi(\AA_{3r/4, 5r/4}(x)))$-geodesic) from $u$ to $v$ contained in $\AA_{\alpha r, r}(x)$, we have respectively
            \begin{align*}
              &\locLFPP[\epsilon][h^{\phi}]\left(\phi(u), \phi(v); \phi\left(\AA_{3r/4, 5r/4}(x) \right)\right) \\
              &\leq \ \sup_{t \in [\DerivativeBound^{-1}, \DerivativeBound]} \ratios[\epsilon t] (1 + \delta) D_h\left(u, v; \AA_{\alpha r, r}(x) \right), \\
              &D_h\left(u, v; \AA_{3r/4, 5r/4}(x) \right) \\
              &\leq \ \sup_{t \in [\DerivativeBound^{-1}, \DerivativeBound]} \invratios[\epsilon t] \left(1 + \delta \right) \locLFPP[\epsilon][h^{\phi}]\left(\phi(u), \phi(v); \phi\left(\AA_{\alpha r, r}(x) \right)\right).
            \end{align*}
          \item $E_{r, \epsilon}^{(2)}(x) = E_{r, \epsilon}^{(2)}(x; \alpha)$: If $\phi \in \confmaps$, $u \in \partial B_{\alpha r}(x)$, and $v \in \partial B_r(x)$, then
            \begin{itemize}
              \item $D_h(u, v; \overline{\AA_{\alpha r, r}(x)}) = D_h(u, v; \AA_{3r/4, 5r/4}(x))$ implies that for all $t \in [\DerivativeBound^{-1}, \DerivativeBound]$,
                \begin{align*}
                  \locLFPP[\epsilon t][h]\left(u, v \right) \
                  &< \ \locLFPP[\epsilon t][h]\left(u, \partial \AA_{3r/4, 5r/4}(x) \right)
                \end{align*}
              \item $\locLFPP[\epsilon][\hphi](\phi(u), \phi(v); \phi(\overline{\AA_{\alpha r, r}(x)})) = \locLFPP[\epsilon][\hphi](\phi(u), \phi(v); \phi(\AA_{3r/4, 5r/4}(x)))$ for some $\phi \in \confmaps$ implies 
                \begin{align*}
                  D_h\left(u, v \right) \
                  &< \ D_h\left(u, \partial \AA_{3r/4, 5r/4}(x) \right).
                \end{align*}
            \end{itemize}
          \item $E_{r, \epsilon}^{(3)}(x) = E_{r, \epsilon}^{(3)}(x; \alpha, A)$: For all $\phi \in \confmaps$,
            \begin{align*}
              D_h\left(\text{around } \AA_{\alpha r, r}(x) \right) \
              &\leq \ A D_h\left(\partial B_{\alpha r}(x), \partial B_r(x) \right), \\
              \locLFPP[\epsilon][\hphi]\left(\text{around } \phi(\AA_{\alpha r, r}(x)) \right) \
              &\leq \ A \locLFPP[\epsilon][\hphi]\left(\phi(\partial B_{\alpha r}(x)), \phi(\partial B_r(x)) \right).
            \end{align*}
          \item $E_{r, \epsilon}^{(4)}(x) = E_{r, \epsilon}^{(4)}(x; \alpha, \delta, \zeta_{-})$: For all $u \in \overline{\AA_{\alpha r, r}(x)}$, $v \in \overline{B_{4 \epsilon^{1 - \zeta_{-}}}(u)}$, and $\phi \in \confmaps$, we have
            \begin{align*}
              &\locLFPP[\epsilon][\hphi]\left(\phi(u), \phi(v); \phi\left(\AA_{3r/4, 5r/4}(x) \right)\right) \\
              &\qquad \qquad \leq \ \delta \sup_{t \in [\DerivativeBound^{-1}, \DerivativeBound]} \ratios[\epsilon t] D_h\left(\partial B_{\alpha r}(x), \partial B_r(x) \right), \\
              &D_h\left(u, v; \AA_{3r/4, 5r/4}(x) \right) \\
              &\leq \ \delta \sup_{t \in [\DerivativeBound^{-1}, \DerivativeBound]} \invratios[\epsilon t] \locLFPP[\epsilon][\hphi]\left(\phi\left(\partial B_{\alpha r}(x) \right), \phi\left(\partial B_r(x) \right)\right).
            \end{align*}
        \end{enumerate}
      \end{defn}

      \begin{figure}[h!]

        {
          \centering
          \begin{center}
            \includegraphics[scale=0.6]{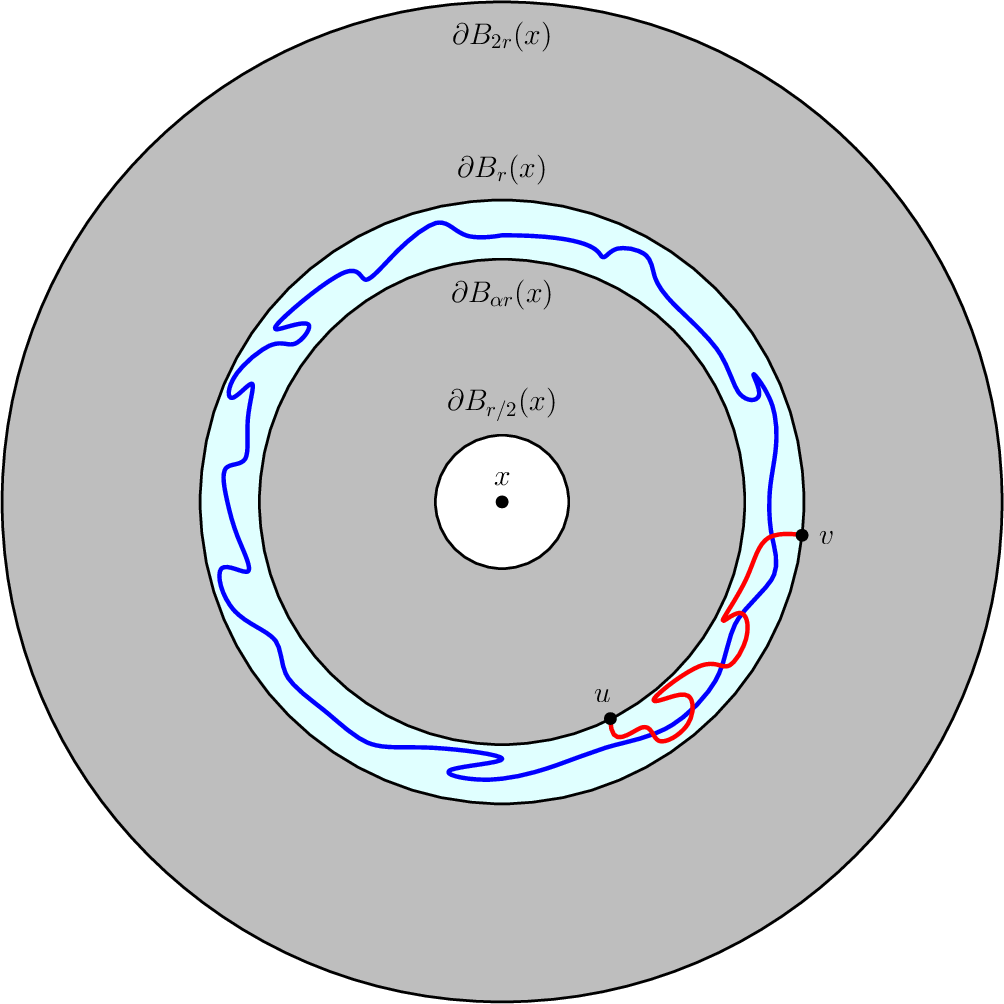}
          \end{center}
        }

        \caption{Illustration of Definition \ref{defn:ImprovingTheLipschitzConstantRegularityEvents}.
          Condition 1 says that if the red path is a $D_h$-geodesic, then $\locLFPP[\epsilon][\hphi](\phi(u), \phi(v)) \leq \sup_{t \in [\DerivativeBound^{-1}, \DerivativeBound]} \ratios[\epsilon t] (1 + \delta) D_h(u,v)$, and analogously with $D_h$ and $\locLFPP[\epsilon][\hphi]$ swapped.
          The second bullet of condition 2 says that if $D_h(u, v)$ is at least the $D_h$-distance from $u$ to the boundary of the gray annulus, then the red path cannot be a $\locLFPP[\epsilon][\hphi]$-geodesic for any $\phi \in \confmaps$.
          The first bullet of condition 2 says that if there is some $t \in [\DerivativeBound^{-1}, \DerivativeBound]$ such that $\locLFPP[\epsilon t][h](u, v)$ is at least the $\locLFPP[\epsilon t][h]$-distance from $u$ to the boundary of the gray annulus, then the red path cannot be a $D_h$-geodesic.
          Condition 3 says there is a path (drawn in blue) around the cyan annulus with $D_h$-length at most $A$ times the $D_h$-distance across the cyan annulus, and analogously for each metric $\locLFPP[\epsilon][\hphi]$ with $\phi \in \confmaps$.
        }
      \end{figure}
  
      The condition that $r > \frac{4 \epsilon^{1 - \zeta_{-}}}{\alpha - 3/4}$ is present to ensure the internal metrics in the fourth condition make sense.
  
      Condition 1 should occur with high probability by Proposition \ref{prop:MainRegularityConditionHighProbability} and the convergence of $\locLFPP$ to $D_h$.
      This is almost correct, but the internal metrics in condition 1 make the argument a bit more complicated.
      The internal metrics are needed to apply Proposition \ref{prop:MainRegularityConditionHighProbability} to compare $\locLFPP[\epsilon][\hphi]$ and $\locLFPP[\epsilon/|\phi'(x)|][h]$, but in order to compare $\locLFPP[\epsilon/|\phi'(x)|][h]$ to $D_h$, we need to remove the internal metrics.
      Condition 2 is what allows us to do this.
  
      Condition 3 is used to say that a $\frac{1}{A + 1}$-proportion of a geodesic lies in one of the annuli $\AA_{\alpha r, r}(x)$.
      Condition 4 is for controlling error terms from estimating point-to-point distances using ball-to-ball distances.

      Let us first check that $E_{r, \epsilon}(x)$ is determined by the field restricted to an annulus.
  
      \begin{lemma}
        \label{lemma:ImprovingTheLipschitzConstantMeasurability}
        For each $0 < \zeta_{-} < \zeta < \zeta_{+} < 1$, $\epsilon \in (0, 1)$ with $\epsilon^{\zeta} \log (\DerivativeBound \epsilon^{-1}) < \frac{1}{16 \DerivativeBound}$, each $r \in (\epsilon^{1 - \zeta}, \epsilon^{1 - \zeta_{+}})$, $x \in W$, $\delta \in (0, 1)$, $\alpha \in (7/8, 1)$, and $A > 1$, we have
        \begin{align*}
          E_{r, \epsilon}(x) \
          &\in \ \sigma\left\{(h - h_{4r}(x))|_{\AA_{r/2, 2r}(x)} \right\}.
        \end{align*}
        \begin{proof}
          Adding $-h_{4 r}(x)$ to $h$ scales each of the metrics in Definition \ref{defn:ImprovingTheLipschitzConstantRegularityEvents} by $e^{-\xi h_{4r}(x)}$, which doesn't affect the occurrence of $E_{r, \epsilon}(x)$.
          Therefore, conditions 1, 3, and 4 of Definition \ref{defn:ImprovingTheLipschitzConstantRegularityEvents} are determined by the internal metric of $D_{h - h_{4r}(x)}$ on $\AA_{3r/4, 5r/4}(x)$, and of $\locLFPP[\epsilon][\hphi - h_{4 r}(x)]$ on $\phi(\AA_{3r/4, 5r/4}(x))$ with $\phi \in \confmaps$.
          By Lemma \ref{lemma:CountableDenseSetOfConformalMaps}, we can replace $\confmaps$ by a countable subset $\lambda_{\DerivativeBound}(V, U) \subset \confmaps$.
          By axiom \ref{axiom:Locality}, the internal metric of $D_{h - h_{4r}(x)}$ on $\AA_{3r/4, 5r/4}(x)$ is determined by $(h - h_{4r}(x))|_{\AA_{3r/4, 5r/4}(x)}$.
          For each $\phi \in \confmaps$, Lemma \ref{lemma:DistortedBumpProperties} implies the internal metric of $\locLFPP[\epsilon][\hphi - h_{4r}(x)]$ on $\phi(\AA_{3r/4, 5r/4}(x))$ is determined by $(h - h_{4r}(x))|_{B_{4 \DerivativeBound \epsilon \log \epsilon^{-1}}(\AA_{3r/4, 5r/4}(x))}$. 
          The condition that $\epsilon^{\zeta} \log (\DerivativeBound \epsilon^{-1}) < \frac{1}{16 \DerivativeBound}$ implies $B_{4 \DerivativeBound \epsilon \log \epsilon^{-1}}(\AA_{3r/4, 5r/4}(x)) \subset \AA_{r/2, 2r}(x)$.
          This shows that $E_{r, \epsilon}^{(1)}(x), E_{r, \epsilon}^{(3)}(x), E_{r, \epsilon}^{(4)}(x) \in \sigma\{(h - h_{4 r}(x))|_{\AA_{r/2, 2r}(x)}\}$.
  
          To see that $E_{r, \epsilon}^{(2)}(x) \in \sigma\{(h - h_{4r}(x))|_{\AA_{r/2, 2r}(x)}\}$, note that by continuity of $(\epsilon, z) \mapsto \hat{h}_{\epsilon}^{*}(z)$ and Lemma \ref{lemma:CountableDenseSetOfConformalMaps}, condition 2 is equivalent to the same condition with $[\DerivativeBound^{-1}, \DerivativeBound]$ replaced by $[\DerivativeBound^{-1}, \DerivativeBound] \cap \QQ$ and $\confmaps$ replaced by a countable subset $\lambda_{\DerivativeBound}(V, U)$.
          Then note that for each $t \in [\DerivativeBound^{-1}, \DerivativeBound]$, the statement that $\locLFPP[\epsilon t][h - h_{4r}(x)](u, v) < \locLFPP[\epsilon t][h - h_{4r}(x)](u, \partial \AA_{3r/4, 5r/4}(x))$ is equivalent to saying $v$ is contained in the open $\locLFPP[\epsilon t][h - h_{4r}(x)]$-ball of radius $\locLFPP[\epsilon t][h - h_{4r}(x)](u, \partial \AA_{3r/4, 5r/4}(x))$ centered at $u$.
          This ball is, by definition, contained in $\AA_{3r/4, 5r/4}(x)$, so the first bullet point of condition 2 is determined by $(h - h_{4r}(x))|_{B_{\epsilon \DerivativeBound \log (\DerivativeBound \epsilon^{-1})}(\AA_{3r/4, 5r/4}(x))}$.
          Since $\epsilon^{\zeta} \log(\DerivativeBound \epsilon^{-1}) < \frac{1}{16 \DerivativeBound}$, we have $B_{\epsilon \DerivativeBound \log(\DerivativeBound \epsilon^{-1})}(\AA_{3r/4, 5r/4}(x)) \subset \AA_{r/2, 2r}(x)$.
  
          Likewise, the condition that $D_{h - h_{4r}(x)}(u,v) < D_{h - h_{4r}(x)}(u, \partial \AA_{3r/4, 5r/4}(x))$ means to say that $v$ is contained in the open $D_{h - h_{4r}(x)}$-ball of radius $D_h(u, \partial \AA_{3r/4, 5r/4}(x))$ centered at $u$, which is contained in $\AA_{3r/4, 5r/4}(x)$.
          Therefore, the second bullet point in condition 2 is also determined by $(h - h_{4r}(x))|_{\AA_{r/2, 2r}(x)}$.
        \end{proof}
      \end{lemma}

      We will now prove that the parameters in $E_{r, \epsilon}(x)$ can be chosen in a way that $E_{r, \epsilon}(x)$ occurs with high probability.
      \begin{prop}
        \label{prop:ImprovingTheLipschitzConstantHighProbability}
        Fix open sets $W \Subset V \Subset U \subset \CC$.
        For any $p \in (0, 1)$, $\DerivativeBound > 1$, $\delta \in (0, 1)$, $0 < \zeta_{-} < \zeta < \zeta_{+} < 1$, there exist $\alpha = \alpha(p) \in (7/8, 1)$, $A = A(p) > 1$, and $\epsilon_0 = \epsilon_0(p, \zeta_{-}, \zeta, \zeta_{+}, \delta) \in (0, 1)$ such that for all $x \in W$, $\epsilon \in (0, \epsilon_0)$, and $r \in (\epsilon^{1 - \zeta}, \epsilon^{1 - \zeta_{+}})$, $\Prob[E_{r, \epsilon}(x)] \geq p$.
      \end{prop}
      To prove Proposition \ref{prop:ImprovingTheLipschitzConstantHighProbability}, it will suffice to prove that each of the events $E_{r, \epsilon}^{(j)}(x)$ with $j = 1, 2, 3, 4$ occurs with high probability when the parameters are selected properly.
      The approach is similar to the arguments in \cite[Section 3.2]{ConformalCovariance} and \cite[Proposition 3.1]{AlmostSureConvergence}, with the key difference that we must select parameters which work for all conformal maps $\phi$.
      The key to this is to use Proposition \ref{prop:MainRegularityConditionHighProbability} to reduce to working with $\locLFPP[\epsilon/|\phi'(z_0)|][h]$.
      From here, the parameters can be selected using similar arguments to \cite[Section 3]{AlmostSureConvergence}.
  
      The proof that $E_{r, \epsilon}^{(2)}(x)$ occurs with high probability will use the following, which is a special case of \cite[Proposition 4.1]{WeakLQGMetrics}.
  
      \begin{lemma}
        \label{lemma:LQGDistancesInTube}
        Fix $b > 0$, $p \in (0, 1)$, and $S > 0$.
        There exists $7/8 < \alpha_0 < 1 < \alpha_1$ such that for each $r > 0$ and each $x \in \CC$, it holds with probability at least $p$ that
        \begin{align*}
          \inf\left\{D_h(u,v; \AA_{\alpha_0 r, \alpha_1 r}(x) : u, v \in \AA_{\alpha_0 r, \alpha_1 r}(x), |u - v| \geq b r \right\} \
          &\geq \ S r^{\xi Q} e^{\xi h_r(x)}.
        \end{align*}
      \end{lemma}
  
      \begin{lemma}
        \label{lemma:ImprovingTheLipschitzConstantCondition2}
        For any $p \in (0, 1)$, $\DerivativeBound > 1$, and $0 < \zeta < \zeta_{+} < 1$, there exist $\alpha = \alpha(p) \in (7/8, 1)$ and $\epsilon_0 = \epsilon_0(p, \zeta, \zeta_{+}) \in (0, 1)$ such that for each $\epsilon \in (0, \epsilon_0)$, each $r \in (\epsilon^{1 - \zeta}, \epsilon^{1 - \zeta_{+}})$, and each $x \in W$, $E_{r, \epsilon}^{(2)}(x)$ occurs with probability at least $p$.
        \begin{proof}
          The idea is that when $u, v \in \AA_{\alpha r, r}(x)$ are far apart with respect to $D_h$, they are far apart with respect to Euclidean distance as well.
          Lemma \ref{lemma:LQGDistancesInTube} plus convergence of LFPP then implies $\locLFPP[\epsilon][h](u, v; \AA_{\alpha r, r}(x)) > \locLFPP[\epsilon][h](u, v; \AA_{3r/4, 5r/4}(x))$.

          Let $C_0$ be as in Proposition \ref{prop:InitialLipschitzCondition}.
          We claim that there exists $0 < s < S$ and $\epsilon_0 > 0$ such that for each $\epsilon \in (0, \epsilon_0)$, each $r \in (\epsilon^{1 - \zeta}, \epsilon^{1 - \zeta_{+}})$, and each $x \in W$, it holds with probability at least $p$ that for all $t \in [\DerivativeBound^{-1}, \DerivativeBound]$,
          \begin{align}
            \sup_{u, v \in \AA_{7r/8, r}(x)} \locLFPP[\epsilon t][h]\left(u, v; \AA_{3r/4, 5r/4}(x) \right) \
            &< \ \frac{1}{16} C_0^{-1} S \ratios[\epsilon t] r^{\xi Q} e^{\xi h_{r}(x)}, \label{eq:Condition2SupLFPP} \\
            \sup_{u, v \in \AA_{7r/8, r}(x)} D_h\left(u, v; \AA_{3r/4, 5r/4}(x) \right) \
            &< \ S r^{\xi Q} e^{\xi h_r(x)}, \label{eq:Condition2SupLQG} \\
            \locLFPP[\epsilon t][h]\left(\partial \AA_{7r/8, r}(x), \partial \AA_{3r/4, 5r/4}(x) \right) \
            &> \ s \ratios[\epsilon t] r^{\xi Q} e^{\xi h_r(x)}, \label{eq:Condition2LFPPAcross} \\
            D_h\left(\partial \AA_{7r/8, r}(x), \partial \AA_{3r/4, 5r/4}(x) \right) \
            &> \ s r^{\xi Q} e^{\xi h_r(x)}. \label{eq:Condition2LQGAcross}
          \end{align}
          Assuming we have \eqref{eq:Condition2SupLFPP}, \eqref{eq:Condition2SupLQG}, \eqref{eq:Condition2LFPPAcross}, and \eqref{eq:Condition2LQGAcross}, let us deduce Lemma \ref{lemma:ImprovingTheLipschitzConstantCondition2}.
          With $s, S > 0$ as in \eqref{eq:Condition2SupLFPP}, \eqref{eq:Condition2SupLQG}, \eqref{eq:Condition2LFPPAcross}, and \eqref{eq:Condition2LQGAcross} with $1 - \frac{1 - p}{7}$ in place of $p$, axiom \ref{axiom:CoordinateChange} implies there exists $b = b(s) > 0$ such that for each $r > 0$ and each $x \in \CC$, with probability at least $1 - \frac{1 - p}{7}$, 
          \begin{align}
            \left(u, v \in \overline{B_r(x)}, \ D_h(u,v) > \frac{1}{3} s r^{\xi Q} e^{\xi h_r(x)} \right) \implies |u - v| \geq b r.
            \label{eq:FarAwayDhDistanceImpliesFarAwayEuclideanDistance}
          \end{align}
          Fix $b' \in (0, b)$.
          Let $7/8 < \alpha_0 < 1 < \alpha_1$ be as in Lemma \ref{lemma:LQGDistancesInTube} with $b'$ in place of $b$ and $1 - \frac{1 - p}{7}$ in place of $p$, so for each $r > 0$ and each $x \in \CC$, it holds with probability at least $1 - \frac{1 - p}{7}$ that
          \begin{align}
            \inf\left\{D_h\left(u, v; \AA_{\alpha_0 r, \alpha_1 r}(x) \right) : u,v \in \AA_{\alpha_0 r, \alpha_1 r}(x), |u - v| \geq b' r \right\} \
            &\geq \ S r^{\xi Q} e^{\xi h_r(x)}. \label{eq:DhDistancesInNarrowTube}
          \end{align}
          Fix $\alpha \in (\alpha_0, 1)$.
          By \cite[Lemma 3.8]{WeakLQGMetrics} and the almost sure convergence of $\LFPP$ and $\locLFPP$ to $D_h$ uniformly on compacts, we can find $R > 0$ and $\epsilon_0 \in (0, 1)$ such that $\overline{U} \subset B_R(0)$ and such that with probability at least $1 - \frac{1 - p}{7}$, 
          \begin{align}
            \sup_{z, w \in B_2(0)} \left|\LFPP(z, w) - D_h(z, w) \right| \
            &< \ \frac{1}{2} D_h\left(\partial B_{\alpha}(0), \partial B_1(0) \right) \ \forall \epsilon \in (0, \epsilon_0^{\zeta} \DerivativeBound), \label{eq:Condition2Convergence} \\
            \LFPP\left(u, v \right) \
            &< \ \LFPP\left(U, \partial B_R(0) \right) \ \forall u, v \in U, \forall \epsilon \in (0, \epsilon_0 \DerivativeBound), \label{eq:Condition2LFPPRemoveInternal} \\
            \locLFPP\left(u, v \right) \
            &< \ \locLFPP\left(U, \partial B_R(0) \right) \ \forall u, v \in U, \forall \epsilon \in (0, \epsilon_0 \DerivativeBound). \label{eq:Condition2locLFPPRemoveInternal}
          \end{align}
          By Lemma \ref{lemma:PropertiesOfLocalizedFieldAndLFPP}, we can decrease $\epsilon_0$ so that with probability at least $1 - \frac{1 - p}{7}$,
          \begin{align}
            \sup_{z \in \overline{B_R(0)}} \left|h_{\epsilon}^{*}(z) - \hat{h}_{\epsilon}^{*}(z) \right| \
            &< \ \frac{1}{\xi} \log 2 \ \forall \epsilon \in (0, \epsilon_0 \DerivativeBound). \label{eq:Condition2CompareLFPPAndlocLFPP}
          \end{align}
          Proposition \ref{prop:MainRegularityConditionHighProbability} implies we can shrink $\epsilon_0$ further so that for each $\epsilon \in (0, \epsilon_0)$ and each $x \in W$, with probability at least $1 - \frac{1 - p}{7}$, for all $\phi \in \confmaps$ and all piecewise $C^1$ paths $P$ in $B_{2 \epsilon^{1 - \zeta_{+}}}(x)$,
          \begin{align}
            \frac{1}{2} \Len\left(P; \locLFPP[\epsilon/|\phi'(x)|][h] \right) \
            &\leq \ \Len\left(\phi \circ P; \locLFPP[\epsilon][\hphi] \right) \
            \leq \ 2 \Len\left(P; \locLFPP[\epsilon/|\phi'(x)|][h] \right).
            \label{eq:Condition2PathwiseComparison}
          \end{align}
          By Proposition \ref{prop:InitialLipschitzCondition}, we can decrease $\epsilon_0$ even further so that for each $\epsilon \in (0, \epsilon_0^{\zeta})$, with probability at least $1 - \frac{1 - p}{7}$, for all $\Phi \in \confmaps$ and all $z, w \in W$,
          \begin{align}
            D_h\left(B_{2 \epsilon^{1 - \zeta}}(z), B_{2 \epsilon^{1 - \zeta}}(w); B_{2 \epsilon^{1 - \zeta}}(W) \right) \
            &\leq \ C_0 \locLFPP[\epsilon][h^{\Phi}]\left(\Phi(z), \Phi(w); \Phi(W) \right).
            \label{eq:Condition2InitialLipschitz}
          \end{align}
          Decrease $\epsilon_0$ so that the following are true:
          \begin{align}
            \epsilon_0^{\zeta(1 - \zeta)} \ 
            &< \ \frac{b - b'}{4} \w \frac{\alpha - \alpha_0}{2} \w \frac{\alpha_1 - 1}{2}, \label{eq:Condition2EpsilonCutoff} \\
            \frac{\fa_{\epsilon t}}{\fa_{\epsilon}} \
            &\leq \ 2 t^{1 - \xi Q} \ \forall t \in [\DerivativeBound^{-1}, \DerivativeBound], \forall \epsilon \in (0, \epsilon_0^{\zeta}). \label{eq:Condition2RegularVariation}
          \end{align}
          Note that in particular, \eqref{eq:Condition2EpsilonCutoff} implies $B_{2 \epsilon^{1 - \zeta} r^{\zeta}}(\overline{\AA_{\alpha r, r}(x)}) \subset \AA_{\alpha_0 r, \alpha_1 r}(x)$ for all $\epsilon \in (0, \epsilon_0)$ and all $r \in (\epsilon^{1 - \zeta}, \epsilon^{1 - \zeta_{+}})$.
  
          Assuming each of the events above occurs, which happens with probability at least $p$, let us show that $E_{r, \epsilon}^{(2)}(x)$ occurs.
          Assume $u \in \partial B_{\alpha r}(x)$, $v \in \partial B_r(x)$, and $t \in [\DerivativeBound^{-1}, \DerivativeBound]$ are such that
          \begin{align*}
            \locLFPP[\epsilon t][h]\left(u, v \right) \
            &\geq \ \locLFPP[\epsilon t][h]\left(u, \partial \AA_{3r/4, 5r/4}(x) \right).
          \end{align*}
          Then
          \begin{align*}
            \locLFPP[\epsilon t][h]\left(u, v \right) \
            &> \ \locLFPP[\epsilon t][h]\left(\partial \AA_{7r/8, r}(x), \partial \AA_{3r/4, 5r/4}(x) \right) \
            \geq \ s \ratios[\epsilon t] r^{\xi Q} e^{\xi h_r(x)} \tag*{(by \eqref{eq:Condition2LFPPAcross})}.
          \end{align*}
          Therefore, if $u' \coloneqq \frac{u - x}{r}$ and $v' \coloneqq \frac{v - x}{r}$,
          \begin{align*}
            D_h\left(u, v \right) \
            &= \ r^{\xi Q} e^{\xi h_r(x)} D_{h(r \cdot + x) - h_r(x)}\left(u', v' \right) \tag*{(axiom \ref{axiom:CoordinateChange})} \\
            &\geq \ \frac{2}{3} r^{\xi Q} e^{\xi h_r(x)} \LFPP[\epsilon t/r][h(r \cdot + x) - h_r(x)]\left(u', v' \right) \tag*{(by \eqref{eq:Condition2Convergence})} \\
            &= \ \frac{2}{3} \invratios[\epsilon t] \LFPP[\epsilon t][h]\left(u, v \right) \tag*{(by \eqref{eq:LFPPScaling})} \\
            &\geq \ \frac{1}{3} \invratios[\epsilon t] \locLFPP[\epsilon t][h]\left(u, v \right) \tag*{(by \eqref{eq:Condition2LFPPRemoveInternal}, \eqref{eq:Condition2locLFPPRemoveInternal}, \eqref{eq:Condition2CompareLFPPAndlocLFPP})} \\
            &> \ \frac{1}{3} s r^{\xi Q} e^{\xi h_r(x)}.
          \end{align*}
          By \eqref{eq:FarAwayDhDistanceImpliesFarAwayEuclideanDistance}, we have $|u - v| \geq b r$.
          Therefore,
          \begin{align*}
            D_h\left(u, v; \overline{\AA_{\alpha r, r}(x)} \right) \
            &\geq \ S r^{\xi Q} e^{\xi h_r(x)} \tag*{(by \eqref{eq:DhDistancesInNarrowTube})} \\
            &> \ D_h\left(u, v; \AA_{3r/4, 5r/4}(x) \right) \tag*{(by \eqref{eq:Condition2SupLQG})}.
          \end{align*}
          This proves the first bullet point in the definition of $E_{r, \epsilon}^{(2)}(x)$.
  
          To prove the second bullet point, assume $u \in \partial B_{\alpha r}(x)$ and $v \in \partial B_r(x)$ are such that $D_h(u, v) \geq D_h(u, \partial \AA_{3r/4, 5r/4}(x))$.
          Then
          \begin{align*}
            D_h\left(u, v \right) \
            &> \ D_h\left(\partial \AA_{7r/8, r}(x), \partial \AA_{3r/4, 5r/4}(x) \right) \
            \geq \ s r^{\xi Q} e^{\xi h_r(x)} \tag*{(by \eqref{eq:Condition2LQGAcross})},
          \end{align*}
          so that \eqref{eq:FarAwayDhDistanceImpliesFarAwayEuclideanDistance} implies $|u - v| \geq b r$.
          If $z \in B_{2 \epsilon^{1 - \zeta} r^{\zeta}}(u)$ and $w \in B_{2 \epsilon^{1 - \zeta} r^{\zeta}}(v)$, then $|z - w| \geq b r - 4 \epsilon^{1 - \zeta} r^{\zeta} > b' r$.
          Therefore, since $B_{2 \epsilon^{1 - \zeta} r^{\zeta}}(\overline{\AA_{\alpha r, r}(x)}) \subset \AA_{\alpha_0 r, \alpha_1 r}(x)$, \eqref{eq:DhDistancesInNarrowTube} implies
          \begin{align}
            &D_h\left(B_{2 \epsilon^{1 - \zeta} r^{\zeta}}(u), B_{2 \epsilon^{1 - \zeta} r^{\zeta}}(v); B_{2 \epsilon^{1 - \zeta} r^{\zeta}}(\overline{\AA_{\alpha r, r}(x)}) \right) \nonumber \\
            &\qquad \qquad \geq \ D_h\left(B_{2 \epsilon^{1 - \zeta} r^{\zeta}}(u), B_{2 \epsilon^{1 - \zeta} r^{\zeta}}(v); \AA_{\alpha_0 r, \alpha_1 r}(x) \right) \nonumber \\
            &\qquad \qquad \geq \ S r^{\xi Q} e^{\xi h_r(x)}.
            \label{eq:DiskDistanceInTube}
          \end{align}
          Assuming $\phi \in \confmaps$ and putting $t \coloneqq 1/|\phi'(x)|$ and $H(\cdot) \coloneqq h(r \cdot + x) - h_r(x)$, we see that
          \begin{align*}
            &\locLFPP[\epsilon][\hphi]\left(\phi(u), \phi(v); \phi\left(\overline{\AA_{\alpha r, r}(x)} \right)\right) \\
            &\geq \ \frac{1}{2} \locLFPP[\epsilon t][h]\left(u, v; \overline{\AA_{\alpha r, r}(x)} \right) \tag*{(by \eqref{eq:Condition2PathwiseComparison})} \\
            &\geq \ \frac{1}{4} \LFPP[\epsilon t][h]\left(u, v; \overline{\AA_{\alpha r, r}(x)} \right) \tag*{(by \eqref{eq:Condition2CompareLFPPAndlocLFPP})} \\
            &= \ \frac{1}{4} \frac{r t \fa_{\epsilon t}^{-1}}{t^{\xi Q} \fa_{\epsilon/r}^{-1}} e^{\xi h_{r}(x)} \LFPP[\epsilon/r][H(t \cdot) + Q \log t]\left(\frac{u'}{t}, \frac{v'}{t}; \frac{1}{t} \overline{\AA_{\alpha, 1}(0)} \right) \tag*{(by \eqref{eq:LFPPScaling})} \\
            &\geq \ C_0^{-1} \frac{1}{4} \frac{r t \fa_{\epsilon t}^{-1}}{t^{\xi Q} \fa_{\epsilon/r}^{-1}} e^{\xi h_r(x)} D_H\left(B_{2 (\epsilon/r)^{1 - \zeta}}(u'), B_{2 (\epsilon/r)^{1 - \zeta}}(v'); B_{2 (\epsilon/r)^{1 - \zeta}}(\overline{\AA_{\alpha, 1}(0)}) \right) \tag*{($\ast$)} \\
            &= \ C_0^{-1} \frac{1}{4} \frac{r t \fa_{\epsilon t}^{-1}}{(r t)^{\xi Q} \fa_{\epsilon/r}^{-1}} D_h\left(B_{2 \epsilon^{1 - \zeta} r^{\zeta}}(u), B_{2 \epsilon^{1 - \zeta} r^{\zeta}}(v); B_{2 \epsilon^{1 - \zeta} r^{\zeta}}(\overline{\AA_{\alpha r, r}(x)} \right) \tag*{(axiom \ref{axiom:CoordinateChange})} \\
            &\geq \ C_0^{-1} \frac{1}{4} \frac{r t \fa_{\epsilon t}^{-1}}{(r t)^{\xi Q} \fa_{\epsilon/r}^{-1}} S r^{\xi Q} e^{\xi h_r(x)} \tag*{(by \eqref{eq:DiskDistanceInTube})} \\
            &> \ 4 \frac{r t \fa_{\epsilon t}^{-1}}{(r t)^{\xi Q} \fa_{\epsilon/r}^{-1}} \invratios[\epsilon t] \locLFPP[\epsilon t][h]\left(u,v; \AA_{3r/4, 5r/4}(x) \right) \tag*{(by \eqref{eq:Condition2SupLFPP})} \\
            &\geq \ 2 \frac{t \fa_{\epsilon t/r}^{-1}}{t^{\xi Q} \fa_{\epsilon/r}^{-1}} \locLFPP[\epsilon][\hphi]\left(\phi(u), \phi(v); \phi\left(\AA_{3r/4, 5r/4}(x) \right)\right) \tag*{(by \eqref{eq:Condition2PathwiseComparison})} \\
            &\geq \ \locLFPP[\epsilon][\hphi]\left(\phi(u), \phi(v); \phi\left(\AA_{3r/4, 5r/4}(x) \right)\right) \tag*{(by \eqref{eq:Condition2RegularVariation})},
          \end{align*}
          where $(\ast)$ is obtained by applying \eqref{eq:Condition2InitialLipschitz} with $\Phi(z) \coloneqq z/t$.
          This shows that the second bullet point in the definition of $E_{r, \epsilon}^{(2)}(x)$ holds.
  
          It remains to prove the existence of $s, S$ such that \eqref{eq:Condition2SupLFPP}, \eqref{eq:Condition2SupLQG}, \eqref{eq:Condition2LFPPAcross}, and \eqref{eq:Condition2LQGAcross} all hold with probability at least $p$.
          By continuity of $D_h$, we can find $0 < s < S$ such that with probability at least $1 - \frac{1 - p}{6}$,
          \begin{align}
            \sup_{u, v \in \AA_{7/8, 1}(0)} D_h\left(u, v; \AA_{3/4, 5/4}(0) \right) \
            &< \ S, \label{eq:Condition2DhInAnnulusAtScale1} \\
            D_h\left(\partial \AA_{7/8, 1}(0), \partial \AA_{3/4, 5/4}(0) \right) \
            &> \ 2 s. \label{eq:Condition2DhDistanceAcrossAtScale1}
          \end{align}
          Then axiom \ref{axiom:CoordinateChange} implies that for each $r > 0$, \eqref{eq:Condition2SupLQG} and \eqref{eq:Condition2LQGAcross} hold with probability at least $1 - \frac{1 - p}{6}$.
          Moreover, \eqref{eq:Condition2DhDistanceAcrossAtScale1}, \eqref{eq:LFPPScaling}, Lemma \ref{lemma:PropertiesOfLocalizedFieldAndLFPP}, and the almost sure convergence of $\LFPP$ to $D_h$ imply we can find $\epsilon_0 > 0$ such that with probability at least $1 - 2 \frac{1 - p}{6}$,
          \begin{align*}
            \locLFPP[\epsilon][h]\left(\partial \AA_{7r/8, r}(x), \partial \AA_{3r/4, 5r/4}(x) \right) \
            &> \ s \ratios r^{\xi Q} e^{\xi h_r(x)} \ \forall \epsilon \in (0, \epsilon_0 \DerivativeBound),
          \end{align*}
          which proves \eqref{eq:Condition2LFPPAcross}.
          
          By continuity, we can find $R > 0$ such that with probability at least $1 - \frac{1 - p}{6}$,
          \begin{align}
            D_h\left(z, w \right) \
            &< \ \frac{1}{3} D_h\left(\partial \AA_{13/16, 17/16}(0), \partial \AA_{3/4, 5/4}(0) \right) \ \forall z,w \in \AA_{3/4, 5/4}(0), |z - w| \leq R. \label{eq:ErrataContinuity}
          \end{align}
          Using the almost sure convergence of $\LFPP$ to $D_h$, we can find $\epsilon_0 > 0$ such that with probability at least $1 - \frac{1 - p}{6}$, 
          \begin{align}
            \sup_{z, w \in \overline{B_2(0)}} \left|\LFPP(z,w) - D_h(z,w) \right| \
            &< \ \frac{1}{3} D_h\left(\partial \AA_{13/16, 17/16}(0), \partial \AA_{3/4, 5/4}(0) \right) \quad \forall \epsilon \in (0, \epsilon_0).
            \label{eq:ErrataConvergence}
          \end{align}
          Note that if \eqref{eq:ErrataContinuity} and \eqref{eq:ErrataConvergence} hold, then
          \begin{align}
            \LFPP\left(\partial \AA_{13/16, 17/16}(0), \partial \AA_{3/4, 5/4}(0) \right) \
            &> \ \frac{2}{3} D_h\left(\partial \AA_{13/16, 17/16}(0), \partial \AA_{3/4, 5/4}(0)\right), \label{eq:ErrataLFPPAcross}
          \end{align}
          so if $z,w \in \AA_{13/16, 17/16}(0)$ with $|z - w| \leq R$, then
          \begin{align}
            \LFPP\left(z, w \right) \
            &< \ D_h\left(z, w \right) + \frac{1}{3} D_h\left(\partial \AA_{13/16, 17/16}(0), \partial \AA_{3/4, 5/4}(0) \right) \nonumber \\
            &< \ \frac{2}{3} D_h\left(\partial \AA_{13/16, 17/16}(0), \partial \AA_{3/4, 5/4}(0) \right) \nonumber \\
            &< \ \LFPP\left(\partial \AA_{13/16, 17/16}(0), \partial \AA_{3/4, 5/4}(0) \right), \label{eq:ErrataRemoveInternal}
          \end{align}
          which implies $\LFPP(z,w) = \LFPP(z,w; \AA_{3/4, 5/4}(0))$.
          Using Lemma \ref{lemma:PropertiesOfLocalizedFieldAndLFPP}, we can decrease $\epsilon_0$ so that with probability at least $1 - \frac{1 - p}{6}$,
          \begin{align}
            \sup_{z \in V} \left|\hat{h}_{\epsilon}^{*}(z) - h_{\epsilon}^{*}(z) \right| \
            &< \ \frac{1}{\xi} \log 2 \ \forall \epsilon \in (0, \epsilon_0). \label{eq:ErrataLocToLFPP}
          \end{align}
          Finally, increase $S$ so that with probability at least $1 - \frac{1 - p}{6}$,
          \begin{align}
            32 C_0 \left[2 + \#((R \ZZ^2) \cap \AA_{13/16, 17/16}(0)) \right] \frac{2}{3} D_h\left(\partial \AA_{13/16, 17/16}(0), \partial \AA_{3/4, 5/4}(0) \right) \
            &< \ S \label{eq:ErrataChoiceOfS}
          \end{align}
          It follows that if $\epsilon \in (0, \epsilon_0)$, $r \in (\epsilon^{1 - \zeta}, \epsilon^{1 - \zeta_{+}})$, and $x \in W$, then with probability at least $p$, if $u, v \in \AA_{7r/8, r}(x)$ and $u' \coloneqq \frac{u - x}{r}$ and $v' \coloneqq \frac{v - x}{r}$, then
          \begin{align*}
            &\locLFPP\left(u, v; \AA_{3r/4, 5r/4}(x) \right) \\
            &\leq \ 2 \LFPP\left(u, v; \AA_{3r/4, 5r/4}(x) \right) \tag*{(by \eqref{eq:ErrataLocToLFPP})} \\
            &= \ 2 \ratios r^{\xi Q} e^{\xi h_r(x)} \LFPP[\epsilon/r][h(r \cdot + x) - h_r(x)]\left(u', v'; \AA_{3/4, 5/4}(x) \right) \tag*{(by \eqref{eq:LFPPScaling})} \\
            &\leq \ 2 \ratios r^{\xi Q} e^{\xi h_r(x)} \#(2 + (R \ZZ^2) \cap \AA_{13/16, 17/16}(0)) \tag*{(triangle inequality)} \\
            &\qquad \qquad \cdot \sup_{z,w \in (R \ZZ^2) \cap \AA_{13/16, 17/16}(0)} \LFPP[\epsilon/r][h(r \cdot + x) - h_r(x)]\left(z,w; \AA_{3/4, 5/4}(0) \right) \\
            &= \ 2 \ratios r^{\xi Q} e^{\xi h_r(x)} \#(2 + (R \ZZ^2) \cap \AA_{13/16, 17/16}(0)) & \tag*{(by \eqref{eq:ErrataRemoveInternal})} \\
            &\qquad \qquad \cdot \sup_{z,w \in (R \ZZ^2) \cap \AA_{13/16, 17/16}(0)} \LFPP[\epsilon/r][h(r \cdot + x) - h_r(x)]\left(z,w \right) \\
            &\leq \ 2 \ratios r^{\xi Q} e^{\xi h_r(x)} \#(2 + (R \ZZ^2) \cap \AA_{13/16, 17/16}(0)) \tag*{(by \eqref{eq:ErrataContinuity}, \eqref{eq:ErrataConvergence})} \\
            &\qquad \qquad \cdot \frac{2}{3} D_{h(r \cdot + x) - h_r(x)}\left(\partial \AA_{13/16, 17/16}(0), \partial \AA_{3/4, 5/4}(0) \right) \\
            &\leq \ \frac{1}{16} C_0^{-1} S r^{\xi Q} e^{\xi h_r(x)}, \tag*{(by \eqref{eq:ErrataChoiceOfS})}
          \end{align*}
          which proves \eqref{eq:Condition2SupLFPP}.
        \end{proof}
      \end{lemma}
  
      \begin{lemma}
        \label{lemma:ImprovingTheLipschitzConstantCondition1}
        For each $\delta, p \in (0, 1)$ and $0 < \zeta < \zeta_{+} < 1$, there exist $\alpha = \alpha(p) \in (7/8, 1)$ and $\epsilon_0 = \epsilon_0(\delta, \zeta, \zeta_{+}, p) > 0$ such that for each $\epsilon \in (0, \epsilon_0)$, each $r \in (\epsilon^{1 - \zeta}, \epsilon^{1 - \zeta_{+}})$, and each $x \in W$, $\Prob[E_{r, \epsilon}^{(1)}(x)] \geq p$.
        \begin{proof}
          The idea is to use Proposition \ref{prop:MainRegularityConditionHighProbability} to compare $\locLFPP[\epsilon][\hphi](\phi(\cdot), \phi(\cdot))$ to $\locLFPP[\epsilon/|\phi'(x)|][h]$ up to a factor of $1 + \delta$, compare $\locLFPP$ to $\LFPP$ using Lemma \ref{lemma:PropertiesOfLocalizedFieldAndLFPP}, rescale by $r$ using \eqref{eq:LFPPScaling} (giving rise to the scaling ratio terms $\ratios[\epsilon t]$ in the definition of $E_{r, \epsilon}^{(1)}(x)$) so that we are working with an annulus with radius not depending on $\epsilon$, then compare $\LFPP$ to $D_h$.

          Fix $\delta' > 0$ such that $e^{\xi \delta'}(1 + \delta')^2 < (1 + \delta)$.
          By Lemma \ref{lemma:ImprovingTheLipschitzConstantCondition2}, we can choose $\alpha \in (7/8, 1)$ and $\epsilon_0 > 0$ so that $P[E_{r, \epsilon}^{(2)}(x; \alpha)] \geq 1 - \frac{1 - p}{5}$ for all $\epsilon \in (0, \epsilon_0)$, $r \in (\epsilon^{1 - \zeta}, \epsilon^{1 - \zeta_{+}})$, and $x \in W$.
          Using Proposition \ref{prop:MainRegularityConditionHighProbability}, we can decrease $\epsilon_0$ so that for each $\epsilon \in (0, \epsilon_0)$, with probability at least $1 - \frac{1 - p}{5}$, for all piecewise $C^1$ paths $P$ in $B_{2 \epsilon^{1 - \zeta_{+}}}(x)$ and all $\phi \in \confmaps$,
          \begin{align}
            \left(1 + \delta' \right)^{-1} \Len\left(P; \locLFPP[\epsilon/|\phi'(x)|][h] \right) \
            &\leq \ \Len\left(\phi \circ P; \locLFPP[\epsilon][\hphi] \right) \nonumber \\
            &\leq \ \left(1 + \delta' \right) \Len\left(\locLFPP[\epsilon/|\phi'(x)|][h] \right). \label{eq:Condition1PathwiseComparison}
          \end{align}
          Using the almost sure convergence of $\locLFPP$ and $\LFPP$ to $D_h$, we can find $R > 0$ and decrease $\epsilon_0$ so that $V \subset B_R(0)$ and that with probability at least $1 - \frac{1 - p}{5}$, 
          \begin{align}
            \locLFPP[\epsilon][h]\left(u,v \right) \
            &= \ \locLFPP[\epsilon][h]\left(u, v; B_R(0) \right) \ \forall u,v \in V, \forall \epsilon \in (0, \epsilon_0 \DerivativeBound) \label{eq:Condition1InternalToNonInternalLoc} \\
            \LFPP[\epsilon][h]\left(u,v \right) \
            &= \ \LFPP[\epsilon][h]\left(u,v; B_R(0) \right) \ \forall u,v \in V, \forall \epsilon \in (0, \epsilon_0 \DerivativeBound). \label{eq:Condition1InternalToNonInternalLFPP}
          \end{align}
          Lemma \ref{lemma:PropertiesOfLocalizedFieldAndLFPP} implies that after shrinking $\epsilon_0$ further, we can ensure that with probability at least $1 - \frac{1 - p}{5}$,
          \begin{align}
            \sup_{z \in B_R(0)} \left|\hat{h}_{\epsilon}^{*}(z) - h_{\epsilon}^{*}(z) \right| \
            &< \ \delta' \ \forall \epsilon \in (0, \epsilon_0 \DerivativeBound). \label{eq:Condition1LFPPToLocalizedLFPP} 
          \end{align}
          Using the almost sure convergence of $\LFPP$ to $D_h$, we can decrese $\epsilon_0$ again so that with probability at least $1 - \frac{1 - p}{5}$,
          \begin{align}
            \sup_{u, v \in B_2(0)} \left|\LFPP[\epsilon][h]\left(u,v \right) - D_h\left(u,v\right) \right| \
            &< \ \frac{\delta'}{\delta' + 1} D_h(\partial B_{\alpha}(0), \partial B_1(0)) \ \forall \epsilon \in (0, \epsilon_0^{\zeta} \DerivativeBound).
            \label{eq:Condition1DistanceComparison}
          \end{align}
          It follows that for each $\epsilon \in (0, \epsilon_0)$, $r \in (\epsilon^{1 - \zeta}, \epsilon^{1 - \zeta_{+}})$, $x \in W$, it holds with probability at least $p$ that if there is a $D_h(\cdot, \cdot; \AA_{3r/4, 5r/4}(x))$-geodesic from $u \in \partial B_{\alpha r}(x)$ to $v \in \partial B_r(x)$ contained in $\overline{\AA_{\alpha r, r}(x)}$, then for all $\phi \in \confmaps$,
          \begin{align*}
            &\locLFPP[\epsilon][\hphi]\left(\phi(u), \phi(v); \phi\left(\AA_{3r/4, 5r/4}(x) \right)\right) \\
            &\leq \ \left(1 + \delta' \right)\locLFPP[\epsilon/|\phi'(x)|][h]\left(u, v; \AA_{3r/4, 5r/4}(x) \right) \tag*{(\text{by \eqref{eq:Condition1PathwiseComparison}})} \\
            &= \ \left(1 + \delta' \right) \locLFPP[\epsilon/|\phi'(x)|][h]\left(u, v \right) \tag*{($E_{r, \epsilon}^{(2)}(x)$ \text{ occurs})} \\
            &= \ \left(1 + \delta' \right) \locLFPP[\epsilon/|\phi'(x)|][h]\left(u, v; B_R(0) \right) \tag*{(\text{by \eqref{eq:Condition1InternalToNonInternalLoc}})} \\
            &\leq \ e^{\xi \delta'} \left(1 + \delta' \right) \LFPP[\epsilon/|\phi'(x)|][h]\left(u, v; B_R(0) \right) \tag*{(\text{by \eqref{eq:Condition1LFPPToLocalizedLFPP}})} \\
            &= \ e^{\xi \delta'} \left(1 + \delta' \right) \LFPP[\epsilon/|\phi'(x)|][h]\left(u, v \right) \tag*{(\text{by \eqref{eq:Condition1InternalToNonInternalLFPP}})} \\
            &= \ e^{\xi \delta'} \left(1 + \delta' \right) \frac{r \fa_{\epsilon/|\phi'(x)|}^{-1}}{r^{\xi Q} \fa_{\epsilon/r|\phi'(x)|}^{-1}} r^{\xi Q} e^{\xi h_r(x)} \LFPP[\epsilon/r|\phi'(x)|][h(r \cdot + x) - h_r(x)]\left(\frac{u - x}{r}, \frac{v - x}{r} \right) \tag*{(\text{by \eqref{eq:LFPPScaling}})} \\
            &\leq \ e^{\xi \delta'} \left(1 + \delta' \right)^2 \frac{r \fa_{\epsilon/|\phi'(x)|}^{-1}}{r^{\xi Q} \fa_{\epsilon/r|\phi'(x)|}^{-1}} r^{\xi Q} e^{\xi h_r(x)} D_{h(r \cdot + x) - h_r(x)}\left(\frac{u - x}{r}, \frac{v - x}{r} \right) \tag*{(\text{by \eqref{eq:Condition1DistanceComparison}})} \\
            &= \ e^{\xi \delta'} \left(1 + \delta' \right)^2 \frac{r \fa_{\epsilon/|\phi'(x)|}^{-1}}{r^{\xi Q} \fa_{\epsilon/r|\phi'(x)|}^{-1}} D_h\left(u, v \right) \tag*{(\text{axiom \ref{axiom:CoordinateChange}})} \\
            &\leq \ \left(1 + \delta \right) \sup_{t \in [\DerivativeBound^{-1}, \DerivativeBound]} \frac{r \fa_{\epsilon t}^{-1}}{r^{\xi Q} \fa_{\epsilon t/r}^{-1}} D_h\left(u, v; \AA_{\alpha r, r}(x) \right).
          \end{align*}
          Likewise, if there is a $\locLFPP[\epsilon][\hphi](\phi(\cdot), \phi(\cdot); \phi(\AA_{3r/4, 5r/4}(x)))$-geodesic from $u$ to $v$ contained in $\overline{\AA_{\alpha r, r}(x)}$, then
          \begin{align*}
            &D_h\left(u, v; \AA_{3r/4, 5r/4}(x) \right) \\
            &= \ D_h\left(u, v \right) \tag*{($E_{r, \epsilon}^{(2)}(x)$ occurs)} \\
            &= \ r^{\xi Q} e^{\xi h_r(x)} D_{h(r \cdot + x) - h_r(x)} \left(\frac{u - x}{r}, \frac{v - x}{r} \right) \tag*{(axiom \ref{axiom:CoordinateChange})} \\
            &\leq \ \left(1 + \delta' \right) r^{\xi Q} e^{\xi h_r(x)} \LFPP[\epsilon/r |\phi'(x)|][h(r \cdot + x) - h_r(x)]\left(\frac{u - x}{r}, \frac{v - x}{r} \right) \tag*{(by \eqref{eq:Condition1DistanceComparison})} \\
            &= \ \left(1 + \delta' \right) \frac{r^{\xi Q} \fa_{\epsilon/r|\phi'(x)|}^{-1}}{r \fa_{\epsilon/|\phi'(x)|}^{-1}} \LFPP[\epsilon/|\phi'(x)|][h]\left(u, v \right) \tag*{(by \eqref{eq:LFPPScaling})} \\
            &\leq \ \left(1 + \delta' \right) \frac{r^{\xi Q} \fa_{\epsilon/r|\phi'(x)|}^{-1}}{r \fa_{\epsilon/|\phi'(x)|}^{-1}} \LFPP[\epsilon/|\phi'(x)|][h]\left(u, v; \AA_{\alpha r, r}(x) \right) \\
            &\leq \ e^{\xi \delta'} \left(1 + \delta' \right) \frac{r^{\xi Q} \fa_{\epsilon/r|\phi'(x)|}^{-1}}{r \fa_{\epsilon/|\phi'(x)|}^{-1}} \locLFPP[\epsilon/|\phi'(x)|][h]\left(u, v; \AA_{\alpha r, r}(x) \right) \tag*{(by \eqref{eq:Condition1LFPPToLocalizedLFPP})} \\
            &\leq \ e^{\xi \delta'} \left(1 + \delta' \right)^2 \frac{r^{\xi Q} \fa_{\epsilon/r|\phi'(x)|}^{-1}}{r \fa_{\epsilon/|\phi'(x)|}^{-1}} \locLFPP[\epsilon][\hphi]\left(\phi(u), \phi(v); \phi\left(\AA_{\alpha r, r}(x) \right)\right) \tag*{(by \eqref{eq:Condition1PathwiseComparison})} \\
            &\leq \ \left(1 + \delta \right) \sup_{t \in [\DerivativeBound^{-1}, \DerivativeBound]} \frac{r^{\xi Q} \fa_{\epsilon t/r}^{-1}}{r \fa_{\epsilon t/}^{-1}} \locLFPP[\epsilon][\hphi]\left(\phi(u), \phi(v); \phi\left(\AA_{\alpha r, r}(x) \right)\right).
          \end{align*}
        \end{proof}
      \end{lemma}
  
      \begin{lemma}
        \label{lemma:ImprovingTheLipschitzConstantCondition3}
        For any $\alpha, p \in (0, 1)$, $\DerivativeBound > 1$, and $0 < \zeta < \zeta_{+} < 1$, there exists $A = A(\alpha, p) > 1$ and $\epsilon_0 = \epsilon_0(p, \alpha, \zeta, \zeta_{+}) > 0$ such that for each $\epsilon \in (0, \epsilon_0)$ and each $r \in (\epsilon^{1 - \zeta}, \epsilon^{1 - \zeta_{+}})$, $\Prob[E_{r, \epsilon}^{(3)}(x)] \geq p$.
        \begin{proof}
          For $D_h$, distances around and across an annulus can be compared by rescaling to order $1$, then using the fact that the distance across an annulus is positive and the distance around is finite.
          The idea is similar for $\locLFPP[\epsilon][\hphi](\phi(\cdot), \phi(\cdot))$, but since the constant $A$ is not allowed to depend on $\epsilon$ or $\phi$, we use Proposition \ref{prop:MainRegularityConditionHighProbability} to replace $\locLFPP[\epsilon][\hphi](\phi(\cdot), \phi(\cdot))$ by $\locLFPP[\epsilon/|\phi'(x)|][h]$, then use convergence of LFPP to $D_h$ to choose $A$.

          By Proposition \ref{prop:MainRegularityConditionHighProbability}, we can choose $\epsilon_0 > 0$ such that for each $\epsilon \in (0, \epsilon_0)$ and each $x \in W$, it holds with probability at least $1 - \frac{1 - p}{4}$ that for all $\phi \in \confmaps$ and all piecewise $C^1$ paths $P$ in $B_{2 \epsilon^{1 - \zeta_{+}}}(x)$,
          \begin{align}
            \frac{1}{2} \Len\left(P; \locLFPP[\epsilon/|\phi'(x)|][h] \right) \
            &\leq \ \Len\left(\phi \circ P; \locLFPP[\epsilon][\hphi] \right) \nonumber \\
            &\leq \ 2 \Len\left(P; \locLFPP[\epsilon/|\phi'(x)|][h] \right).
            \label{eq:Condition3PathwiseComparison}
          \end{align}
          Apply Lemma \ref{lemma:TightnessAroundAnnuli} to choose $A > 0$ and decrease $\epsilon_0$ so that 
          \begin{align}
            \Prob\left\{D_h\left(\text{around } \AA_{\alpha, 1}(0) \right) \vee \sup_{\epsilon \in (0, \epsilon_0^{\zeta} \DerivativeBound)} \LFPP[\epsilon][h]\left(\text{around } \AA_{\alpha, 1}(0) \right) \leq c_1 \right\} \
            &\geq \ 1 - \frac{1 - p}{4}.
            \label{eq:Condition3DistanceAround}
          \end{align}
          By Lemma \ref{lemma:PropertiesOfLocalizedFieldAndLFPP}, we can shrink $\epsilon_0$ so that 
          \begin{align}
            \Prob\left\{\sup_{z \in V} |h_{\epsilon}^{*}(z) - \hat{h}_{\epsilon}^{*}(z)| < 1 \forall \epsilon \in (0, \epsilon_0 \DerivativeBound) \right\} \
            &\geq \ 1 - \frac{1 - p}{4}.
            \label{eq:Condition3LocalToNonLocal}
          \end{align}
          Using almost sure convergence of $\LFPP$ to $D_h$ together with $D_h(\partial B_{\alpha}(0), \partial B_1(0)) > 0$, we can find $c_0 > 0$ and shrink $\epsilon_0$ such that
          \begin{align}
            \Prob\left\{D_h\left(\partial B_{\alpha}(0), \partial B_1(0) \right) \w \inf_{\epsilon \in (0, \epsilon_0^{\zeta} \DerivativeBound)} \LFPP[\epsilon][h]\left(\partial B_{\alpha}(0), \partial B_1(0) \right) \geq c_0^{-1} \right\} \
            &\geq \ 1 - \frac{1 - p}{4}.
            \label{eq:Condition3DistanceAcross}
          \end{align}
          It follows that for each $\epsilon \in (0, \epsilon_0)$, each $x \in W$, and each $r \in (\epsilon^{1 - \zeta}, \epsilon^{1 - \zeta_{+}})$, it holds with probability at least $p$ that 
          \begin{align*}
            D_h\left(\text{around } \AA_{\alpha r, r}(x) \right) \
            &= \ r^{\xi Q} e^{\xi h_r(x)} D_{h(r \cdot + x) - h_r(x)} \left(\text{around } \AA_{\alpha, 1}(0) \right) \tag*{(by axioms \ref{axiom:WeylScaling}, \ref{axiom:CoordinateChange})} \\
            &\leq \ r^{\xi Q} e^{\xi h_r(x)} c_1 c_0 D_{h(r \cdot + x) - h_r(x)}\left(\partial B_{\alpha}(0), \partial B_1(0) \right) \tag*{(by \eqref{eq:Condition3DistanceAround}, \eqref{eq:Condition3DistanceAcross})} \\
            &= \ c_1 c_0 D_h\left(\partial B_{\alpha r}(x), \partial B_r(x) \right) \tag*{(by axioms \ref{axiom:WeylScaling}, \ref{axiom:CoordinateChange})},
          \end{align*}
          and also that for all $\phi \in \confmaps$,
          \begin{align*}
            &\locLFPP[\epsilon][\hphi]\left(\text{around } \phi\left(\AA_{\alpha r, r}(x) \right) \right) \\
            &\leq \ 2 \locLFPP[\epsilon/|\phi'(x)|][h]\left(\text{around } \AA_{\alpha r, r}(x) \right) & \tag*{(by \eqref{eq:Condition3PathwiseComparison})} \\
            &\leq \ 2 e^{\xi} \LFPP[\epsilon/|\phi'(x)|][h]\left(\text{around } \AA_{\alpha r, r}(x) \right) \tag*{(by \eqref{eq:Condition3LocalToNonLocal})} \\
            &= \ 2 e^{\xi} \frac{r \fa_{\epsilon/|\phi'(x)|}^{-1}}{r^{\xi Q} \fa_{\epsilon/r|\phi'(x)|}^{-1}} r^{\xi Q} e^{\xi h_r(x)} \LFPP[\epsilon/r|\phi'(x)|][h(r \cdot + x) - h_r(x)]\left(\text{around } \AA_{\alpha, 1}(0) \right) \tag*{(by \eqref{eq:LFPPScaling})}  \\
            &\leq \ 2 e^{\xi} \frac{r \fa_{\epsilon/|\phi'(x)|}^{-1}}{r^{\xi Q} \fa_{\epsilon/r |\phi'(x)|}^{-1}} r^{\xi Q} e^{\xi h_r(x)} c_1 c_0 \LFPP[\epsilon/r|\phi'(x)|][h(r \cdot + x) - h_r(x)]\left(\partial B_{\alpha}(0), \partial B_1(0) \right) \tag*{(by \eqref{eq:Condition3DistanceAround}, \eqref{eq:Condition3DistanceAcross})} \\
            &= \ 2 e^{\xi} c_1 c_0 \LFPP[\epsilon/|\phi'(x)|][h]\left(\partial B_{\alpha r}(x), \partial B_r(x) \right) \tag*{(by \eqref{eq:LFPPScaling})} \\
            &\leq \ 2 e^{2 \xi} c_1 c_0 \locLFPP[\epsilon/|\phi'(x)|][h]\left(\partial B_{\alpha r}(x), \partial B_r(x) \right) \tag*{(by \eqref{eq:Condition3LocalToNonLocal})} \\
            &\leq \ 4 e^{2 \xi} c_1 c_0 \locLFPP[\epsilon][\hphi]\left(\phi\left(\partial B_{\alpha r}(x) \right), \phi\left(\partial B_r(x) \right)\right) \tag*{(by \eqref{eq:Condition3PathwiseComparison}).}
          \end{align*}
          The desired result therefore holds with $A \coloneqq 4 e^{2 \xi} c_1 c_0$.
        \end{proof}
      \end{lemma}
  
      \begin{lemma}
        \label{lemma:ImprovingTheLipschitzConstantCondition4}
        For each $\delta, p \in (0,1)$, $\alpha \in (7/8, 1)$, and $0 < \zeta_{-} < \zeta < \zeta_{+} < 1$, there exist $\epsilon_0 = \epsilon_0(\zeta_{-}, \zeta, \zeta_{+}, \alpha, p, \delta)\in (0, 1)$ such that for each $\epsilon \in (0, \epsilon_0)$, each $r \in (\epsilon^{1 - \zeta}, \epsilon^{1 - \zeta_{+}})$, and each $x \in W$, $\Prob[E_{r, \epsilon}^{(4)}(x)] \geq p$.
        \begin{proof}
          The main idea is that by rescaling by $r$, it will suffice to show the $D_h$-diameter of $B_{4 \epsilon^{1 - \zeta_{-}}/r}(0)$ is much smaller than the $D_h$-distance across $\AA_{\alpha, 1}(0)$.
          This is true because $r \in (\epsilon^{1 - \zeta}, \epsilon^{1 - \zeta_{+}})$ and $D_h$ is a continuous metric.
          We can transfer this to $\locLFPP[\epsilon][\hphi](\phi(\cdot), \phi(\cdot))$ using Proposition \ref{prop:MainRegularityConditionHighProbability}, \eqref{eq:LFPPScaling}, and the convergence of LFPP to $D_h$.

          By Proposition \ref{prop:MainRegularityConditionHighProbability}, there exists $\epsilon_0 > 0$ so that for each $\epsilon \in (0, \epsilon_0)$ and each $x \in W$, it holds with probability at least $1 - \frac{1 - p}{4}$ that for all $\phi \in \confmaps$ and all piecewise $C^1$ paths $P$ in $B_{2 \epsilon^{1 - \zeta_{+}}}(x)$,
          \begin{align}
            \frac{1}{2} \Len\left(P; \locLFPP[\epsilon/|\phi'(x)|][h] \right) \
            &\leq \ \Len\left(\phi \circ P; \locLFPP[\epsilon][\hphi] \right) \
            \leq \ 2 \Len\left(P; \locLFPP[\epsilon/|\phi'(x)|][h] \right). \label{eq:Condition4PathwiseComparison}
          \end{align}
          Using Lemma \ref{lemma:PropertiesOfLocalizedFieldAndLFPP}, we can shrink $\epsilon_0$ so that with probability at least $1 - \frac{1 - p}{4}$,
          \begin{align}
            \sup_{z \in V} \left|\hat{h}_{\epsilon}^{*}(z) - h_{\epsilon}^{*}(z) \right| \
            &< \ \frac{1}{\xi} \log 2 \ \forall \epsilon \in (0, \epsilon_0).
            \label{eq:Condition4LFPPToLocalizedLFPP}
          \end{align}
          By continuity, we can decrease $\epsilon_0$ further so that with probability at least $1 - \frac{1 - p}{4}$, for all $\epsilon \in (0, \epsilon_0)$,
          \begin{align}
            \sup_{z \in B_2(0)} \sup_{w \in B_{4 \epsilon^{\zeta - \zeta_{-}}}(z)} D_h\left(z, w \right) \
            &< \ \frac{\delta}{12 + \delta} \left(D_h\left(\partial B_{\alpha}(0), \partial B_1(0) \right) \w D_h\left(\partial \AA_{\alpha, 1}(0), \partial \AA_{3/4, 5/4}(0) \right)\right). \label{eq:Condition4Continuity}
          \end{align}
          Using the almost sure convergence of $\LFPP$ to $D_h$, we can shrink $\epsilon_0$ again so that with probability at least $1 - \frac{1 - p}{4}$, for all $\epsilon \in (0, \epsilon_0)$,
          \begin{align}
            \begin{split}
              &\sup_{z, w \in B_2(0)} \left|\LFPP[\epsilon][h]\left(z, w \right) - D_h\left(z, w \right)\right| \\
              &\qquad \qquad < \frac{\delta}{12 + \delta} \left(D_h\left(\partial B_{\alpha}(0), \partial B_1(0) \right) \w D_h\left(\partial \AA_{\alpha, 1}(0), \partial \AA_{3/4, 5/4}(0) \right) \right). 
            \end{split}
            \label{eq:Condition4Convergence}
          \end{align}
  
          It follows that for each $\epsilon \in (0, \epsilon_0)$, $r \in (\epsilon^{1 - \zeta}, \epsilon^{1 - \zeta_{+}})$, and $x \in W$, it holds with probability at least $p$ that for all $u \in \overline{\AA_{\alpha r, r}(x)}$, all $v \in \overline{B_{4 \epsilon^{1 - \zeta_{-}}}(u)}$, and all $\phi \in \confmaps$ that
          \begin{align*}
            &\locLFPP[\epsilon][\hphi]\left(\phi(u), \phi(v); \phi\left(\AA_{3r/4, 5r/4}(x) \right)\right) \\
            &\leq \ 2 \locLFPP[\epsilon/|\phi'(x)|][h]\left(u, v; \AA_{3r/4, 5r/4}(x) \right) \tag*{(by \eqref{eq:Condition4PathwiseComparison})} \\
            &\leq \ 4 \LFPP[\epsilon/|\phi'(x)|][h]\left(u, v; \AA_{3r/4, 5r/4}(x) \right) \tag*{(by \eqref{eq:Condition4LFPPToLocalizedLFPP})} \\
            &= \ 4 \frac{r \fa_{\epsilon/|\phi'(x)|}^{-1}}{r^{\xi Q} \fa_{\epsilon/r|\phi'(x)|}^{-1}} r^{\xi Q} e^{\xi h_r(x)} \LFPP[\epsilon/r|\phi'(x)|][h(r \cdot + x) - h_r(x)]\left(\frac{u - x}{r}, \frac{v - x}{r}; \AA_{3/4, 5/4}(0) \right) \tag*{(by \eqref{eq:LFPPScaling})} \\
            &= \ 4 \frac{r \fa_{\epsilon/|\phi'(x)|}^{-1}}{r^{\xi Q} \fa_{\epsilon/r|\phi'(x)|}^{-1}} r^{\xi Q} e^{\xi h_r(x)} \LFPP[\epsilon/r|\phi'(x)|][h(r \cdot + x) - h_r(x)]\left(\frac{u - x}{r}, \frac{v - x}{r} \right) \tag*{(by \eqref{eq:Condition4Continuity}, \eqref{eq:Condition4Convergence})} \\
            &\leq \ 4 \frac{r \fa_{\epsilon/|\phi'(x)|}^{-1}}{r^{\xi Q} \fa_{\epsilon/r|\phi'(x)|}^{-1}} r^{\xi Q} e^{\xi h_r(x)} \frac{2 \delta}{12 + \delta} D_{h(r \cdot + x) - h_r(x)}\left(\partial B_{\alpha}(0), \partial B_1(0) \right) \tag*{(by \eqref{eq:Condition4Continuity}, \eqref{eq:Condition4Convergence})} \\ 
            &< \ \delta \frac{r \fa_{\epsilon/|\phi'(x)|}^{-1}}{r^{\xi Q} \fa_{\epsilon/r |\phi'(x)|}^{-1}} D_h\left(\partial B_{\alpha r}(x), \partial B_r(x) \right) \tag*{(axiom \ref{axiom:CoordinateChange}).}
          \end{align*}
          Likewise,
          \begin{align*}
            &D_h\left(u, v; \AA_{3r/4, 5r/4}(x) \right) \\
            &= \ r^{\xi Q} e^{\xi h_r(x)} D_{h(r \cdot + x) - h_r(x)}\left(\frac{u - x}{r}, \frac{v - x}{r}; \AA_{3/4, 5/4}(0) \right) \tag*{(axiom \ref{axiom:CoordinateChange})} \\
            &= \ r^{\xi Q} e^{\xi h_r(x)} D_{h(r \cdot + x) - h_r(x)}\left(\frac{u - x}{r}, \frac{v - x}{r} \right) \tag*{(by \eqref{eq:Condition4Continuity})} \\
            &\leq \ r^{\xi Q} e^{\xi h_r(x)} \frac{3 \frac{\delta}{12 + \delta}}{1 - \frac{\delta}{12 + \delta}} \LFPP[\epsilon/r |\phi'(x)|][h(r \cdot + x) - h_r(x)]\left(\partial B_{\alpha}(0), \partial B_1(0) \right) \tag*{(by \eqref{eq:Condition4Continuity}, \eqref{eq:Condition4Convergence})} \\
            &= \ \frac{r^{\xi Q} \fa_{\epsilon/r|\phi'(x)|}^{-1}}{r \fa_{\epsilon/|\phi'(x)|}^{-1}} \frac{\delta}{4} \LFPP[\epsilon/|\phi'(x)|][h]\left(\partial B_{\alpha r}(x), \partial B_r(x) \right) \tag*{(by \eqref{eq:LFPPScaling})} \\
            &\leq \ \frac{r^{\xi Q} \fa_{\epsilon/r|\phi'(x)|}^{-1}}{r \fa_{\epsilon/|\phi'(x)|}^{-1}} \frac{\delta}{2} \locLFPP[\epsilon/|\phi'(x)|][h]\left(\partial B_{\alpha r}(x), \partial B_r(x) \right) \tag*{(by \eqref{eq:Condition4LFPPToLocalizedLFPP})} \\
            &\leq \ \delta \frac{r^{\xi Q} \fa_{\epsilon/r|\phi'(x)|}^{-1}}{r \fa_{\epsilon/|\phi'(x)|}^{-1}} \locLFPP[\epsilon][\hphi]\left(\phi\left(\partial B_{\alpha r}(x) \right), \phi\left(\partial B_r(x) \right) \right) \tag*{(by \eqref{eq:Condition4PathwiseComparison})}.
          \end{align*}
        \end{proof}
      \end{lemma}
      
      \begin{figure}[h!]

        {
          \centering
          \begin{center}
            \includegraphics[scale=0.6]{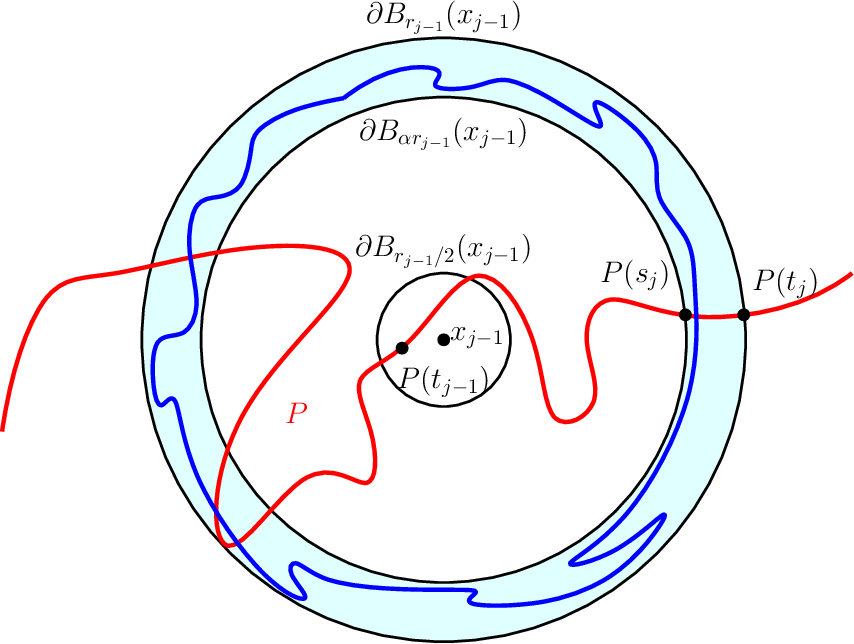}
          \end{center}
        }

        \caption{Illustration of the proof of Proposition \ref{prop:ImprovingTheLipschitzConstant}.
          The cyan annulus is $\AA_{\alpha r_{j-1}, r_{j-1}}(x_{j-1})$, the red path is the $D_h$-geodesic $P$, and the blue path is the path around $\AA_{\alpha r_{j-1}, r_{j-1}}(x_{j-1})$ guaranteed by condition 3 in Definition \ref{defn:ImprovingTheLipschitzConstantRegularityEvents}.
          Time $t_j$ is the first time after $t_{j-1}$ that $P$ exits $B_{r_{j-1}}(x_{j-1})$, and $s_j$ is the last time before $t_j$ that $P$ exits $B_{\alpha r_{j-1}}(x_{j-1})$.
          By condition 1 in Definition \ref{defn:ImprovingTheLipschitzConstantRegularityEvents}, $\locLFPP[\epsilon][\hphi](\phi(P(s_j)), \phi(P(t_j))) \leq \sup_{t \in [\DerivativeBound^{-1}, \DerivativeBound]} \frac{r_{j-1} \fa_{\epsilon t}^{-1}}{r_{j-1}^{\xi Q} \fa_{\epsilon t/r_{j-1}}^{-1}} (1 + \delta) (t_j - s_j)$.
          The blue path has $D_h$-length at most $A(t_j - s_j)$ by condition 3 in Definition \ref{defn:ImprovingTheLipschitzConstantRegularityEvents}.
          Since $P$ is a geodesic, and it crosses the blue path before $t_{j-1}$ and after $s_j$, we must have $s_j - t_{j-1} \leq A(t_j - s_j)$.
          This is used to show a $\frac{1}{A + 1}$-proportion of $P$ is comprised of the segments $P|_{[s_j, t_j]}$.
        }
        \label{fig:Iteration}
      \end{figure}

      \begin{proof}[Proof of Proposition \ref{prop:ImprovingTheLipschitzConstant}]
        The proof is similar to \cite[Proposition 3.6]{ConformalCovariance} and its adaptation in \cite[Proposition 3.1]{AlmostSureConvergence}.
        We will break a geodesic into segments contained in a ``good annulus'' $\AA_{\alpha r, r}(x)$ where $E_{r, \epsilon}(x)$ occurs and segments not contained in a good annulus.
        For the segments in a good annulus, condition 1 in Definition \ref{defn:ImprovingTheLipschitzConstantRegularityEvents} roughly says $\locLFPP[\epsilon][\hphi](\phi(\cdot), \phi(\cdot))$ and $D_h$ are bi-Lipschitz equivalent on $\AA_{\alpha r, r}(x)$ with Lipschitz constant $1 + \delta$ times a term involving scaling ratios $\ratios$.
        For the remaining segments, we can use the Lipschitz constants $C_0, C_0'$ from \eqref{eq:ImprovingTheLipschitzConstantInitialSupLeqLQG} and \eqref{eq:ImprovingTheLipschitzConstantInitialLQGLeqInf} to compare metrics, with condition 4 in Definition \ref{defn:ImprovingTheLipschitzConstantRegularityEvents} used to deal with the error terms from approximating point-to-point distances by ball-to-ball distances.
        Condition 3 in Definition \ref{defn:ImprovingTheLipschitzConstantRegularityEvents} roughly implies that the proportion of any geodesic in good annuli is $\frac{1}{A + 1}$, so we obtain a new Lipschitz condition with constant approximately $\frac{A}{A + 1} C_0 + \frac{1}{A + 1} (1 + \delta) \ratios$.

        We will assume $\epsilon$ is small enough that $B_{4 \epsilon^{1 - \zeta_{+}}}(\tilde{W}) \subset V$.
        Let $F_{\epsilon}$ be the event that for each $x \in B_{\epsilon^{1 - \zeta_{+}}}(\tilde{W}) \cap (\frac{\epsilon^{1 - \zeta}}{4} \ZZ^2)$ there exists $r \in \cR_{\epsilon}$ such that $E_{r, \epsilon}(x)$ occurs.
        Let $H_{\epsilon}$ be the event that \eqref{eq:ImprovingTheLipschitzConstantInitialSupLeqLQG} and \eqref{eq:ImprovingTheLipschitzConstantInitialLQGLeqInf} hold for all $\phi \in \confmaps$.
        By a union bound, $\Prob[F_{\epsilon} \cap H_{\epsilon}] \geq 1 - O_{\epsilon}(\epsilon^{\beta \w 2(1 - \zeta)})$.
  
        Now fix $\epsilon$ and assume $F_{\epsilon} \cap H_{\epsilon}$ occurs, and we will prove that \eqref{eq:ImprovingTheLipschitzConstantImprovedSupLeqLQG} and \eqref{eq:ImprovingTheLipschitzConstantImprovedLQGLeqInf} hold.
        As the proofs of \eqref{eq:ImprovingTheLipschitzConstantImprovedSupLeqLQG} and \eqref{eq:ImprovingTheLipschitzConstantImprovedLQGLeqInf} are nearly identical, we will prove only \eqref{eq:ImprovingTheLipschitzConstantImprovedSupLeqLQG}.
        Note that if $C_0(\epsilon) < \sup_{r \in \cR_{\epsilon}} \sup_{t \in [\DerivativeBound^{-1}, \DerivativeBound]} \ratios[\epsilon t]$, then $C_1(\epsilon) \geq \sup_{r \in \cR_{\epsilon}} \sup_{t \in [\DerivativeBound^{-1}, \DerivativeBound]} \ratios[\epsilon t] > C_0(\epsilon)$ and so \eqref{eq:ImprovingTheLipschitzConstantImprovedSupLeqLQG} holds because \eqref{eq:ImprovingTheLipschitzConstantInitialSupLeqLQG} holds.
        We may therefore assume $C_0(\epsilon) \geq \sup_{r \in \cR_{\epsilon}} \sup_{t \in [\DerivativeBound^{-1}, \DerivativeBound]} \ratios[\epsilon t]$.
        
        Fix $(z,w) \in \hyperlink{G}{G(W, \tilde{W})}$ and $\phi \in \confmaps$.
        Then there is a $D_h$-geodesic $P$ from $z$ to $w$ contained in $\tilde{W}$, and we may assume $P$ is parameterized by $D_h$-length.
        Let $t_0 = 0$ and choose $x_0 \in (\frac{\epsilon^{1 - \zeta}}{4} \ZZ^2) \cap B_{\epsilon^{1 - \zeta_{+}}}(\tilde{W})$ and $r_0 \in \cR_{\epsilon}$ such that $z \in B_{r_0/2}(x_0)$ and $E_{r_0, \epsilon}(x_0)$ occurs.
        Then inductively define $t_j$ to be the first time after time $t_{j-1}$ that $P$ exits $B_{r_{j-1}}(x_{j-1})$, and choose $x_j \in (\frac{\epsilon^{1 - \zeta}}{4} \ZZ^2) \cap B_{\epsilon^{1 - \zeta}}(\tilde{W})$ and $r_j \in \cR_{\epsilon}$ such that $P(t_j) \in B_{r_j/2}(x_j)$ and $E_{r_j, \epsilon}(x_j)$ occurs.
        If no such time $t_j$ exists, leave $t_j$ undefined.
        If $t_j$ does exist, let $s_j$ denote the last time before $t_j$ that $P$ exits $B_{\alpha r_{j-1}}(x_{j-1})$.
        See Fig \ref{fig:Iteration} for an illustration.
        Define
        \begin{align*}
          \underline{J} \
          &\coloneqq \ \max\left\{j \geq 1 : |P(t_{j-1}) - z| \leq 2 \epsilon^{1 - \zeta_{+}} \right\}, \\
          \overline{J} \
          &\coloneqq \ \min\left\{j \geq 0 : |P(t_{j+1}) - w| \leq 2 \epsilon^{1 - \zeta_{+}} \right\}.
        \end{align*}
  
        Assuming $j \in [\underline{J} + 1, \overline{J}] \cap \ZZ$, then $P|_{[s_j, t_j]}$ is a $D_h$-geodesic from $P(s_j) \in \partial B_{\alpha r_j}(x_j)$ to $P(t_j) \in \partial B_{r_j}(x_j)$ contained in $\overline{\AA_{\alpha r_j, r_j}(x_j)}$, so by condition 1 in Definition \ref{defn:ImprovingTheLipschitzConstantRegularityEvents},
        \begin{align*}
          \locLFPP[\epsilon][\hphi]\left(\phi(z), \phi(w); \phi\left(\AA_{3r_j/4, 7r_j/4}(x_j) \right)\right) \
          &\leq \ \sup_{t \in [\DerivativeBound^{-1}, \DerivativeBound]} \frac{r_j \fa_{\epsilon t}^{-1}}{r_j^{\xi Q} \fa_{\epsilon t/r_j}^{-1}} \left(1 + \delta \right) \left(t_j - s_j \right).
        \end{align*}
        Assuming $j \in [\underline{J}, \overline{J} + 1] \cap \ZZ$, condition 4 in Definition \ref{defn:ImprovingTheLipschitzConstantRegularityEvents} and \eqref{eq:ImprovingTheLipschitzConstantInitialSupLeqLQG} imply
        \begin{align*}
          &\locLFPP[\epsilon][\hphi]\left(\phi(P(t_j)), \phi(P(s_{j+1})); \phi\left(B_{4 \epsilon^{1 - \zeta_{+}}}(\tilde{W}) \right) \right) \\
          &\qquad \leq \ \locLFPP[\epsilon][\hphi]\left(\phi\left(B_{4 \epsilon^{1 - \zeta_{-}}}(P(t_j)) \right), \phi\left(B_{4 \epsilon^{1 - \zeta_{-}}}(P(s_{j+1})) \right); \phi\left(B_{4 \epsilon^{1 - \zeta_{-}}}(\tilde{W}) \right)\right) \\
          &\qquad \qquad + \sup_{\omega \in \partial B_{4 \epsilon^{1 - \zeta_{-}}}(P(t_j))} \locLFPP[\epsilon][\hphi]\left(\phi(P(t_j)), \phi(\omega); \phi\left(\AA_{3 r_{j-1}/4, 5r_{j-1}/4}(x_{j-1}) \right)\right) \\
          &\qquad \qquad + \sup_{\omega \in \partial B_{4 \epsilon^{1 - \zeta_{-}}}(P(s_{j+1}))} \locLFPP[\epsilon][\hphi]\left(\phi(P(s_{j+1})), \phi(\omega); \phi\left(\AA_{3 r_{j}/4, 5r_{j}/4}(x_{j}) \right)\right) \\
          &\qquad \leq \ C_0(\epsilon) \left(s_{j+1} - t_j \right) \\
          &\qquad \qquad + \delta \sup_{t \in [\DerivativeBound^{-1}, \DerivativeBound]} \frac{r_{j-1} \fa_{\epsilon t}^{-1}}{r_{j-1}^{\xi Q} \fa_{\epsilon t/r_{j-1}}^{-1}} D_h\left(\partial B_{\alpha r_{j-1}}(x_{j-1}), \partial B_{r_{j-1}}(x_{j-1}) \right) \\
          &\qquad \qquad + \delta \sup_{t \in [\DerivativeBound^{-1}, \DerivativeBound]} \frac{r_{j} \fa_{\epsilon t}^{-1}}{r_{j}^{\xi Q} \fa_{\epsilon t/r_{j}}^{-1}} D_h\left(\partial B_{\alpha r_{j}}(x_{j}), \partial B_{r_{j}}(x_{j}) \right).
        \end{align*}
        If $j \in [\underline{J} - 1, \overline{J} - 1] \cap \ZZ$, then since $P(t_j) \in B_{r_j/2}(x_j)$ and $P(t_{j+1}) \in \partial B_{r_j}(x_j)$, we know $P$ crosses $\AA_{\alpha r_j, r_j}(x_j)$ at least once between times $t_{j}$ and $t_{j+1}$.
        It follows that
        \begin{align*}
          &\locLFPP[\epsilon][\hphi]\left(\phi(P(t_j)), \phi(P(s_{j+1})); \phi\left(B_{4 \epsilon^{1 - \zeta_{+}}}(\tilde{W}) \right) \right) \\
          &\qquad \qquad \leq \ C_0(\epsilon) \left(s_{j+1} - t_j \right) + \delta \sup_{t \in [\DerivativeBound^{-1}, \DerivativeBound]} \frac{r_{j-1} \fa_{\epsilon t}^{-1}}{r_{j-1}^{\xi Q} \fa_{\epsilon t/r_{j-1}}^{-1}} \left(t_j - t_{j-1} \right) \\
          &\qquad \qquad \qquad \qquad + \delta \sup_{t \in [\DerivativeBound^{-1}, \DerivativeBound]} \frac{r_j \fa_{\epsilon t}^{-1}}{r_j^{\xi Q} \fa_{\epsilon t/r_j}^{-1}} \left(t_{j+1} - t_j \right).
        \end{align*}
        Therefore, by the triangle inequality,
        \begin{align*}
          &\locLFPP[\epsilon][\hphi]\left(\phi\left(P(t_{\underline{J}}) \right), \phi\left(P(t_{\overline{J}}) \right); \phi\left(B_{4 \epsilon^{1 - \zeta_{+}}}(\tilde{W}) \right) \right) \\
          &\qquad \leq \ \sum_{j=\underline{J}}^{\overline{J} - 1} \locLFPP[\epsilon][\hphi]\left(\phi\left(P(t_j) \right) \phi\left(P(s_{j+1}) \right); \phi\left(B_{4 \epsilon^{1 - \zeta_{+}}}(\tilde{W}) \right)\right) \\
          &\quad \qquad \ + \sum_{\underline{J} + 1}^{\overline{J}} \locLFPP[\epsilon][\hphi]\left(\phi\left(P(s_j) \right), \phi\left(P(t_j) \right); \phi\left(B_{4 \epsilon^{1 - \zeta_{+}}}(\tilde{W}) \right)\right) \\
          &\qquad \leq \ C_0(\epsilon) \sum_{j=\underline{J}}^{\overline{J} - 1} \left(s_{j+1} - t_j \right) + \delta \sup_{r \in \cR_{\epsilon}} \sup_{t \in [\DerivativeBound^{-1}, \DerivativeBound]} \frac{r \fa_{\epsilon t}^{-1}}{r^{\xi Q} \fa_{\epsilon t/r}^{-1}} \sum_{j=\underline{J}}^{\overline{J} - 1} \left(t_j - t_{j-1} \right) \\
          &\qquad \qquad + \delta \sup_{r \in \cR_{\epsilon}} \sup_{t \in [\DerivativeBound^{-1}, \DerivativeBound]} \frac{r \fa_{\epsilon t}^{-1}}{r^{\xi Q} \fa_{\epsilon t/r}^{-1}} \sum_{j=\underline{J}}^{\overline{J} - 1}\left(t_{j+1} - t_j \right) \\
          &\qquad \qquad + (1 + \delta) \sup_{r \in \cR_{\epsilon}} \sup_{t \in [\DerivativeBound^{-1}, \DerivativeBound]} \ratios[\epsilon t] \sum_{j=\underline{J} + 1}^{\overline{J}} \left(t_j - s_j \right) \\
          &\leq \ \left[C_0(\epsilon) - \sup_{r \in \cR_{\epsilon}} \sup_{t \in [\DerivativeBound^{-1}, \DerivativeBound]} \ratios[\epsilon t] \right] \sum_{j=\underline{J}}^{\overline{J} - 1} \left(s_{j+1} - t_j \right) \\
          &\qquad \qquad + \delta \sup_{r \in \cR_{\epsilon}} \sup_{t \in [\DerivativeBound^{-1}, \DerivativeBound]} \ratios[\epsilon t] \sum_{j=\underline{J}}^{\overline{J} - 1} \left(t_j - t_{j-1} \right) \\
          &\qquad \qquad + \delta \sup_{r \in \cR_{\epsilon}} \sup_{t \in [\DerivativeBound^{-1}, \DerivativeBound]} \frac{r \fa_{\epsilon t}^{-1}}{r^{\xi Q} \fa_{\epsilon t/r}^{-1}} \sum_{j=\underline{J}}^{\overline{J} - 1}\left(t_{j+1} - t_j \right) \\
          &\qquad \qquad + (1 + \delta) \sup_{r \in \cR_{\epsilon}} \sup_{t \in [\DerivativeBound^{-1}, \DerivativeBound]} \ratios[\epsilon t] \sum_{j=\underline{J}}^{\overline{J} - 1} \left(t_{j+1} - t_j \right).
        \end{align*}
        If $j \in [\underline{J} + 1, \overline{J}] \cap \ZZ$, then $|P(t_{j-1}) - z| > 2 \epsilon^{1 - \zeta_{+}} > 2 r_{j-1}$ and $|P(t_{j-1}) - x_{j-1}| < r_{j-1}/2$, so $|z - x_{j-1}| > r_{j-1}$.
        Similarly, $|w - x_{j-1}| > r_{j-1}$.
        So $P$ must cross $\AA_{\alpha r_{j-1}, r_{j-1}}(x_{j-1})$ at least once before time $t_{j-1}$ and at least once after time $s_j$.
        Therefore, by condition 3 in Definition \ref{defn:ImprovingTheLipschitzConstantRegularityEvents},
        \begin{align*}
          s_{j} - t_{j-1} \
          &\leq \ A D_h\left(\partial B_{\alpha r_{j-1}}(x_{j-1}), \partial B_{r_{j-1}}(x_{j-1}) \right) \
          \leq \ A \left(t_j - s_j \right).
        \end{align*}
        Adding $A (s_j - t_{j-1})$ to both sides, we obtain
        \begin{align*}
          s_j - t_{j-1} \
          &\leq \ \frac{A}{A+1} \left(t_j - t_{j-1} \right).
        \end{align*}
        Combining the above estimates together with the fact that $|P(t_{\underline{J}}) - z| \leq 4 \epsilon^{1 - \zeta_{+}}$ and $|P(t_{\overline{J}}) - w| \leq 4 \epsilon^{1 - \zeta_{+}}$, we obtain
        \begin{align*}
          &\locLFPP[\epsilon][\hphi]\left(\phi\left(B_{4 \epsilon^{1 - \zeta_{+}}}(z) \right), \phi\left(B_{4 \epsilon^{1 - \zeta_{+}}}(w) \right); \phi\left(B_{4 \epsilon^{1 - \zeta_{+}}}(\tilde{W}) \right) \right) \\
          &\qquad \leq \ \left[C_0(\epsilon) - \sup_{r \in \cR_{\epsilon}} \sup_{t \in [\DerivativeBound^{-1}, \DerivativeBound]} \ratios[\epsilon t] \right] \frac{A}{A+1} \sum_{j=\underline{J}}^{\overline{J} - 1} \left(t_{j+1} - t_j \right) \\ 
          &\qquad \qquad + \delta \sup_{r \in \cR_{\epsilon}} \sup_{t \in [\DerivativeBound^{-1}, \DerivativeBound]} \ratios[\epsilon t] \sum_{j=\underline{J}}^{\overline{J} - 1} \left(t_j - t_{j-1} \right) \\
          &\qquad \qquad + \delta \sup_{r \in \cR_{\epsilon}} \sup_{t \in [\DerivativeBound^{-1}, \DerivativeBound]} \frac{r \fa_{\epsilon t}^{-1}}{r^{\xi Q} \fa_{\epsilon t/r}^{-1}} \sum_{j=\underline{J}}^{\overline{J} - 1}\left(t_{j+1} - t_j \right) \\
          &\qquad \qquad + (1 + \delta) \sup_{r \in \cR_{\epsilon}} \sup_{t \in [\DerivativeBound^{-1}, \DerivativeBound]} \ratios[\epsilon t] \sum_{j=\underline{J}}^{\overline{J} - 1} \left(t_{j+1} - t_j \right) \\
          &\qquad \leq \ \left[\frac{A}{A + 1} C_0(\epsilon) + \left(\frac{1}{A + 1} + 3 \delta \right) \sup_{r \in \cR_{\epsilon}} \sup_{t \in [\DerivativeBound^{-1}, \DerivativeBound]} \ratios[\epsilon t] \right] D_h\left(z,w; \tilde{W} \right). 
        \end{align*}
      \end{proof}
  
    \subsection{Convergence of coordinate-changed LFPP}
      \label{section:ConvergenceOnDiagonal}
      Our goal now is to show that the coordinate-changed metrics $\locLFPP[\epsilon][h^{\phi}](\phi(\cdot), \phi(\cdot); \phi(V))$ for all $\phi \in \confmaps$ converge uniformly to $D_h(\cdot, \cdot; V)$ on a neighbourhood of the diagonal simultaneously.
      The idea is to iteratively apply Proposition \ref{prop:ImprovingTheLipschitzConstant} to get a sequence of Lipschitz constants $C_j(\epsilon)$ depending on the scaling ratios $\ratios[\epsilon t]$.
      We will argue these ratios are close to $1$ when $\cR_{\epsilon}$ is chosen well (Lemma \ref{lemma:GoodScalingRatios}), and that this implies $C_j(\epsilon) < 1 + \delta$ for $j$ large enough.
      Since the Lipschitz condition in Proposition \ref{prop:ImprovingTheLipschitzConstant} holds with polynomially high probability as $\epsilon \to 0$, Borel-Cantelli applied with $\epsilon = 2^{-n}$ shows we get an almost sure approximate Lipschitz condition with Lipschitz constant $1 + \delta$ for any fixed $\delta > 0$.
      Sending $\delta \to 0$ lets us deduce almost sure convergence.

      Throughout, $\xi < \xi_{\text{crit}}$.
      Let us state the main result of this section.
      To ease notation, if $\rho > 0$ and $W \subset \CC$, define
      \begin{align*}
        \hypertarget{Delta}{\Delta_{\rho}(W)} \
        &\coloneqq \ \left\{(z,w) \in W^2 : |z - w| \leq \rho \right\}.
      \end{align*}

      \begin{thm}
        \label{thm:ConvergenceTheorem}
        Let $h$ be a whole-plane GFF.
        Fix open sets $W \Subset V \Subset U$ and $\DerivativeBound > 1$.
        There exists random $\rho = \rho_{W, V, U, \DerivativeBound} > 0$ such that almost surely, $\locLFPP[\epsilon][\hphi](\phi(\cdot), \phi(\cdot); \phi(V))$ converges to $D_h(\cdot, \cdot; V)$ uniformly on $\diagonal$ uniformly over $\phi \in \confmaps$ as $\epsilon \to 0$ along $\{2^{-n}\}_{n \in \NN}$.
      \end{thm}

      We remark that the simultaneous convergence of $\locLFPP[\epsilon][\hphi](\phi(\cdot), \phi(\cdot))$ for all $\phi \in \confmaps$ is only shown along a discrete subsequence $\{2^{-j}\}_{j \in \NN}$.
      To prove convergence along $(0, 1)$, one would need to compare $(h \circ \phi^{-1})_{\epsilon}^{*}(\phi(z))$ with $(h \circ \phi^{-1})_{\delta}^{*}(\phi(z))$ when $\epsilon, \delta$ are small and $\frac{\epsilon}{\delta} \approx 1$.
      This is not difficult for a fixed $\phi$, but doing so for all $\phi$ simultaneously seems to require new ideas.

      Before proving Theorem \ref{thm:ConvergenceTheorem}, we will prove a few lemmas.
      Our first lemma shows, roughly, that no more than a $\frac{2}{3}$ proportion of $r \in (\epsilon^{1 - \zeta}, \epsilon^{1 - \zeta_{+}}) \cap \{8^{-j}\}_{j \in \NN}$ have $\sup_{t \in [\DerivativeBound^{-1}, \DerivativeBound]} \ratios[\epsilon t]$ or $\sup_{t \in [\DerivativeBound^{-1}, \DerivativeBound]} \invratios[\epsilon t]$ far from $1$.

      Let us introduce some notation needed for the lemma statement.
      For a sequence $0 < \zeta(0) < \zeta_{+}(0) = \zeta_{-}(1) < \zeta(1) < \zeta_{+}(1) = \zeta_{-}(2) < \zeta(2) < \zeta_{+}(2) = \zeta_{-}(3) < \cdots < 1$ with $1 - \zeta_{+}(n) < \frac{1}{2} (1 - \zeta(n))$ for all $n$, define
      \begin{align}
        \cS_{\epsilon}^{(n)} \
        &\coloneqq \ \left\{\cR \subset (\epsilon^{1 - \zeta(n)}, \epsilon^{1 - \zeta_{+}(n)}) \cap \{8^{-j}\}_{j \in \NN} : \# \cR > \frac{1}{3} \#((\epsilon^{1 - \zeta(n)}, \epsilon^{1 - \zeta_{+}(n)}) \cap \{8^{-j}\}_{j \in \NN}) \right\}, \nonumber \\
        X_{\epsilon}^{(n)} \
        &\coloneqq \ \min_{\cR \in \cS_{\epsilon}^{(n)}} \sup_{r \in \cR} \sup_{t \in [\DerivativeBound^{-1}, \DerivativeBound]} \ratios[\epsilon t], \label{eq:BestScalingRatios} \\
        Y_{\epsilon}^{(n)} \
        &\coloneqq \ \min_{\cR \in \cS_{\epsilon}^{(n)}} \sup_{r \in \cR} \sup_{t \in [\DerivativeBound^{-1}, \DerivativeBound]} \invratios[\epsilon t]. \label{eq:BestInverseScalingRatios}
      \end{align}
      By Lemma \ref{lemma:APrioriScalingRatiosBound}, there exist $\cR_{\epsilon}^{(0)} \subset (\epsilon^{1 - \zeta(0)}, \epsilon^{1 - \zeta_{+}(0)}) \cap \{8^{-j}\}_{j \in \NN}$ and $L_0 = L_0(\zeta(0), \zeta_{+}(0)) \in (0, \infty)$ with $\# \cR_{\epsilon}^{(0)} > \frac{1}{3} \# ((\epsilon^{1 - \zeta(0)}, \epsilon^{1 - \zeta_{+}(0)}) \cap \{8^{-j}\}_{j \in \NN}$ such that
      \begin{align}
        \limsup_{\epsilon \to 0} \sup_{r \in \cR_{\epsilon}^{(0)}} \sup_{t \in [\DerivativeBound^{-1}, \DerivativeBound]} \ratios[\epsilon t] \
        &< \ L_0, \label{eq:FiniteLimsupOfBestScalingRatios} \\
        \limsup_{\epsilon \to 0} \sup_{r \in \cR_{\epsilon}^{(0)}} \sup_{t \in [\DerivativeBound^{-1}, \DerivativeBound]} \invratios[\epsilon t] \
        &< \ L_0. \label{eq:FiniteLimsupOfBestInverseScalingRatios}
      \end{align}
      The following lemma is used to control the scaling ratios $\ratios[\epsilon t]$.

      \begin{lemma}
        \label{lemma:GoodScalingRatios}
        Fix a sequence $0 < \zeta(0) < \zeta_{+}(0) = \zeta_{-}(1) < \zeta(1) < \zeta_{+}(1) = \zeta_{-}(2) < \zeta(2) < \zeta_{+}(2) = \zeta_{-}(3) < \cdots < 1$ with $1 - \zeta_{+}(n) < \frac{1}{2} (1 - \zeta(n))$ for all $n$.
        Define $X_{\epsilon}^{(n)}$ and $Y_{\epsilon}^{(n)}$ by \eqref{eq:BestScalingRatios} and \eqref{eq:BestInverseScalingRatios}, and let $L_0 \in (0, \infty)$ be such that \eqref{eq:FiniteLimsupOfBestScalingRatios} and \eqref{eq:FiniteLimsupOfBestInverseScalingRatios} are both true.
        Let $C_0$ be as in Proposition \ref{prop:InitialLipschitzCondition} with $(\zeta(0), \zeta_{+}(0))$ in place of $(\zeta, \zeta_{+})$, and let $A$ be as in Proposition \ref{prop:ImprovingTheLipschitzConstant}.
        Fix $\delta \in (0,1)$, and choose $N$ large enough that $(\frac{A}{A + 1})^N C_0 L_0 < \frac{\delta/4}{1 + \delta}$.
        There exists $\epsilon_N \in (0, 1)$ such that for all $\epsilon \in (0, \epsilon_N)$,
        \begin{align}
          \#\left\{1 \leq n \leq 3N : X_{\epsilon}^{(n)} > \frac{1}{1 + \delta} \right\} > 2N, \label{eq:LargeProportionGoodScalingRatios} \\
          \#\left\{1 \leq n \leq 3N : Y_{\epsilon}^{(n)} > \frac{1}{1 + \delta} \right\} > 2 N. \label{eq:LargeProportionGoodInverseScalingRatios}
        \end{align}
        \begin{proof}
          The argument is essentially the same as \cite[Lemma 4.19]{AlmostSureConvergence}, so we just give a brief overview.
          The idea is that if the claim is false, then there is a set $S$ of $N$ indices in $\{1,2, \ldots, 3N\}$ such that along a sequence of $\epsilon$'s with limit $0$, we have, say, $Y_{\epsilon}^{(n)} \leq \frac{1}{1 + \delta}$ for all $n \in S$.
          If we iteratively apply Proposition \ref{prop:ImprovingTheLipschitzConstant} with $(\zeta_{-}, \zeta, \zeta_{+}) = (\zeta_{-}(n), \zeta(n), \zeta_{+}(n))$ and with $\cR_{\epsilon}$ a minimizer of $Y_{\epsilon}^{(n)}$, then we will obtain a near-Lipschitz condition between $D_h$ and $\locLFPP$ with Lipschitz constant strictly less than $1$ (provided $N$ is large enough).
          This contradicts the fact that $\locLFPP \to D_h$.
        \end{proof}
      \end{lemma}

      Using Lemma \ref{lemma:GoodScalingRatios}, we can now make precise the idea that $C_j(\epsilon), C_j'(\epsilon) < 1 + \delta$ for all $j$ sufficiently large.

      \begin{prop}
        \label{prop:LipschitzConstant1PlusDelta}
        Let $W \Subset \tilde{W} \Subset V \Subset U$ be open sets.
        There exists $\beta = \beta(\delta) > 0$ and $\zeta = \zeta(\delta) \in (0, 1)$ such that with probability $1 - O_{\epsilon}(\epsilon^{\beta})$ as $\epsilon \to 0$, for all $\phi \in \confmaps$,
        \begin{align}
          &\locLFPP[\epsilon][\hphi]\left(\phi\left(B_{4 \epsilon^{1 - \zeta}}(z) \right), \phi\left(B_{4 \epsilon^{1 - \zeta}}(w) \right); \phi\left(B_{4 \epsilon^{1 - \zeta}}(\tilde{W}) \right)\right) \nonumber \\
          &\qquad \qquad \leq \ \left(1 + \delta \right) D_h\left(z,w; \tilde{W} \right) \ \forall (z,w) \in \hyperlink{G}{G(W, \tilde{W})}, \label{eq:1PlusDeltaLipschitzConstantLFPPLeqLQG} \\
          &D_h\left(B_{4 \epsilon^{1 - \zeta}}(z), B_{4 \epsilon^{1 - \zeta}}(w); B_{4 \epsilon^{1 - \zeta}}(\tilde{W}) \right) \nonumber \\
          &\qquad \qquad \leq \ \left(1 + \delta \right) \locLFPP[\epsilon][\hphi]\left(\phi(z), \phi(w); \phi(\tilde{W})\right) \ \forall (z,w) \in \hyperlink{Gphi}{G_{\epsilon}^{\phi}(W, \tilde{W})}. \label{eq:1PlusDeltaLipschitzConstantLQGLeqLFPP}
        \end{align}
        \begin{proof}
          Choose a sequence $0 < \zeta(0) < \zeta_{+}(0) = \zeta_{-}(1) < \zeta(1) < \zeta_{+}(1) = \zeta_{-}(2) < \zeta(2) < \zeta_{+}(2) = \zeta_{-}(3) < \cdots$ such that $1 - \zeta_{+}(n) < \frac{1}{2} (1 - \zeta(n))$ for all $n$.
          Let $C_0$ be as in Proposition \ref{prop:InitialLipschitzCondition}.
          Let $A$ be as in Proposition \ref{prop:ImprovingTheLipschitzConstant}.
          Let $L_0$ be as in \eqref{eq:FiniteLimsupOfBestScalingRatios} and \eqref{eq:FiniteLimsupOfBestInverseScalingRatios}.
          Choose $N$ large enough that $(\frac{A}{A + 1})^N C_0 L_0 < \frac{\delta/4}{1 + \delta}$.
          Adopt notation from Lemma \ref{lemma:GoodScalingRatios} and the paragraph preceding it.
          By Lemma \ref{lemma:GoodScalingRatios}, for each $\epsilon \in (0, \epsilon_N)$, the set
          \begin{align*}
            I_{\epsilon} \
            &\coloneqq \ \left\{1 \leq n \leq 3 N : X_{\epsilon}^{(n)} \w Y_{\epsilon}^{(n)} > \frac{1}{1 + \delta} \right\}
          \end{align*}
          has cardinality at least $N$.
          Lemma \ref{lemma:RegularlyVaryingScalingConstants}, implies we can shrink $\epsilon_N$ so that for each $\epsilon \in (0, \DerivativeBound \epsilon_N^{\zeta(N)})$ and each $t \in [\DerivativeBound^{-2}, \DerivativeBound^2]$,
          \begin{align}
            \frac{1}{\sqrt{1 + \delta}} t^{1 - \xi Q} \
            &\leq \ \frac{\fa_{\epsilon t}}{\fa_{\epsilon}} \
            \leq \ \sqrt{1 + \delta} t^{1 - \xi Q}.
            \label{eq:ActualGoodLipschitzConstantRegularVariation}
          \end{align}
          For each $S \subset \{1, 2, \ldots, 3N\}$ with $\# S \geq N$, let 
          \begin{align*}
            \cE_S \
            &\coloneqq \ \left\{\epsilon \in (0, \epsilon_N) : S \subset I_{\epsilon} \right\}.
          \end{align*}
          Note that $\epsilon \in \cE_{I_{\epsilon}}$ for each $\epsilon \in (0, \epsilon_N)$, so $(0, \epsilon_N)$ is the union of the $\cE_S$ as $S$ ranges over all subsets of $\{1, 2, \ldots, 3N\}$ with cardinality at least $N$.

          Fix $S \subset \{1, 2, \ldots, 3N\}$ with $\# S \geq N$ such that $\inf \cE_S = 0$.
          We will need the following lemma.
          \begin{lemma}
            For each $n \in S$ and $\epsilon \in \cE_S$, the set
            \begin{align*}
              \cR_{\epsilon}^{(n)} \
              &\coloneqq \ \left\{r \in (\epsilon^{1 - \zeta(n)}, \epsilon^{1 - \zeta_{+}(n)}) \cap \{8^{-j}\}_{j \in \NN} : \frac{1}{(1 + \delta)^2} \leq \ratios[\epsilon t] \leq (1 + \delta)^2 \ \forall t \in [\DerivativeBound^{-1}, \DerivativeBound] \right\}
            \end{align*}
            has cardinality at least $\frac{1}{3} \# ((\epsilon^{1 - \zeta(n)}, \epsilon^{1 - \zeta_{+}(n)}) \cap \{8^{-j}\}_{j \in \NN})$.
            \begin{proof}
              Indeed, if this were false, then one of the sets
              \begin{align*}
                \cU_{\epsilon}^{(n)} \
                &\coloneqq \ \left\{r \in (\epsilon^{1 - \zeta(n)}, \epsilon^{1 - \zeta_{+}(n)}) \cap \{8^{-j}\}_{j \in \NN} : \inf_{t \in [\DerivativeBound^{-1}, \DerivativeBound]} \invratios[\epsilon t] < \frac{1}{(1 + \delta)^2} \right\}, \\
                \cL_{\epsilon}^{(n)} \
                &\coloneqq \ \left\{r \in (\epsilon^{1 - \zeta(n)}, \epsilon^{1 - \zeta_{+}(n)}) \cap \{8^{-j}\}_{j \in \NN} : \inf_{t \in [\DerivativeBound^{-1}, \DerivativeBound]} \ratios[\epsilon t] < \frac{1}{(1 + \delta)^2} \right\}
              \end{align*}
              has cardinality $> \# \frac{1}{3} ((\epsilon^{1 - \zeta(n)}, \epsilon^{1 - \zeta_{+}(n)}) \cap \{8^{-j}\}_{j \in \NN})$.
              But if $\# \cU_{\epsilon}^{(n)} > \frac{1}{3} \#((\epsilon^{1 - \zeta(n)}, \epsilon^{1 - \zeta_{+}(n)}) \cap \{8^{-j}\}_{j \in \NN})$, then since for each $r \in \cU_{\epsilon}^{(n)}$ and each $t, s \in [\DerivativeBound^{-1}, \DerivativeBound]$,
              \begin{align*}
                \frac{\frac{r^{\xi Q} \fa_{\epsilon t/r}^{-1}}{r \fa_{\epsilon t}^{-1}}}{\frac{r^{\xi Q} \fa_{\epsilon s/r}^{-1}}{r \fa_{\epsilon s}^{-1}}} \
                &= \ \frac{\fa_{\epsilon s/r}}{\fa_{\epsilon t/r}} \frac{\fa_{\epsilon t}}{\fa_{\epsilon s}} \\
                &= \ \frac{\fa_{(\epsilon t/r)(s/t)}}{\fa_{\epsilon t/r}} \frac{\fa_{\epsilon s (t/s)}}{\fa_{\epsilon s}} \\
                &\leq \ \sqrt{1 + \delta} \left(\frac{s}{t} \right)^{1 - \xi Q}  \sqrt{1 + \delta} \left(\frac{t}{s} \right)^{1 - \xi Q} \tag*{(by \eqref{eq:ActualGoodLipschitzConstantRegularVariation})} \\
                &= \ 1 + \delta,
              \end{align*}
              we see that
              \begin{align*}
                Y_{\epsilon}^{(n)} \
                &\leq \ \sup_{r \in \cU_{\epsilon}^{(n)}} \sup_{t \in [\DerivativeBound^{-1}, \DerivativeBound]} \invratios[\epsilon t] \
                \leq \ \left(1 + \delta \right) \sup_{r \in \cU_{\epsilon}^{(n)}} \inf_{t \in [\DerivativeBound^{-1}, \DerivativeBound]} \invratios[\epsilon t] \
                < \ \frac{1}{1 + \delta},
              \end{align*}
              which contradicts the fact that $n \in S$ and $\epsilon \in \cE_S$.
            \end{proof}
          \end{lemma}

          Denote the first $N$ elements of $S$ by $n_1 < n_2 < \cdots < n_N$, and let $n_0 = 0$.
          Let $\delta' \coloneqq \frac{\delta}{6 (A + 1)}$ and let $\cR_{\epsilon}^{(0)}$ be as in \eqref{eq:FiniteLimsupOfBestScalingRatios} and \eqref{eq:FiniteLimsupOfBestInverseScalingRatios}.
          Recursively define $C_j(\epsilon)$ and $C_j'(\epsilon)$ by $C_0(\epsilon) \coloneqq C_0 L_0$, $C_0'(\epsilon) \coloneqq C_0 L_0$, and
          \begin{align*}
            C_j
            &\coloneqq \frac{A}{A + 1} \left[\sup_{r \in \cR_{\epsilon}^{(n_j)}} \sup_{t \in [\DerivativeBound^{-1}, \DerivativeBound]} \ratios[\epsilon t] \vee C_{j-1}(\epsilon) \right] + \left[\frac{1}{A + 1} + 3 \delta' \right] \sup_{r \in \cR_{\epsilon}^{(n_j)}} \sup_{t \in [\DerivativeBound^{-1}, \DerivativeBound]} \ratios[\epsilon t], \\
            C_j'
            &\coloneqq \frac{A}{A + 1} \left[\sup_{r \in \cR_{\epsilon}^{(n_j)}} \sup_{t \in [\DerivativeBound^{-1}, \DerivativeBound]} \invratios[\epsilon t] \vee C_{j-1}'(\epsilon) \right] + \left[\frac{1}{A + 1} + 3 \delta' \right] \sup_{r \in \cR_{\epsilon}^{(n_j)}} \sup_{t \in [\DerivativeBound^{-1}, \DerivativeBound]} \invratios[\epsilon t].
          \end{align*}
          By Proposition \ref{prop:InitialLipschitzCondition} with $(\zeta(0), \zeta_{+}(0), \cR_{\epsilon}^{(0)}, \tilde{W})$ in place of $(\zeta, \zeta_{+}, \cR_{\epsilon}, W)$, then Proposition \ref{prop:ImprovingTheLipschitzConstant} with $(\zeta_{-}(n_j), \zeta(n_j), \zeta_{+}(n_j), \cR_{\epsilon}^{(n_j)}, C_{j-1}(\epsilon), C_{j-1}'(\epsilon),\delta')$ in place of $(\zeta_{-}, \zeta, \zeta_{+}, \cR_{\epsilon},$ $C_0(\epsilon), C_0'(\epsilon),\delta)$, we see that for some $\beta = \beta(S) > 0$, with probability $1 - O_{\epsilon}(\epsilon^{\beta})$ as $\epsilon \to 0$,
          \begin{align}
            &\locLFPP[\epsilon][\hphi]\left(\phi\left(B_{4 \epsilon^{1 - \zeta_{+}(n_N)}}(z) \right), \phi\left(B_{4 \epsilon^{1 - \zeta_{+}(n_N)}}(w) \right); \phi\left(B_{4 \epsilon^{1 - \zeta_{+}(n_N)}}(\tilde{W}) \right)\right) \nonumber \\
            &\qquad \qquad \leq \ C_N(\epsilon) D_h\left(z,w; \tilde{W} \right) \ \forall (z,w) \in \hyperlink{G}{G(W, \tilde{W})}, \label{eq:IteratedLipschitzConditionLFPPLeqLQG} \\
            &D_h\left(B_{4 \epsilon^{1 - \zeta_{+}(n_N)}}(z), B_{4 \epsilon^{1 - \zeta_{+}(n_N)}}(w); B_{4 \epsilon^{1 - \zeta_{+}(n_N)}}(\tilde{W}) \right) \nonumber \\
            &\qquad \qquad \leq \ C_N'(\epsilon) \locLFPP[\epsilon][\hphi]\left(z,w; \tilde{W} \right) \ \forall (z,w) \in \hyperlink{Gphi}{G_{\epsilon}^{\phi}(W, \tilde{W})}. \label{eq:IteratedLipschitzConditionLQGLeqLFPP}
          \end{align}

          We claim that $C_N(\epsilon) \vee C_N'(\epsilon) \leq (1 + \delta)^{3}$ for all $\epsilon \in \cE_S$.
          Indeed, if $C_j(\epsilon) \leq (1 + \delta)^3$ for some $1 \leq j \leq N - 1$, then 
          \begin{align*}
            C_{j+1}(\epsilon) \
            &\leq \ \frac{A}{A + 1} \left(1 + \delta \right)^3 + \left[\frac{1}{A + 1} + 3 \delta' \right] \left(1 + \delta \right)^2 \\
            &< \ \frac{A}{A + 1} \left(1 + \delta \right)^3 + \frac{1}{A + 1} \left(1 + \delta \right)^2 + 3 \frac{\delta}{3(A + 1)} (1 + \delta)^2 & \left(\delta' < \frac{\delta}{3(A + 1)} \right) \\
            &= \ \left(1 + \delta \right)^3,
          \end{align*}
          so $C_N(\epsilon) \leq (1 + \delta)^3$.
          If instead $C_j(\epsilon) > (1 + \delta)^3$ for all $1 \leq j \leq N - 1$, then by induction and the fact that $\ratios[\epsilon t] \leq (1 + \delta)^2$ for all $r \in \cR_{\epsilon}^{(n_j)}$ and all $t \in [\DerivativeBound^{-1}, \DerivativeBound]$,
          \begin{align*}
            C_N(\epsilon) \
            &\leq \ \left(\frac{A}{A + 1} \right)^{N} C_0 L_0 + \sum_{j=0}^{N-1} \left(\frac{A}{A + 1} \right)^j \left[\frac{1}{A + 1} + 3 \delta' \right] \left(1 + \delta \right)^2 \\
            &\leq \ \frac{\delta}{2} (1 + \delta)^2 + \left[1 + 3 (A + 1) \delta' \right] (1 + \delta)^2 \qquad \left(\left(\frac{A}{A + 1} \right)^N C_0 L \leq \frac{\delta/4}{1 + \delta} < \frac{\delta}{2} (1 + \delta)^2 \right)\\
            &= \ \frac{\delta}{2} \left(1 + \delta \right)^2 + \left(1 + \delta \right)^2 + 3 (A + 1) \frac{\delta}{6(A + 1)} (1 + \delta)^2 \qquad \left(\delta' = \frac{\delta}{6(A + 1)} \right) \\
            &= \ \left(1 + \delta \right)^3.
          \end{align*}
          Similarly, $C_{N}'(\epsilon) \leq (1 + \delta)^3$.

          Combining the bound $C_N(\epsilon) \vee C_N'(\epsilon) \leq (1 + \delta)^3$ with \eqref{eq:IteratedLipschitzConditionLFPPLeqLQG} and \eqref{eq:IteratedLipschitzConditionLQGLeqLFPP}, we have shown that \eqref{eq:1PlusDeltaLipschitzConstantLFPPLeqLQG} and \eqref{eq:1PlusDeltaLipschitzConstantLQGLeqLFPP} hold with probability $1 - O_{\epsilon}(\epsilon^{\beta(S)})$ as $\epsilon \to 0$ along $\cE_S$ with $(1 + \delta)^3$ in place of $1 + \delta$ and $\zeta_{+}(n_N)$ in place of $\zeta$.
          Repeat this with all finitely many $S \subset \{1, 2, \ldots, 3N\}$ with $\# S \geq N$ and $\inf \cE_S = 0$, and let $\beta$ be the smallest of the $\beta(S)$ and $\zeta$ be the largest of the $\zeta_{+}(n_N)$ to prove the proposition with $(1 + \delta)^3$ in place of $1 + \delta$.
        \end{proof}
      \end{prop}

      Our next lemma will allow us to control the error terms from estimating point-to-point distances by ball-to-ball distances.
      It will rely on the following, which is \cite[equation (40)]{AlmostSureConvergence}.

      \begin{lemma}
        \label{lemma:AlmostSureFieldMollificationBound}
        Let $h$ be a whole-plane GFF, and fix $R > 0$ and $\eta \in (0, 1)$.
        There is a random variable $C > 0$ such that almost surely, for all $\epsilon$ sufficiently small,
        \begin{align*}
          \sup_{z \in B_R(0)} \left|\hat{h}_{\epsilon}^{*}(z) \right| \
          &\leq \ (2 + \eta) \log \epsilon^{-1} + C.
        \end{align*}
      \end{lemma}

      \begin{lemma}
        \label{lemma:ErrorTerms}
        Fix $\zeta \in (0, 1)$ and open sets $W \Subset \tilde{W} \Subset V \Subset U$.
        Almost surely,
        \begin{align}
          \lim_{n \to \infty} \sup_{\substack{z, w \in W \\ |z - w| \leq 4 \cdot 2^{-n(1 - \zeta)}}} \sup_{\phi \in \confmaps} \locLFPP[2^{-n}][\hphi]\left(\phi(z), \phi(w); \phi\left(\tilde{W}\right)\right) \
          &= \ 0.
          \label{eq:ErrorTermsConvergeTo0}
        \end{align}
        \begin{proof}
          The idea of the proof is that when $\zeta \in (0, 1/3)$, Lemma \ref{lemma:QuantitativeCoordinateChangedFieldComparison} and Borel-Cantelli imply that \eqref{eq:ErrorTermsConvergeTo0} will follow by showing that
          \begin{align}
            \sup_{\substack{z,w \in W \\ |z - w| \leq 4 \cdot 2^{-n(1 - \zeta)}}} \sup_{t \in [\DerivativeBound^{-1}, \DerivativeBound]} \locLFPP[2^{-n} t][h]\left(z, w; \tilde{W} \right) \
            &\to \ 0.
            \label{eq:ErrorTermsReduction}
          \end{align}
          By estimating the $\locLFPP[\epsilon t][h]$-length of a straight line segment of Euclidean length $\epsilon^{1 - \zeta}$, we can show that \eqref{eq:ErrorTermsReduction} holds for all $\zeta \in (0, 1)$ sufficiently close to $0$.
          To extend this to the case of arbitrary $\zeta \in (0, 1)$, we fix $\zeta' \in (0, 1)$ close enough to $0$ for the previous case to apply, then estimate $\locLFPP[\epsilon][\hphi](\phi(z), \phi(w); \phi(\tilde{W}))$ by the distances from $\phi(z)$ (resp. $\phi(w)$) to a point of $\phi(\partial B_{4 \epsilon^{1 - \zeta'}}(z))$ (resp. $\phi(\partial B_{4 \epsilon^{1 - \zeta'}}(w))$) plus the distance between $\phi(B_{4 \epsilon^{1 - \zeta'}}(z))$ and $\phi(B_{4 \epsilon^{1 - \zeta'}}(w))$.
          The former distances converge to $0$ almost surely along the sequence $\epsilon = 2^{-n}$ because $\zeta'$ is close enough to $0$.
          To handle the distance between $\phi(B_{4 \epsilon^{1 - \zeta'}}(z))$ and $\phi(B_{4 \epsilon^{1 - \zeta'}}(w))$, we use Proposition \ref{prop:InitialLipschitzCondition} to compare to the $D_h$-distance between $z$ and $w$, which converges to $0$ by continuity.

          First assume $\zeta \in (0, 1/3)$.
          Fix an open set $W'$ with $W \Subset W' \Subset \tilde{W}$.
          Throughout the proof, we always assume $\epsilon$ is small enough that $B_{4 \epsilon^{1 - \zeta}}(W') \subset \tilde{W}$.
          By Lemma \ref{lemma:QuantitativeCoordinateChangedFieldComparison} and a union bound over $(\frac{\epsilon^{1 - \zeta}}{4} \ZZ^2) \cap W'$, there are constants $c, \tilde{c}, \epsilon_0 > 0$ such that for each $\epsilon \in (0, \epsilon_0)$,
          \begin{align}
            &\Prob\left\{\sup_{z_0 \in (\frac{\epsilon^{1 - \zeta}}{4} \ZZ^2) \cap W'} \sup_{z \in B_{4 \epsilon^{1 - \zeta}}(z_0)} \sup_{\phi \in \confmaps} \left|\langle h, \Psi_{\epsilon}^{\phi, z} - \Psi_{\epsilon/|\phi'(z_0)|}^{\id, z} \rangle \right| > 1 \right\} \nonumber \\
            &\qquad \qquad \leq \ \tilde{c} \epsilon^{-2 (1 - \zeta)} \exp\left(-c \frac{\epsilon^{-2 + 6 \zeta}}{(\log \epsilon^{-1})^3} \right). \label{eq:ErrorTermUnionBound}
          \end{align}
          Taking $\epsilon = 2^{-n}$, \eqref{eq:ErrorTermUnionBound} is summable, so by Borel-Cantelli, almost surely, for all $n$ sufficiently large,
          \begin{align}
            \sup_{z_0 \in (\frac{2^{-n(1 - \zeta)}}{4} \ZZ^2) \cap W'} \sup_{z \in B_{4 \cdot 2^{-n(1 - \zeta)}}(z_0)} \sup_{\phi \in \confmaps} \left|\langle h, \Psi_{2^{-n}}^{\phi, z} - \Psi_{2^{-n}/|\phi'(z_0)|}^{\id, z} \rangle \right| \
            &\leq \ 1. \label{eq:ErrorTermsBorelCantelli}
          \end{align}
          By Lemma \ref{lemma:LogMollificationConvergence} and Lemma \ref{lemma:RegularlyVaryingScalingConstants}, we can decrease $\epsilon_0$ so that for all $\epsilon \in (0, \epsilon_0)$,
          \begin{align}
            \sup_{z \in \tilde{W}} \sup_{\phi \in \confmaps} \left|\langle - \log|\phi'|, \Psi_{\epsilon}^{\phi, z} \rangle - ( -\log|\phi'(z)|) \right| \
            &< \ 1, \label{eq:ErrorTermsLogDerivatives} \\
            \sup_{t \in [\DerivativeBound^{-1}, \DerivativeBound]} \left|\frac{\fa_{t \epsilon}}{\fa_{\epsilon}} - t^{1 - \xi Q} \right| \
            &\leq \ \DerivativeBound^{1 - \xi Q}. \label{eq:ErrorTermsScalingConstants}
          \end{align}
          Note that if \eqref{eq:ErrorTermsScalingConstants} holds, then for $s, t \in [\DerivativeBound^{-1}, \DerivativeBound]$,
          \begin{align}
            \frac{\fa_{\epsilon/t}}{\fa_{\epsilon}} s^{1 - \xi Q} \
            &= \ s^{1 - \xi Q} \left(\frac{\fa_{\epsilon/t}}{\fa_{\epsilon}} - t^{-(1 - \xi Q)} \right) + (s/t)^{1 - \xi Q} \
            \leq \ 2 \DerivativeBound^{2(1 - \xi Q)}. 
            \label{eq:ErrorTermsScalingConstantsAgain}
          \end{align}
          Now assume $\epsilon = 2^{-n}$ with $n$ large enough that $\epsilon \in (0, \epsilon_0)$.
          Fix $z, w \in W$ with $|z - w| \leq \epsilon^{1 - \zeta}$.
          Fix $z_0 \in (\frac{\epsilon^{1 - \zeta}}{4} \ZZ^2) \cap W'$ such that $|z - z_0| < \epsilon^{1 - \zeta}$.
          Then $z, w \in B_{4 \epsilon^{1 - \zeta}}(z_0) \subset \tilde{W}$, so
          \begin{align*}
            &\locLFPP[\epsilon][\hphi]\left(\phi(z), \phi(w); \phi(\tilde{W}) \right) \\
            &\leq \ \locLFPP[\epsilon][\hphi]\left(\phi(z), \phi(w); \phi(B_{4 \epsilon^{1 - \zeta}}(z_0)) \right) \\
            &\leq \ e^{\xi} e^{\xi Q} \fa_{\epsilon/|\phi'(z_0)|}^{-1} \tag*{(by \eqref{eq:ErrorTermsBorelCantelli}, \eqref{eq:ErrorTermsLogDerivatives})} \\
            &\qquad \qquad \cdot \inf_{\substack{P \colon z \to w \\ P \subset B_{4 \epsilon^{1 - \zeta}}(z_0)}} \int\limits_{0}^{1} \frac{\fa_{\epsilon/|\phi'(z_0)|}}{\fa_{\epsilon}} |\phi'(P(t))|^{1 - \xi Q} e^{\xi \hat{h}_{\epsilon/|\phi'(z_0)|}^{*}(P(t))} |P'(t)| \, \D t  \\
            &\leq \ e^{\xi} e^{\xi Q} 2 \DerivativeBound^{2(1 - \xi Q)} \fa_{\epsilon/|\phi'(z_0)|}^{-1} \inf_{\substack{P \colon z \to w \\ P \subset B_{4 \epsilon^{1 - \zeta}}(z_0)}} \int\limits_{0}^{1}  e^{\xi \hat{h}_{\epsilon/|\phi'(z_0)|}^{*}(P(t))} |P'(t)| \, \D t \tag*{(by \eqref{eq:ErrorTermsScalingConstantsAgain})} \\
            &= \ e^{\xi} e^{\xi Q} 2 \DerivativeBound^{2(1 - \xi Q)} \locLFPP[\epsilon/|\phi'(z_0)|][h]\left(z, w; B_{4 \epsilon^{1 - \zeta}}(z_0) \right).
          \end{align*}
          Bounding the metric on the right-hand side by the length of a straight line segment from $z$ to $w$, Lemma \ref{lemma:AlmostSureFieldMollificationBound} and the fact that $\fa_{\epsilon} \geq c' \epsilon^{1 - \xi Q} (\log \epsilon^{-1})^{-b}$ for some $b, c' > 0$ \cite[Theorem 1.11]{UpToConstants} imply
          \begin{align*}
            \locLFPP[\epsilon/|\phi'(z_0)|][h]\left(z, w; B_{2 \epsilon^{1 - \zeta}}(z_0) \right) \
            &\leq \ C |z - w| \epsilon^{-\xi (2 + \eta)} \epsilon^{\xi Q - 1} \left(\log \epsilon^{-1} \right)^b
          \end{align*}
          for any $\eta \in (0, 1)$, where $C > 0$ is random, but does not depend on $z_0$ or $\phi$.
          Since $|z - w| \leq 4 \epsilon^{1 - \zeta}$ and $Q > 2$, if $\eta$ and $\zeta$ are close enough to $0$, the right-hand side will be at most $2 C$ times a positive power of $\epsilon$ times $(\log \epsilon^{-1})^b$, so converges to $0$ as $\epsilon \to 0$.
          This proves the lemma in the case that $\zeta$ is sufficiently close to $0$.

          To deduce the result for arbitrary $\zeta \in (0, 1)$, let $\zeta' \in (0, \zeta)$ be close enough to $0$ for the previous argument to apply, and use Proposition \ref{prop:InitialLipschitzCondition} with $\zeta'$ in place of $\zeta_{+}$ to see that with probability $1 - O_{\epsilon}(\epsilon^{\beta})$ for some $\beta = \beta(\zeta') > 0$, for all $z,w \in W$ with $|z - w| \leq \epsilon^{1 - \zeta}$, 
          \begin{align*}
            &\locLFPP[\epsilon][\hphi]\left(\phi(z), \phi(w); \phi(\tilde{W}) \right) \\
            &\qquad \qquad \leq \ \sup_{u \in B_{4 \epsilon^{1 - \zeta'}}(z)} \locLFPP[\epsilon][\hphi]\left(\phi(z), \phi(u); \phi(\tilde{W}) \right) \\
            &\qquad \qquad \qquad \qquad + \sup_{v \in B_{4 \epsilon^{1 - \zeta'}}(w)} \locLFPP[\epsilon][\hphi]\left(\phi(v), \phi(w); \phi(\tilde{W}) \right) \\
            &\qquad \qquad \qquad \qquad + \locLFPP[\epsilon][\hphi]\left(\phi\left(B_{4 \epsilon^{1 - \zeta'}}(z) \right), \phi\left(B_{4 \epsilon^{1 - \zeta'}}(w) \right); \phi\left(B_{4 \epsilon^{1 - \zeta'}}(W') \right)\right) \\
            &\qquad \qquad \leq \ \sup_{u \in B_{4 \epsilon^{1 - \zeta'}}(z)} \locLFPP[\epsilon][\hphi]\left(\phi(z), \phi(u); \phi(\tilde{W}) \right) \\
            &\qquad \qquad \qquad \qquad + \sup_{v \in B_{4 \epsilon^{1 - \zeta'}}(w)} \locLFPP[\epsilon][\hphi]\left(\phi(v), \phi(w); \phi(\tilde{W}) \right) \\
            &\qquad \qquad \qquad \qquad + C_0 D_h\left(z, w; W' \right),
          \end{align*}
          where $C_0$ is as in Proposition \ref{prop:InitialLipschitzCondition}.
          By Borel-Cantelli, this inequality holds for $\epsilon = 2^{-n}$ for all $n$ sufficiently large.
          The first two terms on the right-hand side converge to $0$ by the previous case (applied with an open set $W'$ such that $W \Subset W' \Subset \tilde{W}$ in place of $W$ to ensure $B_{4 \epsilon^{1 - \zeta}}(z) \subset W'$ for all $z \in W$ and all $\epsilon$ sufficiently small), while the third term converges to $0$ almost surely by continuity of $D_h$.
        \end{proof}
      \end{lemma}

      Recall that for $\rho > 0$, we defined
      \begin{align*}
        \diagonal \
        &\coloneqq \ \left\{(z,w) \in W^2 : |z - w| \leq \rho \right\}.
      \end{align*}
      Our next lemma shows that $\cap_{\phi \in \confmaps} \hyperlink{Gphi}{G_{\epsilon}^{\phi}(W, \tilde{W})}$ contains $\diagonal$ for a random $\rho > 0$.

      \begin{lemma}
        \label{lemma:GeodesicSets}
        For each connected open $U \subset \CC$, there is a measurable set $F_U \subset \cD'(U)$ satisfying the following.
        \begin{enumerate}[label=(\roman*)]
          \item If $h \in F_U$, then for all open $W \Subset \tilde{W} \Subset V \Subset U$ and $\DerivativeBound > 1$, there exists $\rho > 0$ and $N \in \NN$ such that
            \begin{align*}
              \diagonal \
              &\subset \ \bigcap_{\phi \in \confmaps} \hyperlink{Gphi}{G_{2^{-n}}^{\phi}\left(W, \tilde{W} \right)} \ \forall n \geq N.
            \end{align*}
          \item If $h$ is a whole-plane GFF, then $h|_U \in F_U$ almost surely.
          \item If $h \in \cD'(U)$ and $\phi \colon U \to \phi(U)$ is a conformal map, then $h \in F_U$ if and only if $\hphi \in F_{\phi(U)}$.
        \end{enumerate}
        \begin{proof}
          Given open sets $W \Subset \tilde{W} \Subset V \Subset U$, $\DerivativeBound > 1$, $\rho > 0$, and $M \in \NN$, define $F_{W, \tilde{W}, V, U, \DerivativeBound, \rho, M}$ to be the set of $h \in \cD'(U)$ such that for all $\phi \in \confmaps$,
          \begin{align}
            \locLFPP[2^{-m}][\hphi]\left(\phi(z), \phi(w); \phi(V) \right) \
            &< \ \frac{1}{2} \locLFPP[2^{-m}][\hphi]\left(\phi(W), \phi(\partial \tilde{W}) \right) \ \forall m \geq M \ \forall (z,w) \in \diagonal.
            \label{eq:CoordinateChangeDistanceToBoundary}
          \end{align}
          Note that if \eqref{eq:CoordinateChangeDistanceToBoundary} holds, then $\diagonal \subset \hyperlink{Gphi}{G_{2^{-m}}^{\phi}(W, \tilde{W})}$ for all $m \geq M$.
          Moreover, $\confmaps$ can be replaced by $\lambda_{\DerivativeBound}(V, U)$ from Lemma \ref{lemma:CountableDenseSetOfConformalMaps}, so $F_{W, \tilde{W}, V, U, \DerivativeBound, \rho, M}$ is a measurable subset of $\cD'(U)$.

          Now choose open sets $W_n \Subset \tilde{W}_n \Subset V_n \Subset U$ increasing to $U$.
          Define $F_U$ to be the event that for all $n \in \NN$ there exists $j \geq n$ such that for all $\DerivativeBound \geq 2$ there exists $M, \ell \in \NN$ such that $F_{W_n, \tilde{W}_j, V_j, U, \DerivativeBound, 1/\ell, M}$ occurs.

          From \eqref{eq:CoordinateChangeDistanceToBoundary}, it is clear that (i) holds.
          To see (iii), use the fact that if $\phi \colon U \to \phi(U)$ and $\psi \colon \phi(U) \to \psi(\phi(U))$ are conformal maps and $W \Subset \tilde{W} \Subset V \Subset \phi(U)$, then $(\hphi)^{\psi} = \hphi[\psi \circ \phi]$ (recall $\hphi \coloneqq h \circ \phi^{-1} + Q \log|(\phi^{-1})'|$) and
          \begin{align*}
            \locLFPP[\epsilon][(\hphi)^{\psi}]\left(\psi(z), \psi(w); \psi(V) \right)
            &= \locLFPP[\epsilon][\hphi[\psi \circ \phi]]\left((\psi \circ \phi)(\phi^{-1}(z)), (\psi \circ \phi)(\phi^{-1}(w)); (\psi \circ \phi)(\phi^{-1}(V)) \right) \\
            \locLFPP[\epsilon][(\hphi)^{\psi}]\left(\psi(W), \psi(\partial \tilde{W}) \right)
            &= \locLFPP[\epsilon][\hphi[\psi \circ \phi]]\left((\psi \circ \phi)(\phi^{-1}(W)), (\psi \circ \phi)(\partial \phi^{-1}(\tilde{W}))\right)
          \end{align*}
          together with the fact that $h \in F_U$ implies $h \in F_{\phi^{-1}(W), \phi^{-1}(\tilde{W}), \phi^{-1}(V), U, \DerivativeBound, \rho, M}$ for $M$ and $\tilde{W} \Subset V \Subset \phi(U)$ sufficiently large and $\rho$ sufficiently small.

          Let us now verify (ii).
          Let $h$ be a whole-plane GFF.
          The idea of the proof is that by continuity, when $\rho$ is sufficiently small, $D_h(z,w) < \frac{1}{2} D_h(W, \partial \tilde{W})$ for all $z,w \in W$ with $|z - w| \leq \rho$.
          Using Proposition \ref{prop:InitialLipschitzCondition}, we can transfer this inequality to $\locLFPP[\epsilon][\hphi]$ for sufficiently small $\epsilon$.

          Fix $\zeta \in (0, 1)$.
          By Lemma \ref{lemma:ErrorTerms} (applied with $W$ replaced by a neighbourhood of $\overline{W}$ contained in $\tilde{W}$, almost surely,
          \begin{align*}
            \lim_{n \to \infty} \sup_{z \in W} \sup_{w \in \partial B_{4 n^{-a(1 - \zeta)}}(z)} \sup_{\phi \in \confmaps} \locLFPP[2^{-n}][\hphi]\left(\phi(z), \phi(w); \phi\left(\tilde{W} \right)\right) \
            &= \ 0.
          \end{align*}
          Fix an open set $W'$ such that $\tilde{W} \Subset W' \Subset V$.
          Use Proposition \ref{prop:InitialLipschitzCondition} to find $C_0, \beta > 0$ such that with probability at least $1 - O_{\epsilon}(\epsilon^{\beta})$ as $\epsilon \to 0$, 
          \begin{align}
            &\sup_{\phi \in \confmaps} \locLFPP[\epsilon][\hphi]\left(\phi\left(B_{4 \epsilon^{1 - \zeta}}(z) \right), \phi\left(B_{4 \epsilon^{1 - \zeta}}(w) \right); \phi\left(B_{4 \epsilon^{1 - \zeta}}(\tilde{W}) \right)\right) \nonumber \\
            &\qquad \qquad \leq \ C_0 D_h\left(z, w; \tilde{W} \right) \ \forall z,w \in \tilde{W}, \label{eq:GeodesicSetsInitialLipschitz1} \\
            &D_h\left(B_{4 \epsilon^{1 - \zeta}}(z), B_{4 \epsilon^{1 - \zeta}}(w); B_{4 \epsilon^{1 - \zeta}}(W') \right) \nonumber \\
            &\qquad \qquad \leq \ C_0 \inf_{\phi \in \confmaps} \locLFPP[\epsilon][\hphi]\left(\phi(z), \phi(w); \phi\left(W'\right)\right) \ \forall z,w \in W'. \label{eq:GeodesicSetsInitialLipschitz2}
          \end{align}
          Borel-Cantelli implies that almost surely, \eqref{eq:GeodesicSetsInitialLipschitz1} and \eqref{eq:GeodesicSetsInitialLipschitz2} hold for $\epsilon = 2^{-n}$ for all $n$ sufficiently large.
          Fix $\eta > 0$ such that $\inf_{z \in W, w \in \partial \tilde{W}} |z - w| > 2 \eta$ and $\inf_{z \in \tilde{W}, w \in \partial W'} |z - w| > 2 \eta$.
          By continuity of $D_h$, there exists $\rho > 0$ such that 
          \begin{align}
            \sup_{(z,w) \in \diagonal} D_h\left(z, w \right) \
            &< \ \frac{1}{4} C_0^{-2} D_h\left(B_{\eta}(W), B_{\eta}(\partial \tilde{W}) \right). \label{eq:GeodesicSetsLQGMetric}
          \end{align}
          Choose $N$ large enough that $4 \cdot 2^{-n(1 - \zeta)} < \eta$, \eqref{eq:GeodesicSetsInitialLipschitz1} and \eqref{eq:GeodesicSetsInitialLipschitz2} both hold for $\epsilon = 2^{-n}$ with $n \geq N$, and for each $n \geq N$,
          \begin{align}
            \begin{split}
              &\sup_{z \in W} \sup_{w \in \partial B_{4 \cdot 2^{-n(1 - \zeta)}}(z)} \sup_{\phi \in \confmaps} \locLFPP[2^{-n}][\hphi]\left(\phi(z), \phi(w); \phi\left(\tilde{W} \right)\right) \\
              &\qquad \qquad < \ \frac{1}{8} C_0^{-1} D_h\left(B_{\eta}(W), B_{\eta}(\partial \tilde{W}) \right). 
            \end{split}
            \label{eq:GeodesicSetsErrorTerms}
          \end{align}
          Then for each $n \geq N$ and each $\phi \in \confmaps$, with $\epsilon = 2^{-n}$,
          \begin{align*}
            &\sup_{(z,w) \in \diagonal} \locLFPP[\epsilon][\hphi]\left(\phi(z), \phi(w); \phi(V) \right) \\
            &\leq \ \sup_{(z,w) \in \diagonal} \locLFPP[\epsilon][\hphi]\left(\phi(z), \phi(w); \phi\left(B_{4 \epsilon^{1 - \zeta}}(\tilde{W}) \right)\right) \\
            &\leq \ \sup_{(z,w) \in \diagonal} \locLFPP[\epsilon][\hphi]\left(\phi(B_{4 \epsilon^{1 - \zeta}}(z)), \phi(B_{4 \epsilon^{1 - \zeta}}(w)); \phi(B_{4 \epsilon^{1 - \zeta}}(\tilde{W})) \right) \\
            &\qquad \qquad + 2 \sup_{u \in W} \sup_{v \in \partial B_{4 \epsilon^{1 - \zeta}}(u)} \locLFPP[\epsilon][\hphi]\left(\phi(u), \phi(v); \phi\left(B_{4 \epsilon^{1 - \zeta}}(\tilde{W}) \right)\right) \\
            &\leq \ C_0 \sup_{(z,w) \in \diagonal} D_h\left(z, w; \tilde{W} \right) \tag*{(by \eqref{eq:GeodesicSetsInitialLipschitz1})} \\
            &\qquad \qquad + 2 \sup_{u \in W} \sup_{v \in \partial B_{4 \epsilon^{1 - \zeta}}(u)} \locLFPP[\epsilon][\hphi]\left(\phi(u), \phi(v); \phi\left(\tilde{W} \right)\right) \\
            &< \ \frac{1}{2} C_0^{-1} D_h\left(B_{4 \epsilon^{1 - \zeta}}(W), B_{4 \epsilon^{1 - \zeta}}(\partial \tilde{W}) \right) \tag*{(by \eqref{eq:GeodesicSetsLQGMetric}, \eqref{eq:GeodesicSetsErrorTerms})}  \\
            &\leq \ \frac{1}{2} \locLFPP[\epsilon][\hphi]\left(W, \partial \tilde{W} \right), \tag*{(by \eqref{eq:GeodesicSetsInitialLipschitz2})}
          \end{align*}
        \end{proof}
      \end{lemma}

      \begin{proof}[Proof of Theorem \ref{thm:ConvergenceTheorem}]
        Fix an open set $\tilde{W}$ such that $W \Subset \tilde{W} \Subset V$.
        By Proposition \ref{prop:LipschitzConstant1PlusDelta}, for each $j \in \NN$, we can find $\beta_j > 0$ and $\zeta_j \in (0, 1)$ such that with probability $1 - O_{\epsilon}(\epsilon^{\beta_j})$ as $\epsilon \to 0$, for all $\phi \in \confmaps$,
        \begin{align}
          &\locLFPP[\epsilon][\hphi]\left(\phi\left(B_{4 \epsilon^{1 - \zeta_j}}(z) \right), \phi\left(B_{4 \epsilon^{1 - \zeta_j}}(w) \right); \phi\left(B_{4 \epsilon^{1 - \zeta_j}}(\tilde{W}) \right)\right) \nonumber \\
          &\qquad \qquad \leq \ \frac{j+1}{j} D_{h}\left(z, w; \tilde{W} \right) \ \forall (z,w) \in \hyperlink{G}{G(W, \tilde{W})}, \label{eq:iConvergenceLipschitz1} \\
          &D_h\left(B_{4 \epsilon^{1 - \zeta_j}}(z), B_{4 \epsilon^{1 - \zeta_j}}(w); B_{4 \epsilon^{1 - \zeta_j}}(\tilde{W}) \right) \nonumber \\
          &\qquad \qquad \leq \ \frac{j + 1}{j} \locLFPP[\epsilon][\hphi]\left(\phi(z), \phi(w); \phi\left(\tilde{W} \right)\right) \ \forall (z,w) \in \hyperlink{Gphi}{G_{\epsilon}^{\phi}(W, \tilde{W})}. \label{eq:iConvergenceLipschitz2}
        \end{align}
        Fix an open set $W'$ with $W \Subset W' \Subset \tilde{W}$.
        By Lemma \ref{lemma:ErrorTerms} and Lemma \ref{lemma:GeodesicSets}, we can find a random $\rho > 0$ and deterministic $N_j \in \NN$ such that for each $j$, almost surely, 
        \begin{align}
          &\sup_{(z,w) \in \diagonal[4 \epsilon^{1 - \zeta_{j}}][W']} D_h\left(z,w; \tilde{W} \right) \
          < \ \frac{1}{j}, \quad \epsilon = 2^{-n}, \ \forall n \geq N_j, \label{eq:iConvergenceLQGError} \\
          &\sup_{(z,w) \in \diagonal[4 \epsilon^{1 - \zeta_j}][W']} \sup_{\phi \in \confmaps} \locLFPP[\epsilon][\hphi]\left(\phi(z), \phi(w); \phi(\tilde{W}) \right) \
          < \ \frac{1}{j} \quad \epsilon = 2^{-n}, \ \forall n \geq N_j, \label{eq:iConvergenceLFPPError} \\
          &\diagonal \
          \subset \ \hyperlink{G}{G(W, \tilde{W})} \cap \bigcap_{\phi \in \confmaps} \hyperlink{Gphi}{G_{\epsilon}^{\phi}\left(W, \tilde{W} \right)} \quad \epsilon = 2^{-n}, \ \forall n \geq N_1. \label{eq:iConvergenceGeodesicSets}
        \end{align}
        By Borel-Cantelli, we can increase $N_j$ so that \eqref{eq:iConvergenceLipschitz1} and \eqref{eq:iConvergenceLipschitz2} both hold for $\epsilon = 2^{-n}$ for all $n \geq N_j$.
        We may assume $(N_j)_{j=1}^{\infty}$ is strictly increasing, and that $B_{4 \cdot 2^{-N_1}}(W) \Subset W'$.

        Let $\epsilon = 2^{-n}$ and assume $n \geq N_j$.
        For $(z,w) \in \diagonal$ and $\phi \in \confmaps$,
        \begin{align*}
          &\locLFPP[\epsilon][\hphi]\left(\phi(z), \phi(w); \phi(V) \right) \\
          &= \ \locLFPP[\epsilon][\hphi]\left(\phi(z), \phi(w); \phi\left(B_{4 \epsilon^{1 - \zeta_j}}(\tilde{W}) \right)\right) \tag*{(by \eqref{eq:iConvergenceGeodesicSets})} \\
          &\leq \ \locLFPP[\epsilon][\hphi]\left(\phi(B_{4 \epsilon^{1 - \zeta_j}}(z)), \phi(B_{4 \epsilon^{1 - \zeta_j}}(w)); \phi(B_{4 \epsilon^{1 - \zeta_j}}(\tilde{W})) \right) \\
          &\qquad \qquad + 2 \sup_{(u, v) \in \diagonal[4 \epsilon^{1 - \zeta_j}][W']} \locLFPP[\epsilon][\hphi]\left(\phi(u), \phi(w); \phi(\tilde{W}) \right) \\
          &\leq \ \frac{j+1}{j} D_h\left(z, w; \tilde{W} \right) + \frac{2}{j} \tag*{(by \eqref{eq:iConvergenceLipschitz1}, \eqref{eq:iConvergenceLFPPError})} \\
          &= \ \frac{j+1}{j} D_h\left(z, w; V \right) + \frac{2}{j}, \tag*{(by \eqref{eq:iConvergenceGeodesicSets})}.
        \end{align*}
        It follows that
        \begin{align}
          &\limsup_{2^{-n} = \epsilon \to 0} \sup_{(z,w) \in \diagonal} \sup_{\phi \in \confmaps} \left[\locLFPP[\epsilon][\hphi]\left(\phi(z), \phi(w); \phi(V) \right) - D_h\left(z, w; V \right) \right] \nonumber \\
          &\leq \ \limsup_{j \to \infty} \left[\frac{1}{j} \sup_{(z, w) \in \diagonal} D_h\left(z,w; V \right) + \frac{2}{j} \right] \nonumber \\
          &= \ 0,
          \label{eq:iConvergenceUpperBound}
        \end{align}
        where we have used the fact that $\sup_{(z,w) \in \diagonal} D_h(z, w; V) < \infty$ almost surely by continuity.
        Similarly, if $(z,w) \in \diagonal$ and $\phi \in \confmaps$,
        \begin{align*}
          &D_h\left(z, w; V \right) \\
          &= \ D_h\left(z, w; B_{4 \epsilon^{1 - \zeta_j}}(\tilde{W}) \right) \tag*{(by \eqref{eq:iConvergenceGeodesicSets})} \\
          &\leq \ D_h\left(B_{4 \epsilon^{1 - \zeta_j}}(z), B_{4 \epsilon^{1 - \zeta_j}}(w); B_{4 \epsilon^{1 - \zeta_j}}(\tilde{W}) \right) + 2 \sup_{(u,v) \in \diagonal[4 \epsilon^{1 - \zeta_j}][W']} D_h\left(u, v; \tilde{W} \right) \\
          &\leq \ \frac{j+1}{j} \locLFPP[\epsilon][\hphi]\left(\phi(z), \phi(w); \phi(\tilde{W}) \right) + \frac{2}{j} \tag*{(by \eqref{eq:iConvergenceLipschitz2}, \eqref{eq:iConvergenceLQGError})} \\
          &= \ \frac{j+1}{j} \locLFPP[\epsilon][\hphi]\left(\phi(z), \phi(w); \phi(V) \right) + \frac{2}{j}, \tag*{(by \eqref{eq:iConvergenceGeodesicSets})},
        \end{align*}
        hence
        \begin{align*}
          &\liminf_{2^{-n} = \epsilon \to 0} \inf_{(z, w) \in \diagonal} \inf_{\phi \in \confmaps} \left[\locLFPP[\epsilon][\hphi]\left(\phi(z), \phi(w); \phi(V) \right) - D_h\left(z,w; V \right)\right] \\
          &\geq \ \liminf_{j \to \infty} \inf_{2^{-n} = \epsilon < 2^{-N_j}} \inf_{(z,w) \in \diagonal} \inf_{\phi \in \confmaps} \left[-\frac{1}{j} \locLFPP[\epsilon][\hphi]\left(\phi(z), \phi(w); \phi(V)\right) - \frac{2}{j} \right] \\
          &\geq \ -\limsup_{j \to \infty} \sup_{2^{-n} = \epsilon < 2^{-N_j}} \left[\frac{1}{j} \sup_{(z,w) \in \diagonal} \sup_{\phi \in \confmaps} \locLFPP[\epsilon][\hphi]\left(\phi(z), \phi(w); \phi(V)\right) + \frac{2}{j} \right] \\
          &= \ 0,
        \end{align*}
        where the last step uses \eqref{eq:iConvergenceUpperBound} to see that
        \begin{align*}
          &\limsup_{j \to \infty} \sup_{2^{-n} = \epsilon < 2^{-N_j}} \sup_{(z,w) \in \diagonal} \sup_{\phi \in \confmaps} \locLFPP[\epsilon_j][\hphi]\left(\phi(z), \phi(w); \phi(V) \right) \\
          &\qquad \qquad \leq \ \sup_{(z,w) \in \diagonal} D_h\left(z,w; V \right) \
          < \ \infty.
        \end{align*}
      \end{proof}

  \section{Proof of Theorem \ref{thm:StrongCoordinateChange}}
    \label{section:CoordinateChange}
    We will conclude by constructing an LQG metric satisfying the strong coordinate change formula.
    We must define a collection of measurable functions $D^U \colon \cD'(U) \to \{\text{continuous metrics on } U\}$, one for each open set $U \subset \CC$, satisfying the LQG metric axioms plus axiom \ref{axiom:StrongCoordinateChange}.
    The idea is that we have shown that on a neighbourhood of the diagonal in $U$, we have the simultaneous convergence of metrics of the form $\locLFPP[\epsilon][\hphi](\phi(\cdot), \phi(\cdot); \phi(U))$ for all conformal maps $\phi$ to the same common limit when $h$ is a whole-plane GFF restricted to $U$.
    The limit can be extended to a function on $U \times U$ by taking the length metric induced by the limiting metric near the diagonal.
    When $h$ is a whole-plane GFF, the metric $D_{h|_U}^U$ almost surely equals $D_h(\cdot, \cdot; U)$ (where $D_h$ is the LQG metric on $\CC$ constructed in \cite{AlmostSureConvergence}), which implies $D^U$ satisfies the LQG metric axioms \ref{axiom:LengthSpace}-\ref{axiom:WeylScaling}.
    The strong coordinate change follows from the fact that $\locLFPP[\epsilon][\hphi](\cdot, \cdot; \phi(U))$ converges to the same limit as $\locLFPP[\epsilon][h](\phi^{-1}(\cdot), \phi^{-1}(\cdot); U)$.

    Let us be more precise.
    Fix a connected open set $U \subset \CC$.
    Let $F_U$ be the event from Lemma \ref{lemma:GeodesicSets}.
    Fix open sets $W \Subset V \Subset U$ and $\DerivativeBound > 1$.
    By (i) of Lemma \ref{lemma:GeodesicSets} (applied with $\tilde{W}$ an arbitrary open set $W \Subset \tilde{W} \Subset V$), if $h \in F_U$, then there exists $\rho > 0$ such that
    \begin{align}
      \diagonal \
      \subset \ \bigcap_{\phi \in \confmaps} \hyperlink{Gphi}{G_{2^{-n}}^{\phi}(W, V)} \ \forall n \gg 0, \label{eq:ADefnGeodesicSets}
    \end{align}
    Let $A_{W, V, U, \DerivativeBound}$ be the set of $h \in F_U$ such that there exists $\rho > 0$ such that
    \begin{align}
      \lim_{2^{-j} = \epsilon \to 0} \sup_{\phi, \psi \in \confmaps} \sup_{(z,w) \in \diagonal} \left|\locLFPP[\epsilon][\hphi]\left(\phi(z), \phi(w); \phi(V) \right) - \locLFPP[\epsilon][\hphi[\psi]]\left(\psi(z), \psi(w); \psi(V) \right)\right|
      &= 0, \label{eq:ADefnConformalMaps} \\
      \lim_{2^{-j} = \epsilon \to 0} \sup_{2^{-k} = \delta \leq \epsilon} \sup_{(z,w) \in \diagonal} \left|\locLFPP[\epsilon][h]\left(z, w; V \right) - \locLFPP[\delta][h]\left(z, w; V \right)\right|
      &= 0. \label{eq:ADefnConvergence}
    \end{align}
    By Lemma \ref{lemma:CountableDenseSetOfConformalMaps} and continuity of $\locLFPP[\epsilon][h](\cdot, \cdot)$, $A_{W, V, U, \DerivativeBound}$ is a measurable subset of $\cD'(U)$.
    Choose open sets $W_n \Subset V_n \Subset U$ increasing to $U$.
    Then define
    \begin{align*}
      A_U \
      &\coloneqq \ \bigcap_{n=1}^{\infty} \bigcap_{j=n}^{\infty} \bigcap_{\DerivativeBound = 2}^{\infty} A_{W_n, V_j, U, \DerivativeBound}.
    \end{align*}
    Theorem \ref{thm:ConvergenceTheorem} implies that if $h$ is a whole-plane GFF, then $h|_U \in A_U$ almost surely.
    Note that if $\phi \colon U \to \phi(U)$ is a conformal map, then $A_{\phi(U)} = \{\hphi : h \in A_U\}$.

    If $h \in A_{W, V, U, \DerivativeBound}$, then $\locLFPP[\epsilon][h](\cdot, \cdot; V)$ converges uniformly on $\diagonal$, where $\rho$ is small enough that \eqref{eq:ADefnGeodesicSets}, \eqref{eq:ADefnConformalMaps}, and \eqref{eq:ADefnConvergence} all hold.
    Moreover, \eqref{eq:ADefnGeodesicSets} implies that if $V \subset V'$, then $\locLFPP[\epsilon][h](\cdot, \cdot; V')$ converges to the same limit on $\diagonal$.
    Therefore, if $h \in A_U$, we can unambiguously define 
    \begin{align}
      \tilde{D}_h^U(z,w) \
      &\coloneqq \ \lim_{2^{-k} = \epsilon \to 0} \locLFPP[\epsilon][h]\left(z, w; V_{j} \right) \text{ if } h \in A_{W_n, V_{j}, U, \DerivativeBound}, (z,w) \in \diagonal[\rho][W_n]. \label{eq:TildeLQGMetricNearDiagonal}
    \end{align}
    This defines $\tilde{D}_h^U$ on a neighbourhood of the diagonal in $U \times U$.
    Recall that if $d$ is a metric on $U$, its induced length metric is the metric $d^I$ defined by 
    \begin{align*}
      d^I(x, y) \
      &\coloneqq \ \inf_{P \colon x \to y} \Len(P; d),
    \end{align*}
    where the infimum is over all continuous paths $P \colon [0, 1] \to U$ from $x$ to $y$, and where $\Len(P; d)$ is defined as in \eqref{eq:PathLengthDefinition}.
    By the triangle inequality, this definition makes sense even if $d$ is defined only on a neighbourhood of the diagonal.
    So we can extend $\tilde{D}_h^U$ to a function on $U \times U$ by defining it to equal the length metric induced by $\tilde{D}_h^U$.
    Note that the length metric induced by $\tilde{D}_h^U$ is indeed an extension of $\tilde{D}_h^U$ in the sense that it agrees with \eqref{eq:TildeLQGMetricNearDiagonal}.
    To prove this, one must show that when $(z,w) \in \diagonal[\rho][W_n]$ and $\epsilon > 0$, there is a path $P$ from $z$ to $w$ with $\tilde{D}_h^U$-length at most $\tilde{D}_h^U(z, w) + \epsilon$, which can be done by an approximate midpoints argument similar to the proof of \cite[Theorem 2.4.16]{Burago}.
    When $h \not\in A_U$, we define $\tilde{D}_h^U$ to equal the length metric induced by the Euclidean metric.

    We would like to say the collection of functions $(h \mapsto \tilde{D}_h^U)_{U \subset \CC}$ is an LQG metric, but technically we don't know that $\tilde{D}_h^U$ is a continuous metric for every $h \in \cD'(U)$ (indeed, it is a continuous metric when $h \not\in A_U$ or for almost every realization of a whole-plane GFF restricted to $U$, but we need $\tilde{D}_h^U$ to induce the Euclidean topology for \textit{every} $h \in \cD'(U)$).
    To deal with this, introduce the set
    \begin{align*}
      B_U \
      &\coloneqq \ \left\{h \in A_U : \tilde{D}_h^U \text{ induces the Euclidean topology} \right\}.
    \end{align*}
    If $\phi \colon U \to \phi(U)$ is a conformal transformation and $h \in A_U$, then $\tilde{D}_h^U = \tilde{D}_{\hphi}^{\phi(U)}(\phi(\cdot), \phi(\cdot))$ on a neighbourhood of the diagonal in $U$.
    From this, we see that $\tilde{D}_h^U$ induces the Euclidean topology if and only if $\tilde{D}_{\hphi}^{\phi(U)}$ does, i.e. $B_{\phi(U)} = \{\hphi : h \in B_U\}$.

    Finally, we define $D_h^U$ to equal $\tilde{D}_h^U$ when $h \in B_U$ and equal the length metric induced by the Euclidean metric on $U$ when $h \not\in B_U$.
    The remainder of this section will show that $(D^U)_{U \subset \CC}$ is an LQG metric satisfying the strong coordinate change formula.

    Theorem \ref{thm:ConvergenceTheorem} shows that if $h$ is a whole-plane GFF, then $h|_U \in B_U$ almost surely because $\tilde{D}_{h|_U}^U = D_h(\cdot, \cdot; U)$ and the metric on the right-hand side induces the Euclidean topology.
    As the LQG metric axioms are required to hold for a GFF plus continuous function on $U$, we will need the following lemma.

    \begin{lemma}
      \label{lemma:GeneralFields}
      If $h$ is a GFF plus continuous function $f$ on $U$, then $h \in B_U$ almost surely.
      Moreover, if $\fh^U$ is a random harmonic function on $U$ such that $h - f + \fh^U$ has the law of a whole-plane GFF on $U$, then $D_h^U = e^{\xi (f - \fh^U)} \cdot D_{h - f + \fh^U}^U$ whenever $h \in B_U$.
      \begin{proof}
        Such a harmonic function $\fh^U$ exists by the standard fact that if $X, Y, Z$ are random variables with $X \stackrel{d}{=} Y$, then there is a random variable $W$ such that $(X, Z) \stackrel{d}{=} (Y, W)$ (apply this with $X$ and $Z$ the zero boundary and harmonic parts of a whole-plane GFF restricted to $U$, and $Y = h - f$).
        Since $h - f + \fh^U$ has the law of a whole-plane GFF restricted to $U$, we know $h - f + \fh^U \in B_U$ almost surely for the reason given in the paragraph before this lemma.
        Now use a similar proof to \cite[Lemma 7]{WeylScaling} to show that $h - f + \fh^U \in B_U$ implies $h \in B_U$, and that $D_h^U = e^{\xi (f - \fh^U)} \cdot D_{h - f + \fh^U}^U$.
      \end{proof}
    \end{lemma}

    \begin{thm}
      $(D^U)_{U \subset \CC}$ satisfies the LQG metric axioms \ref{axiom:LengthSpace}, \ref{axiom:Locality}, and \ref{axiom:WeylScaling}, plus the strong coordinate change formula \ref{axiom:StrongCoordinateChange}.
      \begin{proof}
        Since $D_h^U$ is defined to be a length metric for \textit{all} $h \in \cD'(U)$, axiom \ref{axiom:LengthSpace} requires no justification.

        For the proofs of axiom \ref{axiom:Locality} and \ref{axiom:WeylScaling}, let $h$ be a GFF plus continuous function $f$ on $U$, and let $\fh^U$ be as in Lemma \ref{lemma:GeneralFields}.
        Then for any continuous function $g \colon U \to \RR$, by Lemma \ref{lemma:GeneralFields}
        \begin{align*}
          D_{h + g}^U \
          &= \ e^{\xi(f + g - \fh^U)} \cdot D_{h - f + \fh^U}^U \
          = \ e^{\xi g} \cdot e^{\xi (f - \fh^U)} \cdot D_{h - f + \fh^U}^U \
          = \ e^{\xi g} \cdot D_{h}^U.
        \end{align*}
        Therefore, axiom \ref{axiom:WeylScaling} holds.

        For axiom \ref{axiom:Locality}, if $V \subset U$ is an open set, then since
        \begin{align*}
          D_h^U(\cdot, \cdot; V) \
          &= \ e^{\xi (f|_V - \fh^U|_V)} \cdot \left(D_{h - f + \fh^U}^U \left(\cdot, \cdot; V \right) \right),
        \end{align*}
        where $h - f + \fh^U$ has the law of a whole-plane GFF restricted to $U$, it will suffice to check that if $H$ is a whole-plane GFF, then almost surely, $D_{H|_U}^U(\cdot, \cdot; V) = D_{H|_V}^V$.
        This is true because $D_{H|_U}^U = D_H(\cdot, \cdot; U)$ almost surely by Theorem \ref{thm:ConvergenceTheorem}, therefore
        \begin{align*}
          D_{H|_U}^U(\cdot, \cdot; V) \
          &= \ \text{internal metric of } D_H(\cdot, \cdot; U) \text{ on } V \
          = \ D_H(\cdot, \cdot; V) \
          = \ D_{H|_V}^V.
        \end{align*}

        Let us now explain why the strong coordinate change formula holds.
        The key fact is that for a conformal map $\phi \colon U \to \phi(U)$, we have $B_{\phi(U)} = \{\hphi : h \in B_U\}$.
        Suppose $h$ is a GFF plus continuous function on $U$, and assume the event $\{h \in B_U\}$ occurs.
        For a conformal map $\phi \colon U \to \phi(U)$, let $W_n^{\phi}, V_{n}^{\phi}$ be the open sets from the definition of $A_{\phi(U)}$, and write $W_n = W_n^{\id}$ and $V_{n} = V_{n}^{\id}$.
        Fix $n \in \NN$.
        We will show that on a neighbourhood of the diagonal in $W_n \times W_n$, the metrics $D_{\hphi}^{\phi(U)}(\phi(\cdot), \phi(\cdot))$ and $D_h^U$ are equal.
        Sending $n \to \infty$ so that $W_n \nearrow U$, since $D_h^U$ and $D_{\hphi}^{\phi(U)}(\phi(\cdot), \phi(\cdot))$ are both length metrics, they must be equal.

        Since $\phi(W_n) \Subset \phi(U)$, we can find $N \in \NN$ such that $\phi(W_n) \subset W_N^{\phi}$.
        Fix $J \geq N$.
        Choose $j \geq n$ such that $V_J^{\phi} \Subset \phi(V_j)$.
        Fix $\DerivativeBound \geq 2$ such that $\phi \in \confmaps[\DerivativeBound][V_j][U]$. 
        Note that $h \in A_{W_n, V_j, U, \DerivativeBound}$ and $\hphi \in A_{W_N^{\phi}, V_J^{\phi}, \phi(U), 2}$.
        Let $\rho > 0$ be as in the definition of $A_{W_n, V_{j}, U, \DerivativeBound}$ and $\rho^{\phi} > 0$ be as in the definition of $A_{W_N^{\phi}, V_{J}^{\phi}, \phi(U), 2}$.
        We may assume $\rho$ and $\rho^{\phi}$ are small enough that \eqref{eq:ADefnGeodesicSets} holds with $(W_n, V_j, \rho)$ and $(W_N^{\phi}, V_j^{\phi}, \rho^{\phi})$ in place of $(W, V, \rho)$.
        Let $0 < \tilde{\rho} < \rho$ be such that $\phi(\diagonal[\tilde{\rho}][W_n]) \coloneqq \{(\phi(z), \phi(w)) : (z, w) \in \diagonal[\tilde{\rho}][W_n] \} \subset \diagonal[\rho^{\phi}][W_N^{\phi}]$.
        Then on $\diagonal[\tilde{\rho}][W_n]$, 
        \begin{align*}
          &D_{\hphi}^{\phi(U)}\left(\phi(\cdot), \phi(\cdot) \right) \\
          &\qquad \qquad = \ \lim_{2^{-k} = \epsilon \to 0} \locLFPP[\epsilon][\hphi]\left(\phi(\cdot), \phi(\cdot); V_{J}^{\phi} \right) & (\phi(\diagonal[\tilde{\rho}][W_n]) \subset \diagonal[\rho^{\phi}][W_N^{\phi}], \hphi \in B_{\phi(U)}) \\
          &\qquad \qquad = \ \lim_{2^{-k} = \epsilon \to 0} \locLFPP[\epsilon][\hphi]\left(\phi(\cdot), \phi(\cdot); \phi(V_{j}) \right) & (A_{W_N^{\phi}, V_{J}^{\phi}, \phi(U), 2} \subset A_{W_N^{\phi}, \phi(V_{j}), \phi(U), 2}) \\
          &\qquad \qquad = \ \lim_{2^{-k} = \epsilon \to 0} \locLFPP[\epsilon][h]\left(\cdot, \cdot; V_{j} \right) & (h \in A_{W_n, V_{j}, U, \DerivativeBound}) \\
          &\qquad \qquad = \ D_h^U & (h \in B_U).
        \end{align*}
      \end{proof}
    \end{thm}

  \sloppy 
  \hfuzz=2pt 
  \printbibliography
\end{document}